\documentclass[11pt]{article}
\usepackage{color,epsfig}
\usepackage{amsthm}
\usepackage{amsfonts}
\usepackage{mathrsfs}
\usepackage[thicklines]{cancel}
\usepackage{amsfonts,graphicx,psfrag}
\usepackage{subfig}

\usepackage{ulem}

\def\be{\begin{equation}}
\def\ee{\end{equation}}
\def\bea{\begin{eqnarray}}
\def\eea{\end{eqnarray}}
\def\bes{\begin{eqnarray*}}
\def\ees{\end{eqnarray*}}

\def\nn{\nonumber}
\def\<{\langle}
\def\>{\rangle}
\def\lb{\label}
\def\bs{\setminus}

\def\R{{\bf R}}
\def\C{{\bf C}}

\def\N{{\bf N}}
\def\U{{\bf U}}

\def\ga{{\gamma}}

\def\th{{\theta}}
\def\om{{\omega}}

\def\ep{{\epsilon}}
\def\lm{{\lambda}}

\def\diag{{\rm diag}}

\def\Sp{{\rm Sp}}

\def\dm{{\rm \diamond}}

\newtheorem{lemma}{Lemma}[section]
\newtheorem{theorem}[lemma]{Theorem}
\newtheorem{corollary}[lemma]{Corollary}

\newtheorem{proposition}[lemma]{Proposition}
\newtheorem{definition}[lemma]{Definition}

\newtheorem{remark}[lemma]{Remark}
\newtheorem{question}{Question}[section]
\newtheorem{fact}{Fact}[section]
\newtheorem*{thmA}{Theorem A}
\newtheorem*{thmB}{Theorem B}
\newtheorem*{thmC}{Theorem C}
\newtheorem*{propD}{Proposition D}

\title{\bf The positive fundamental group of $\Sp(2n)$}

\author{Jian Wang$^{1}$\thanks{Partially supported by NSFC (Nos. {12071231, 12361141812}) and the Fundamental Research Funds for the Central University (No. 63243062). E-mail: wangjian@nankai.edu.cn} \quad and \quad
	Qinglong Zhou$^{2}$\thanks{Partially supported by NSFC (No.12171426), the Natural Science Foundation of Zhejiang Province (No. Y19A010072) and the Fundamental Research Funds for the Central Universities (No. 226-2024-00136).
		E-mail: zhouqinglong@zju.edu.cn. }\\
	$^{1}$ School of Mathematical Sciences and LPMC,\\ Nankai University, Tianjin 300071, China\\
	$^{2}$ School of Mathematical Science,\\Zhejiang University, Hangzhou 310058, Zhejiang, China\\
}
\date{}

\begin{document}

\maketitle

\begin{abstract}
	In this paper, 
	we examine the homotopy classes of positive loops in $\Sp(2n)$.
	We demonstrate that two positive loops are homotopic
	if and only if they are homotopic through positive loops. This provides a positive answer to a conjecture raised by McDuff in \cite{McD}. As a consequence, we extend several results of McDuff \cite{McD} and Chance \cite{Cha} to higher dimensional symplectic manifolds without dimensional restrictions.
\end{abstract}

\renewcommand{\theequation}{\thesection.\arabic{equation}}
\section{Introduction}
\label{sec:1}

\subsection{The main results}
A positive path in the group of real symplectic matrices $\Sp(2n)$
is a smooth path $\gamma(t)$ whose derivative $\gamma'$ satisfies
\begin{equation}\label{symp-ode}
\gamma'(t)=JP(t)\gamma(t),  
\end{equation}
where $P(t)$ is a positive definite symmetric matrix
and $J$ is the standard symplectic matrix.

In 1997, Lalonde and McDuff defined the positive fundamental group
$\pi_{1,pos}(\Sp(2n))$ to be the semigroup 
generated by positive loops with the base point at the identity,
where two loops are considered equivalent 
if one can deformed to the other via a family of positive loops.
In \cite{LaM}, the authors posed the natural question: 
$$\mathrm{\,Is\; the\; obvious\; map\;}
\pi_{1,pos}(\Sp(2n))\rightarrow \pi_1(\Sp(2n))
\mathrm{\; injective ?}$$
In 2008, McDuff \cite{McD} conjectured its truth. Our affirmative response to this conjecture is provided by the following theorem.

\begin{thmA}[Main result]\label{Thm:injective.n>2}
The natural map
$$
\pi_{1,pos}(\Sp(2n))\rightarrow \pi_1(\Sp(2n))
$$
is injective and onto $\mathbf{N}\backslash\{1,2,\cdots,n-1\}$ for any $n\in \mathbf{N}$.
\end{thmA}

The cases of $\Sp(2)$ and $\Sp(4)$ in Theorem A is due to Slimowitz \cite{Sli}. However, the proof of Theorem A for $\Sp(4)$ in \cite{Sli} has gaps (see Section \ref{sec:sp4} below), and its method is difficult to extend to higher dimensions. Our method effectively and rigorously addresses the case of $\Sp(4)$ and extends it to arbitrary even dimensions, thereby offering a comprehensive solution to the problem.



\bigskip

Our results can extend several results of McDuff \cite{McD} and Chance \cite{Cha} to higher dimensional symplectic manifolds. In their work, they explored loops of non-autonomous Hamiltonian diffeomorphisms with fixed maxima and further studied the Hofer geometry of the group of Hamiltonian diffeomorphisms. Finally, they were able to provide criteria for 4-manifolds to be uniruled.

Let $(M,\omega)$ be a closed, connected symplectic manifold. We suppose that $H_t: M\rightarrow \mathbf{R}$ is a smooth family of Hamiltonians parameterized by $t\in[0,1]$, and let $X_t^H$ denote the associated Hamiltonian vector field defined by $\omega(X_t^H,\cdot)=dH_t$. Also, let $\phi_t^H$ denote the corresponding flow. A point $x_0\in M$ is called a fixed global maximum if $H_t(x_0)\geq H_t(y)$ for all values of $t$ and for all $y\in M$. A similar definition of a fixed global minimum can be given.

Suppose $\ga=\{\phi_t^H\}$ is a loop in the group 
of Hamiltonian diffeomorphisms, $\mathrm{Ham}(M,\omega)$. We assume that $x_0$ is a fixed global maximum of $\ga$. We say that $\gamma$ is a circle action near $x_0$ if the generating Hamiltonian $H_t$ has the form $const - \sum m_i|z_i^2|$ in appropriate coordinates $(z_1,\cdots,z_k)$ of $x_0$, where $m_i$ are positive integers. 
We say this action is effective (i.e., no element other than the identity
acts trivially) if and only if $\mathrm{gcd}(m_1,\cdots,m_k)=1$.

A symplectic manifold is called {\it uniruled} if some point class nonzero
Gromov–Witten invariant does not vanish. More specifically this means that
there exist $a_2,\cdots,a_k\in H_*(M)$ so that
	$$\<pt, a_2,\cdots,a_k\>^M_{k,\beta} = 0 \,\mathrm{\, for\,\, some\, }\, 0 \neq \beta \in H_2^S(M),$$
where $pt$ is the point class in $H_0(M)$. We refer the reader to \cite{MS} for details on the Gromov–Witten invariants. \medskip

It is well known that the only uniruled symplectic closed 2–manifold is $S^2$. We note that, for the symplectic closed 2–manifold $M$ other than $S^2$, there is no loop $\ga$ in $\mathrm{Ham}(M,\omega)$ that is a nonconstant circle action near $x_0$, where $x_0$ is fixed by $\ga$ (see \cite[Remark 1.3]{McD}).

\medskip

In \cite{McD}, McDuff used the Seidel element and methods of relative Gromov–Witten invariants to show that the following celebrated result
\begin{theorem}[\cite{McD}, Theorem 1.1]\label{thm:McD}
Suppose that $\mathrm{Ham}(M,\omega)$ contains a loop $\ga$ with a fiexed maximum near which $\ga$ is an effective circle action. Then $(M,\omega)$ is uniruled.
\end{theorem}

Further, McDuff obtained the following result through a technical lemma (see \cite[Lemma 2.22]{McD})

\begin{proposition} [\cite{McD}, Proposition 1.4]\label{prop:McD}
Suppose that $\mathrm{dim} M\leq4$ and the loop $\ga$ has a nondegenerate maximum at $x_{max}$. Then $\ga$ can be homotoped so that it is a nonconstant circle action near its maximum, that when $\mathrm{dim} M=4$ can be assumed effective. Thus $(M,\omega)$ is uniruled.
\end{proposition}

The statement of the above proposition is taken from the arXiv version (see arXiv:0706.0675v4). However, there is a slight error in its journal version (see \cite[Remark 2.1]{Cha}).
The dimensional restriction here is due to the fact that the homotopy
results of Slimowitz \cite{Sli} have only been proved for $\mathrm{dim}\leq 4$. Later, Chance \cite{Cha} studied the degenerate case of the fixed global maximum, combining with results of Slimowitz \cite{Sli}, he obtained the following result 

\begin{theorem}[\cite{Cha}, Theorem 1.2]\label{thm: Cha1}Let $M$ be a symplectic manifold with $\mathrm{dim} M \leq 4$. If $\{\phi^H_t\}$ is any nonconstant loop of Hamiltonian diffeomorphisms with $x_0\in M$ a fixed global maximum, then there is a nonconstant loop $\{\phi^K_t\}$ with $x_0$ still a fixed global maximum, which is an effective $S^1$ action near $x_0$. Furthermore, $\{\phi^K_t\}$ will be homotopic to some iterate of $\{\phi^H_t\}$.
\end{theorem}

We note that there is also a slight error in the above theorem. This theorem is based on the same technical lemma (\cite[Lemma 2.22]{McD}) as Proposition \ref{prop:McD}. For the same reason, the theorem holds only when $\mathrm{dim} M=4$ (see \cite[Remark 2.1]{Cha}). When $M=S^2$, the autonomous two-time rotation is a counterexample; furthermore, for any $\{\phi^H_t\}$ satisfying the condition in the above theorem, $\{\phi^K_t\}$ can not be an effective $S^1$ action near $x_0$ since it is homotopic to some iterate of $\{\phi^H_t\}$. In fact, $\{\phi^K_t\}$ can be a nonconstant circle action near $x_0$.

\medskip

Combining with the results Theorem \ref{thm:McD} and Proposition \ref{prop:McD} of McDuff, the facts that the only uniruled symplectic closed 2–manifold is $S^2$, and that there is no loop $\ga$ in $\mathrm{Ham}(M,\omega)$ that is a nonconstant circle action near a fixed point of $\ga$ when $M$ is a symplectic closed 2–manifold other than $S^2$, Chance proved the following result

\begin{theorem}[\cite{Cha}, Theorem 1.3]\label{thm: Cha2}
If $\mathrm{dim} M\leq4$ and there exists a nonconstant loop of
Hamiltonian diffeomorphisms with a fixed global maximum, then $(M,\omega)$ is uniruled.
\end{theorem}

As a consequence of Theorem A, without modifying their proofs, we obtain

\begin{thmB}The condition $\mathrm{dim}M\leq4$ in Theorem \ref{thm: Cha2} and in the conclusion of the first part of Proposition \ref{prop:McD} can be all removed. The condition about dimension in Theorem \ref{thm: Cha1} and in the conclusion of the second part of Proposition \ref{prop:McD} can be extended to $\mathrm{dim} M\geq 4$.  \end{thmB}
\smallskip

Given a path $\{\phi^H_t\}, t\in [0, 1]$, its Hofer length is defined as

$$\mathcal{L}(\{\phi^H_t\}) = \int_0^1 \left (\max_x H_t(x)-\min_x H_t(x)\right)dt.$$
This allows one to construct a nondegenerate Finsler metric on $\mathrm{Ham}(M,\omega)$, whose geometry has been studied extensively, see, e.g., \cite{BP,KL,LaM95,P}. In particular, in \cite{LaM95} Lalonde and McDuff show that if $\{\phi^H_t\}$ is a Hofer length minimizing geodesic then its generating Hamiltonian has at least one fixed global minimum and one fixed global maximum. 

Let us consider the unit sphere $\mathbb{S}^2=\{z\in \mathbf{R}^3 \mid \|z\|=1\}$ and its standard symplectic form $\omega_0$ on it. We note that $\pi_1(\mathrm{Ham}(\mathbb{S}^2,\omega_0))\simeq\mathbf{Z}_2$. Let $\{\phi^H_t\}$ be the autonomous one-time rotation where $H$ is the height function. It is clear that $\{\phi^H_t\}$ represents the nontrivial element in $\pi_1(\mathrm{Ham}(\mathbb{S}^2,\omega_0))$. By a result of \cite[Corollary 1.5, pp. 38]{LaM95}, $\{\phi^H_t\}$ is a Hofer length minimizing geodesic which is minimal in its homotopic class.

\medskip 
When $M$ is a symplectic closed 2–manifold other than $S^2$, we note that $\pi_1(\mathrm{Ham}(M,\omega))\simeq\{0\}$. Thus, combining the aforementioned fact and Theorem B, we obtain the following theorem, which generalizes \cite[Theorem 1.4]{Cha}.

\begin{thmC}\label{thm: Cha3} Let $(M,\omega)$ be a closed and connected symplectic manifold. We suppose that $\gamma \in \pi_1(\mathrm{Ham}(M,\omega))$ is nontrivial. If there exists a representative $\{\phi^H_t\}$ of $\gamma$ which is Hofer length minimizing, then $(M,\omega)$ is uniruled.
\end{thmC}
\medskip

Kre\v{\i}n extensively investigated the structure of positive paths $\{\gamma(t)\}_{t\in[0,1]}$ in $\Sp(2n)$ in a series of papers in the early 1950s. Subsequently, Gelfand and Lidskii \cite{GL55} examined the topological properties of general paths, particularly in the context of stability theory for periodic linear flows. Concurrently, Bott \cite{B56}  independently explored positive flows within the complexified linear group, focusing on the geometry of closed geodesics. In 1962, the Kre\v{\i}n–Lyubarski\v{\i} theorem \cite{KL62} established the analytic properties of eigenvalues and eigenvectors when systems (\ref{symp-ode}) are analytic. Recently, Y.~Chang, Y.~Long, and the first author in \cite[Theorem 1.3]{CLW} investigated the bifurcation of eigenvalues along positive paths when the eigenvalues lie on the unit circle, extending the result of Kre\v{\i}n–Lyubarski\v{\i} to the $C^1$-case. As a result of these works, the structure of positive paths is well understood provided that the spectrum of $\{\ga(t)\}$ remains on the unit circle. However, we need to investigate their behavior outside this circle. In \cite[pp. 363]{LaM}, Lalonde and McDuff noted that ‘‘Surprisingly, a thorough study of the behavior of these paths outside the unit circle does not seem to exist." In their article, they conducted such a study in the case of generic positive paths of dimension 4.

Let $\pi:\Sp(2n)\to\mathcal{C}onj(\Sp(2n)$ denote
the projection:
$$\pi(A)=\bigcup_{X}
    \Big\{
    XAX^{-1}:\;X\in\Sp(2n)
    \Big\}\in\mathcal{C}onj(\Sp(2n)).$$ For every $A\in \Sp(2n)$, let $\mathcal{P}_A\subset T_A\Sp(2n)$ be the positive cone on $A$: $\mathcal{P}_A=\{JPA \mid P=P^T, P>0\}$. As an intermediate step in our results, we have the following proposition (see the proof in Section \ref{sec:4.1}), which provides information about the behavior of positive paths outside the unit circle.
\begin{propD}
	If $A\in\Sp(2n)$ is a truly hyperbolic symplectic matrix, i.e. there are no eigenvalues of $A$ on the unit circle,
	then the projection $\pi:\Sp(2n)\rightarrow\mathcal{C}onj(\Sp(2n))$
	maps the positive cone $\mathcal{P}_A$ at $A$ onto the tangent space to $\mathcal{C}onj(\Sp(2n))$ at $\pi(A)$.
\end{propD}

\subsection{Outline of the proof}

Now let's introduce the idea central to our proof.
For technical reasons, we will consider 
the concept of
{\it generic path} (as defined in \cite{LaM,Sli}; or see Definition \ref{def:generic homotopy} below) and {\it generic homotopy} (see Definition \ref{def:generic homotopy} below).

Consider a positive loop $\gamma=\{\gamma(t)\}_{t\in[0,1]}$ based at the identity. 
Since the set of positive paths is open in the $C^1$-topology, $\gamma$ can be positively homotopic to
a generic loop based at the identity. Therefore, we only need to consider
the generic loops.

The main difficulties involve 
dealing with collisions of the positive paths, where a collision refers to a non-diagonalizable symplectic matrix with multiple eigenvalues.

\subsubsection{The idea of the proof for the case in $\Sp(4)$}\label{sec:sp4}


In the first step, we will move all the collisions from $\U\backslash\{\pm1\}$ to $\R\backslash\{0,\pm1\}$, where $\U$ is the unit circle in $\C$.
Then, we obtain a positive path that
only has collisions on $\R\backslash\{0,\pm1\}$.
In the second step, we will eliminate all
the collisions on $\R\backslash\{0,\pm1\}$.
Finally, according to the Positive Decomposition Lemma (see Proposition \ref{Lm:positive.decomposition} below), 
such a loop can be decomposed into the direct sum
of two positive path in $\Sp(2)$.
Hence, according to Slimowitz's result for $\Sp(2)$, we complete this case. Below, we provide a more detailed description.

{\bf Step 1.} 
Let $\gamma$ be a generic positive loop in $\Sp(4)$. 
A key observation is the constraints
of the collisions,
arising from the positivity of paths in $\Sp(4)$.
Indeed, we divide the collisions into two cases based on whether they belong to $\U\backslash\{\pm1\}$ or $\R\backslash\{0,\pm1\}$.
In the case where the collisions on $\U\backslash\{\pm1\}$, they fall into the following two cases:

(i) if the eigenvalues of $\gamma$ leaves $\U$, enters $\C\backslash(\U\cup\R)$
through a collision
	at time $t^*$,
	then after the previous collision (if it exist),
	there must exist a pair of eigenvalues of the path which comes from $\{1,1\}$ or $\{-1,-1\}$ at a time $t_0 < t^*$, see the top figure in Figure \ref{move_collision};
	
(ii) if the eigenvalues of $\gamma$ leaves $\C\backslash(\U\cup\R)$, enters $\U$
through a collision
	at time $t^*$, then before the next collision (if it exist), there must exist a pair of eigenvalues of the path entering $\{1,1\}$ or $\{-1,-1\}$ at a time $t_0 > t^*$.

We refer to the positive path $\mu=\gamma|_{[t_0-\epsilon,t^*+\epsilon]}$
in (i) (or $\mu=\gamma|_{[t^*-\epsilon,t_0+\epsilon]}$ in (ii)) as an {\it elementary collide-out path} ({\it elementary collide-in path}).
For more details, see Definition \ref{elementary.collide.path} below.

\begin{figure}[ht]
	\centering
	\includegraphics[height=12.0cm]{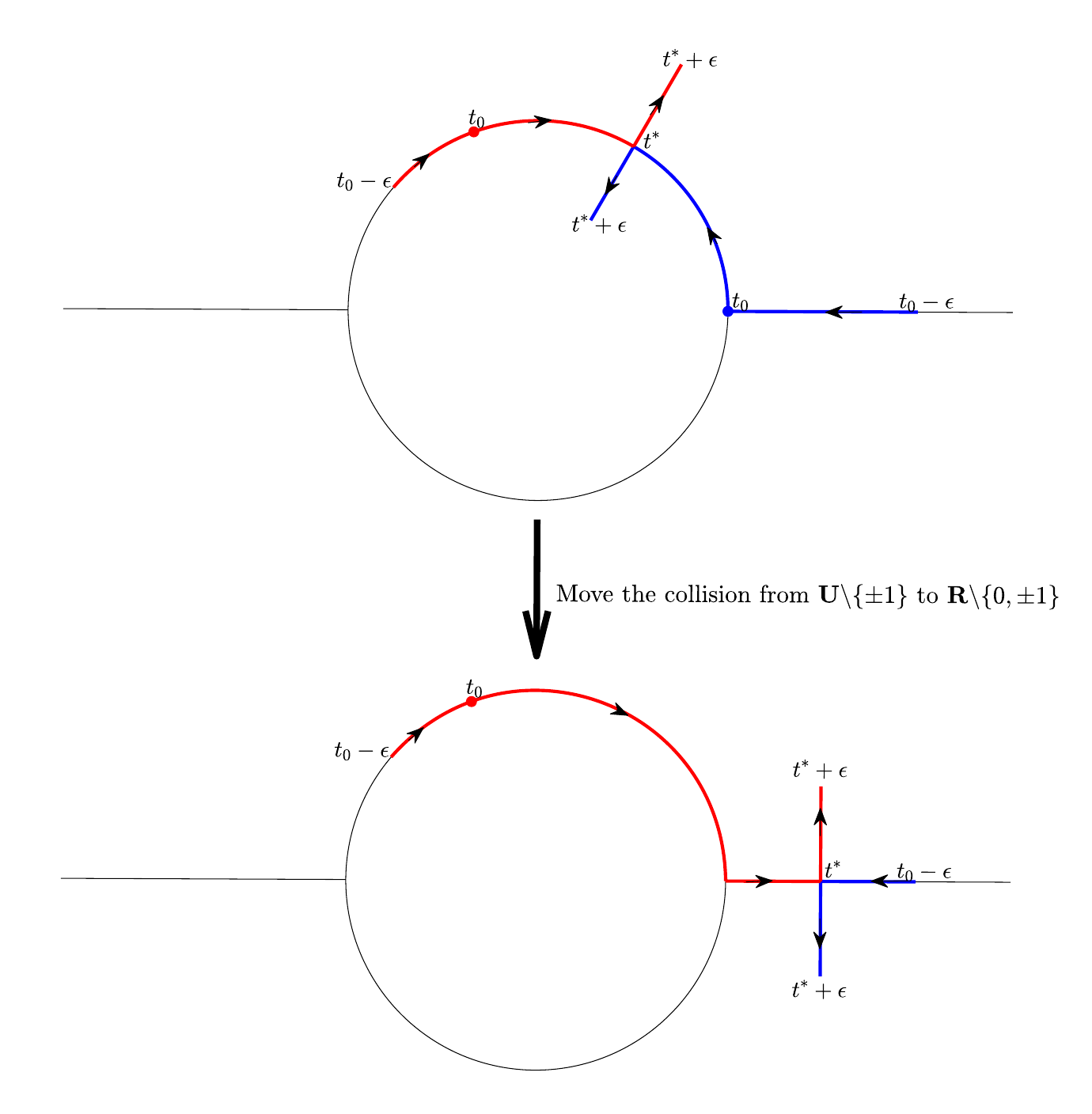}
	\vskip -0.5 cm
	\caption{Move the collision in Step 1. The curves with different colors represent two
 different pairs of eigenvalues, respectively.}
 \label{move_collision}
\end{figure}

In the case where the collisions on $\R\backslash\{\pm1\}$, we have only one case:

(iii) the eigenvalues of $\gamma$ leave $\C\backslash(\U\cup\R)$ and enter $\R\backslash\{0,\pm1\}$,
or leave $\R\backslash\{0,\pm1\}$ and enter $\C\backslash(\U\cup\R)$,
through a collision at time $t^*$.

The main objective of this step is to transfer
the collisions of type (i) or type (ii)
to type (iii). 
Thus, for any collision point of $\gamma$ on $\U\backslash\{\pm1\}$,
there exists a sub-path of $\gamma$ on time interval $[t_1,t_2]$ that
contains the collision point,
which is either  an elementary collide-in path
or an elementary collide-out path.
Furthermore, we can assume that $\gamma(t_1)$
and $\gamma(t_2)$ are generic points.
By Theorem \ref{Th:transfer.collisions}
(refer to Figure \ref{move_collision} for an illustration),
there exists a positive path $\mu(t),t\in[t_1,t_2]$, which has only one collision on $\R\backslash\{0,\pm1\}$,
and a $C^1$-connectable positive homotopy 
(see Definition \ref{pos.connectable.homotopy} below), denoted as $M(s,t)$ such that
$M(0,t)=\gamma(t),M(1,t)=\mu(t)$ for $t\in[t_1,t_2]$
and $M(s,t_1)=X_1(s)^{-1}\gamma(t_1)X_1(s),\;M(s,t_2)=X_2(s)^{-1}\gamma(t_2)X_2(s)$ for some continuous paths $X_1,X_2$.
To extend such a homotopy to the entire loop,
we define
$$
    H(s,t)=\left\{
    \matrix{X_1(s)^{-1}\gamma(t)X_1(s), & 0\le t\le t_1,\cr
    M(s,t), & t_1\le t\le t_2,\cr
    X_2(s)^{-1}\gamma(t)X_2(s), & t_2\le t\le 1.
    }\right.
$$
Since $M$ is a $C^1$-connectable positive homotopy, $H(s,t)$ is $C^1$ differtiable with respect to $t$. 
This effectively moves a collision on $\U\backslash\{\pm1\}$ to $\R\backslash\{0,\pm1\}$. 

As the path is generic, the number of collisions is finite. 
By repeating the operation above, we can move all the collisions from $\U\backslash\{\pm1\}$ to $\R\backslash\{0,\pm1\}$.

{\bf Step 2.} 
Now, let's consider a positive path $\gamma$ with possible collisions on $\R\backslash\{0,\pm1\}$.
Suppose the eigenvalues first collide at $t=t_1$,
then they must collide at $\R^+$ (or $\R^-$) from $\mathcal{O}_\mathcal{R}$ to $\mathcal{O}_\mathcal{C}$.
Since the loop must return to $I$,
there must be a subsequent collision
from $\mathcal{O}_\mathcal{C}$ to $\mathcal{O}_\mathcal{R}$.
Let's assume they successively collide at $t=t_2$,
which must collide from $\mathcal{O}_\mathcal{C}$ to $\mathcal{O}_\mathcal{R}$.
We have only four cases:

(i) The path starts from $\mathcal{O}_{\mathcal{R}^+,\mathcal{R}^+}$,
enters $\mathcal{O}_\mathcal{C}$, and then returns to $\mathcal{O}_{\mathcal{R}^+,\mathcal{R}^+}$;

(ii) The path starts from $\mathcal{O}_{\mathcal{R}^-,\mathcal{R}^-}$,
enters $\mathcal{O}_\mathcal{C}$, and then returns to $\mathcal{O}_{\mathcal{R}^-,\mathcal{R}^-}$;

(iii) The path starts from $\mathcal{O}_{\mathcal{R}^+,\mathcal{R}^+}$,
enters $\mathcal{O}_\mathcal{C}$, and then returns to $\mathcal{O}_{\mathcal{R}^-,\mathcal{R}^-}$;

(iv) The path starts from $\mathcal{O}_{\mathcal{R}^-,\mathcal{R}^-}$,
enters $\mathcal{O}_\mathcal{C}$, and then returns to $\mathcal{O}_{\mathcal{R}^+,\mathcal{R}^+}$.

Here, $\mathcal{O}_\mathcal{C}$, $\mathcal{O}_\mathcal{U}$, $\mathcal{O}_\mathcal{R}$, and $\mathcal{O}_\mathcal{U,R}$ are four open regions in $\Sp(4)$ whose union is dense in $\Sp(4)$. These regions are defined by Lalonde and McDuff in Section 3.2 of \cite{LaM}. For convenience, these open regions, along with other higher codimension regions, are also listed in Section \ref{subsec:5.1} of the current paper.
Moreover, $\mathcal{O}_{\mathcal{R^+,R^+}}$ 
(or $\mathcal{O}_{\mathcal{R^-,R^-}}$) represent
the subset of $\mathcal{O}_{\mathcal{R}}$,
which consists of all matrices with both pair of eigenvalues in $\R^+\backslash\{1\}$ (or $\R^-\backslash\{-1\}$).

For Cases (iii) and (iv), employing a  similar technology of Step 1, we can move the collision from $\R^+$ (or $\R^-$) to $\R^-$ (or $\R^+$) through $\U$.
Consequently, these two cases can be transformed into Case (i) or Case (ii).
Thus, $\gamma|_{[t_1-\ep,t_2+\ep]}$
becomes a truly hyperbolic path.
Then, by Lemma \ref{Lemma:hyperbolic.positive.path} and Theorem \ref{Thm:remove.hyperbolic.collisions},
there exists another truly hyperbolic path
$\mu(t),t\in[t_1-\ep,t_2+\ep]$ entirely within $\mathcal{O}_\mathcal{R^+,R^+}$ (or $\mathcal{O}_\mathcal{R^-,R^-}$),
such that $\gamma$ and $\mu$ are $C^1$-connectably homotopic.
Therefore, the two collisions at times $t=t_1$ and $t=t_2$ have been eliminated.

By repeating the operation above, we can eliminate all the collisions on $\R\backslash\{0,\pm1\}$.

{\bf Step 3.} 
Now the loops have no collisions of eigenvalues except for $\pm1$.
According to Lemma 4.6 in \cite{Sli} (the corrected
version in our paper), 
there exist two positive loops $\gamma_1$
and $\gamma_2$ in $\Sp(2)$
such that $\gamma$ and $\diag(\gamma_1,\gamma_2)$ are positively
homotopic.
By Theorem 3.3 and  Lemma 4.7 in \cite{Sli}, we complete the proof of our main theorem for $\Sp(4)$.
\medskip

We now point out the main differences between our proof
and the proof in \cite{Sli} for $\Sp(4)$.

In the proof of Lemma 4.3 in \cite{Sli},
the author constructed a model of positive path
such that its eigenvalues remain on
one pair of conjugate rays in $\pi(\mathcal{O}_\mathcal{C})$
and simply go out along these rays
to a point where the norm of the eigenvalue reaches its maximum and then returns. Note that the eigenvalues of the two collisions in the model are the same. The first sentence in the proof of Lemma 4.3 suggests that any two positive paths are positively homotopic if they share the same conjugacy classes at their endpoints. However, there are two issues with this claim. The first relates to the Part Speed Changing Lemma (refer to Lemma \ref{Lemma:change.part.time} below). The second concerns the global movement of collisions, similar to Lemma \ref{Th:move.collisions.on.U} in our paper.

In the proof of Lemma 4.2 in \cite{Sli}, 
any generic positive path in $\mathcal{C}onj(\Sp(4))$ can be broken up into finitely many sections
lying in $\mathcal{S}$ connected by parts of type $(1)$, $(2)$, $(3)$, or $(4)$,
where $\mathcal{S}$ is the set of all open strata with eigenvalues in $\U\cup \R$ along with some boundary components to make it a connected set. Its cutting method (cut points on $\pi(\mathcal{O}_\mathcal{U})$) cannot be extended to higher-dimensional cases.
However, our cutting method is based on the constraints of the collisions, as stated in Step 1.
Such a method can effectively be extended to higher-dimensional cases.
Corresponding to our cutting method, the models we constructed are the elementary collide-in path and the elementary collide-out path (refer to Step 1). 

\bigskip

In the current paper, we provide a rigorous proof for the positive decomposition of all-dimensional positive paths (see Lemma \ref{Lm:positive.decomposition} below). This Positive Decomposition Lemma guarantees the existence of the positive paths $\{X_t\},\{Y_t\}$ in \cite[Lemma 4.6]{Sli}. Since the motion on conjugacy classes of
$X_t$ (or $Y_t$) satisfies the constraints imposed by the positivity of $\pi(\diag(X_t,Y_t))$, there exists a positive path in $\Sp(2)$
with the prescribed motion on conjugacy classes.
In a more general situation,
it is interesting to ask the following question which is used in \cite{Sli}:

\begin{question}\label{qu.lift}
  Given a path $\{a_t\}_{t\in[0,1]}$ in  $\mathcal{C}onj(\Sp(2n))$ described by the motions of the eigenvalues, is there a positive path $\{\ga_t\}_{t\in[0,1]}$ such that $\pi(\ga_t)=a_t$ for $t\in[0,1]$?
\end{question}
This is a  nontrivial problem. In Section 3(b) of \cite{Sei},
such an existence problem for $\Sp(2)$ is discussed by Seidel.
For more general situations, the partial necessary conditions are provided in \cite[Theorem 1.3]{CLW}, where the authors gave the regularity of time $t$
when the eigenvalues collide on $\U$.
To solve Question \ref{qu.lift}, a key point is to know
the regularity of the collisions among the eigenvalues with respect to $t$, not just on $\U$. Unlike [23], our proofs do not require an answer to Question 1.1. In our proof,
we construct positive paths and homotopies directly in $\Sp(2n)$, solely with the aid of the geometric intuitions in $\mathcal{C}onj(\Sp(2n))$, thus enabling us to avoid such obstacles.

\subsubsection{The idea of the proof of the case in $\Sp(2n)$ for $n\ge3$}
\label{subsec, n>2}
A key method is to reduce the dimension. Roughly speaking, if the eigenvalues of some positive path $\gamma$ can be decomposed into two independent parts on some time interval—meaning that the only collisions exist within each part separately, and there are no collisions between these two parts—then $\gamma$ can be decomposed into the direct sum of two lower-dimensional positive paths $\gamma_1$ and $\gamma_2$ in the sense of $C^1$-connectable positive homotopy. Subsequently, we can consider the collisions of $\gamma_1$ and $\gamma_2$, respectively.

For paths in $\Sp(4)$, collisions, except for $\pm1$, must occur between two pairs of eigenvalues. However, when considering $\Sp(2n)$ for $n\geq 3$, collisions, except for $\pm1$, may involve more than two pairs of eigenvalues. To distinguish between different collisions, we use the term ``{\it $n$-collision}" to represent a collision among $n$ pairs of eigenvalues (or $n$ quadruple of eigenvalues on  $\C\backslash(\U\cup\R)$). In particular, for "2-collision" and "3-collision", we also denote them as ``{\it double-collision}" and ``{\it triple-collision}".

For the collisions of a generic positive loop $\gamma$ in $\Sp(2n)$, where $n \geq 3$, the constraints raised from
the positiveness of paths become
more complicated. As we consider the generic positive loop, we do not need to consider the collisions in $\C\backslash(\U\cup\R)$ since this case belongs to the codimension 2 or higher strata in $\Sp(2n)$. Hence we can only consider the collisions on $\U\backslash\{\pm1\}$ or $\R\backslash\{0,\pm1\}$.

However, for collisions on $\U\backslash\{\pm1\}$,
except for a similar constraints outlined in Step 1 for $\Sp(4)$,
we have more situations.
Namely, if a quadruplet of eigenvalues of $\gamma$ exits from $\U\backslash\{\pm1\}$ and enters $\C\backslash(\U\cup\R)$ at time $t_2$, denoted by ${\lambda_1,{1\over\lambda_1}}$ and ${\lambda_2,{1\over\lambda_2}}$ (when $\lambda_1,\lambda_2\in\C\backslash(\U\cup\R)$, we assume the relation $\lambda_1\bar\lambda_2=1$ holds,
and hence $\{\lambda_1,{1\over\lambda_1},\lambda_2,{1\over\lambda_2}\}$ forms a quadruple
of eigenvalues), 
then one of the following cases must occur:
	
	(i) after the most recent preceding double-collision (if it exists) with ${\lambda_i,{1\over\lambda_i}},i=1,2$ of $\gamma$, one pair of eigenvalues, denoted as ${\lambda_1,{1\over\lambda_1}}$, enters from $\{1,1\}$ or $\{-1,-1\}$ at time $t_1 < t_2$;
 
	(ii) one pair of eigenvalues, denoted as $\{\lambda_2,{1\over\lambda_2}\}$,
	collide with a third pair of eigenvalues $\{\lambda_3,{1\over\lambda_3}\}$ of $\gamma$
	at time $t_1<t_2$ from $\C\backslash(\U\cup\R)$,
	and  then $\{\lambda_i,{1\over\lambda_i}\},i=1,2,3$ remain on $\U\backslash\{\pm1\}$ for $t\in(t_1,t_2)$, see the top figure in Figure \ref{eliminate_collisions}.

\begin{figure}[ht]
	\centering
	\vskip -2 mm
	\includegraphics[height=15.0cm]{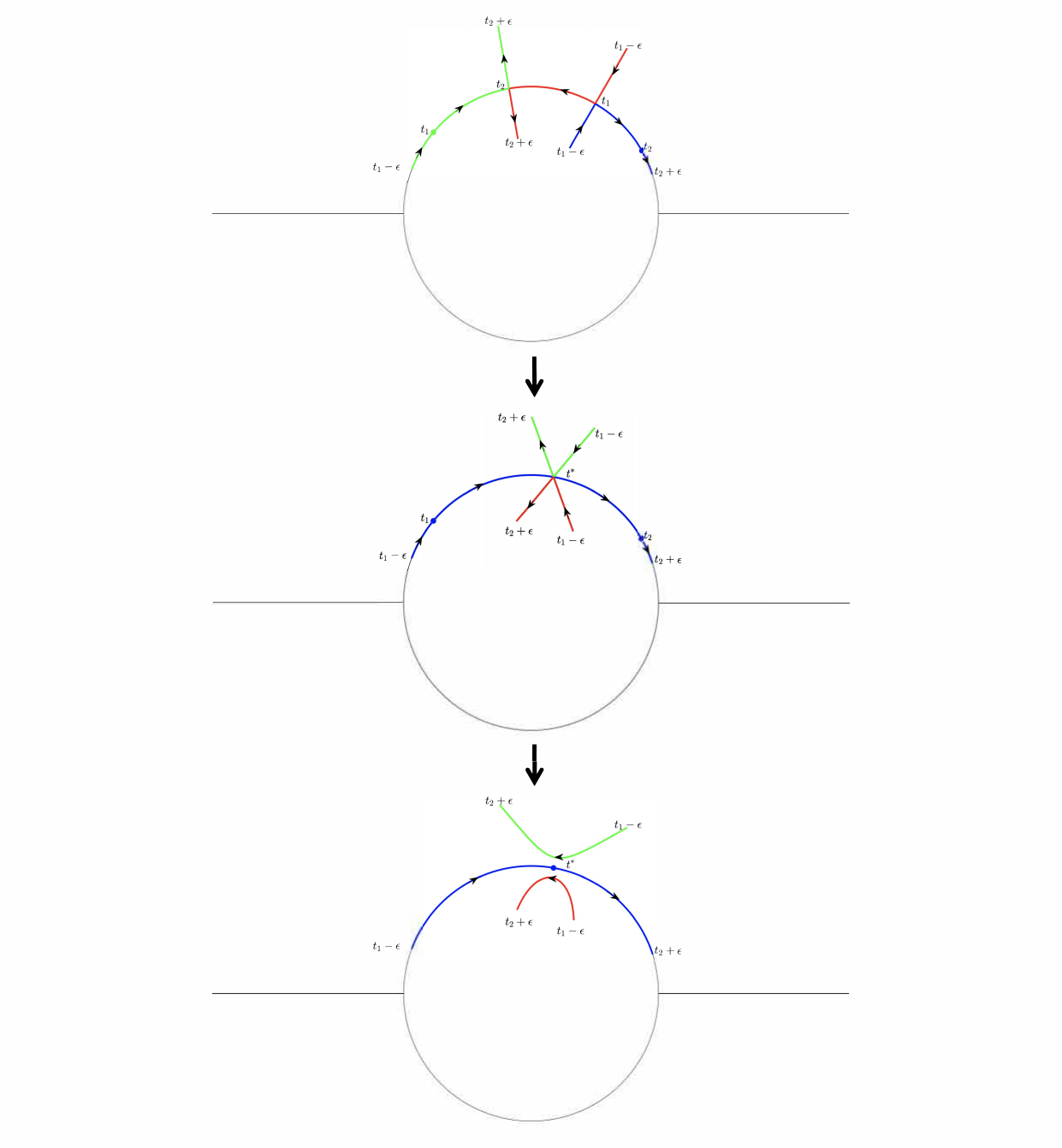}
	\vskip -0.5 cm
	\caption{Eliminate the two double-collisions through a triple-collision path. The curves with different colors represent three
 different pairs of eigenvalues, respectively.}
 \label{eliminate_collisions}
\end{figure}

If the first case occurs, 
together with the Positive Decomposition Lemma,
we can handle the collision in a similar way to $\Sp(4)$. Thus, we only need to address the second case.

For Case (ii), generally, on the time interval $t \in (t_1, t_2 + 2\epsilon)$ for some sufficiently small $\epsilon > 0$, the third pair of eigenvalues ${\lambda_3, {1\over\lambda_3}}$ may collide with eigenvalues other than ${\lambda_i, {1\over\lambda_i}}, i=1,2$. By adjusting the speed of certain eigenvalues (see Lemma \ref{Lemma:change.part.time} below), we can move these collisions involving ${\lambda_3, {1\over\lambda_3}}$ between times $t_2+\epsilon$ and $t_2+2\epsilon$ via a positive homotopy. Consequently, the other eigenvalues do not collide with ${\lambda_i, {1\over\lambda_i}}, i=1,2,3$ within the time interval $(t_1-\epsilon, t_2+\epsilon)$. Hence, based on the Positive Decomposition Lemma,
there exists a symplectic path $X(t), t\in(t_1-\epsilon, t_2+\epsilon)$ such that

$$
\gamma(t)=X(t)^{-1}\left(\matrix{\gamma_1(t)& O\cr O& \gamma_2(t)}\right)X(t).
$$

Here $\gamma_1\in\Sp(6), \gamma_2\in(\Sp(2(n-3)))$ 
are positive paths, and
$$
\sigma(\gamma_1(t))
=\{\lambda_1(t),{1\over\lambda_1(t)},
   \lambda_2(t),{1\over\lambda_2(t)},
   \lambda_3(t),{1\over\lambda_3(t)}\}.
$$
We refer to the positive path
$\mu=\gamma_1|_{[t_1-\epsilon,t_2+\epsilon]}$
as an {\it elementary two double-collisions path} (see Definition \ref{elementary.two.double-collision.path} below).

Now, similar to the argument for $\Sp(4)$, the proof contains three steps.

In the first step, 
for Case (i), we will move the collision from $\U\backslash\{\pm1\}$ to $\R\backslash\{0,\pm1\}$,
while for Case (ii), we will eliminate
the two collisions directly on $\U$
through a triple-collision path
(refer to Figure \ref{eliminate_collisions} for an illustration, and for a detailed proof,
see Theorem \ref{Thm:Sp6.remove.collisions.on.U} below).
Repeating these two procedures,
we obtain a positive path which
only has collisions on $\R\backslash\{0,\pm1\}$.
In the second step, we will eliminate all
the collisions on $\R\backslash\{0,\pm1\}$.
At last, according to the Positive Decomposition Lemma, 
such a loop can be decomposed into the direct sum
of $n$ positive paths in $\Sp(2)$;
and hence, according to the result of $\Sp(2)$ by Slimowitz, we complete the proof.

\bigskip


This paper is organized as follows:
\begin{itemize}
    \item In Section \ref{sec:2}, we provide the normal forms of symplectic matrices, discuss families of symplectic matrices in general position, and present bifurcation diagrams for generic families that reflect the class partition structure in the space of $\Sp(2n)$. 
    \item Section \ref{sec:3} focuses on basic properties of positive paths and positive homotopies,
    as well as useful lemmas frequently employed in our proof. 
    We prove Theorem \ref{Thm:pos.homotopy.of.loops.has.no.collisions}.
    \item Section \ref{sec:4} delves into the study of positive homotopy in the truly hyperbolic set.
    \item In Section \ref{sec:5}, we explore positive homotopy in $\Sp(4)$. We prove Theorem \ref{Th:transfer.collisions} and Theorem \ref{Thm.remove.collisions.on.R.of.Sp4} for $\Sp(4)$.
    \item Section \ref{sec:6} examines positive homotopy in $\Sp(2n)$ for $n > 2$. We prove Theorem \ref{Thm:Sp6.remove.collisions.on.U} and Theorem \ref{Thm.remove.collisions.on.R.of.Sp2n} for $\Sp(2n)$. 
    \item The appendix contains the proofs of some technical lemmas and facts.
\end{itemize}

\setcounter{equation}{0}
\section{Preliminaries}
\label{sec:2}

\subsection{The normal forms of symplectic matrices}
\label{subsec:2.1}

Consider the standard symplectic vector space $(\R^{2n}, \omega)$ with coordinates $(x_1,y_1,\cdots,x_n,y_n)$ and the symplectic form $\omega=\sum_{i=1}^ndx_i\wedge dy_i$. Let $J_{2n}=\diag(J_2,\cdots,J_2)$ be the standard symplectic matrix, where $J_2=\left(\matrix{0& -1\cr 1& 0}\right)$. For simplicity, we will omit the subscript and denote it by $J$ when there is no confusion.

As usual, the symplectic group $\Sp(2n)$ is defined by
$$ \Sp(2n) = \{M\in {\rm GL}(2n,\R)\,|\,M^TJM=J\}, $$
whose topology is induced from that of $\R^{4n^2}$. A path in $\Sp(2n)$ is a continuous map $\gamma: I \rightarrow \Sp(2n)$ defined on a nontrivial finite closed (or open, or half-closed) interval $I$, for example, $I=[a,b]$, $(a,b)$, $[a,b)$, or $(a,b]$, where $a < b$. For clarity, we will use either $\{\gamma(t)\}_{t \in I}$ or $\{\gamma_t\}_{t \in I}$ to represent this path. Sometimes, when it doesn't cause ambiguity, we omit $t \in I$.

The Lie algebra ${\bf sp}(2n)$ of $\Sp(2n)$ is
the set of all matrices which satisfy the equation
$$A^TJ+JA=0.$$
Here $A\in{\bf sp}(2n)$ if and only if $JA$ is symmetric.

An $M\in\Sp(2n)$ is {\bf truly hyperbolic} if
none of its eigenvalues belongs to $\U$.
We denote by $\Sp^{th}(2n)$ the subsets of
all truly hyperbolic matrices in $\Sp(2n)$.
For $\omega\in\U$ and $M\in\Sp(2n)$, we define
$$
D_{\omega}(M)=(-1)^{n-1}\omega^{-n}\det(M-\omega I_{2n}).
$$
According to Lemma 2 on pp.37 of \cite{Lon},
$D$ is a real and real smooth function
on $\U\times\Sp(2n)$.

Then we define the $\om$-singular surface by
$\Sp(2n)_{\omega}^0 = \{M\in\Sp(2n)\,|\, D_{\omega}(M)=0\}$ and the $\om$-regular set by
$\Sp(2n)_{\omega}^{\ast} = \Sp(2n)\setminus \Sp(2n)_{\omega}^0$. The orientation of $\Sp(2n)_{\omega}^0$ at any of its point
$M$ is defined to be the positive direction $\frac{d}{dt}Me^{t J}|_{t=0}$ of the path $Me^{t J}$ with $t>0$ small
enough. Let $\nu_{\omega}(M)=\dim_{\bf C}\ker_{\bf C}(M-\omega I_{2n})$ and
$\mathcal{P}_{1}(2n) = \{\gamma\in C([0,1],\Sp(2n))\;|\;\gamma(0)=I_{2n}\}$. We denote
$$\xi(t)={\rm diag}(2-t, (2-t)^{-1})\quad\mathrm{and}\quad\xi^n(t)=\diag(\underbrace{\xi(t),\cdots,\xi(t)}_n) \quad \mathrm{for}\quad 0\le t\le1.$$
For any $\gamma\in \mathcal{P}_{1}(2n)$ we define $\nu_\omega(\gamma)=\nu_\omega(\gamma(1))$ and
$$  i_\omega(\gamma)=[\Sp(2n)_\omega^0:\gamma\ast\xi^n], \qquad {\rm if}\;\;\gamma(1)\not\in \Sp(2n)_{\omega}^0,  $$
i.e., the usual homotopy intersection number, and the orientation of the joint path $\gamma\ast\xi^n$ is
its positive time direction under homotopy with fixed end points. When $\gamma(1)\in \Sp(2n)_{\omega}^0$,
we define $i_{\omega}(\gamma)$ be the index of the left rotation perturbation path $\gamma_{-\epsilon}$ with $\epsilon>0$
small enough (cf. Def. 5.4.2 on pp. 129 of \cite{Lon}). The pair
$(i_{\omega}(\gamma), \nu_{\omega}(\gamma)) \in {\bf Z}\times \{0,1,\cdots,2n\}$ is called the (Maslov-type) index function of $\gamma$
at $\omega$. When $\nu_{\omega}(\gamma)=0$ ($\nu_{\omega}(\gamma)>0$), the path $\gamma$ is called
$\omega$-{\it non-degenerate} ($\omega$-{\it degenerate}).
For more details we refer to \cite{Lon}.

As in \cite{Lon}, for $\lm\in\R\bs\{0\}$, $a\in\R$, $\th\in (0,\pi)\cup (\pi,2\pi)$,
$b=\left(\matrix{b_1 & b_2\cr
	b_3 & b_4\cr}\right)$ with $b_i\in\R$ for $i=1, \cdots, 4$, and $c_j\in\R$
for $j=1, 2$, we denote respectively some normal forms by
\bea
&& D(\lm)=\left(\matrix{\lm & 0\cr
	0  & \lm^{-1}\cr}\right), \qquad
R(\th)=\left(\matrix{\cos\th & -\sin\th\cr
	\sin\th  & \cos\th\cr}\right),  \qquad
 N_1(\lm, a)=\left(\matrix{\lm & a\cr
	0   & \lm\cr}\right), \nn\\
&&N_2(\th,b)=\left(\matrix{\cos\th &   b_1 &       -\sin\th &         b_2 \cr
	0 & \cos\th &       0 & -\sin\th \cr
	\sin\th &   b_3 &  \cos\th &         b_4 \cr
	0 &   \sin\th & 0 &  \cos\th \cr}\right), \qquad
M_2(\lm,c)=\left(\matrix{\lm &   c_1 &       1 &         0 \cr
	0 & \lm^{-1} &       0 & 0 \cr
	0 &   c_2 &  \lm &         (-\lm)c_2 \cr
	0 &   -\lm^{-2} & 0 &  \lm^{-1} \cr}\right). \nn\eea
Note that we use $\th$ instead of $\om=e^{\sqrt{-1}\th}$
in the representation of $N_2(\th,b)$,
and the normal forms have been adjusted due to the different standard symplectic matrices in our paper and \cite{Lon}.
Moreover, $N_2(\theta,b)$ is {\bf trivial} if $(b_2-b_3)\sin\th>0$, or {\bf non-trivial}
if $(b_2-b_3)\sin\th<0$, following the Definition 1.8.11 on pp. 41 of \cite{Lon}. 

Given any two $2m_k\times 2m_k$ matrices of square block form
$M_k=\left(\matrix{A_k&B_k\cr
                   C_k&D_k\cr}\right)$ with $k=1, 2$,
the ``$\diamond$"-sum of $M_1$ and $M_2$ is defined (cf. \cite{Lon} called product there) by
the following $2(m_1+m_2)\times 2(m_1+m_2)$ matrix $M_1\dm M_2$:
$$
M_1\dm M_2=\left(\matrix{A_1 &   0 & B_1 &   0\cr
                             0   & A_2 &   0 & B_2\cr
                             C_1 &   0 & D_1 &   0\cr
                             0   & C_2 &   0 & D_2\cr}\right),
$$
and $M^{\dm k}$ denotes the $k$-fold  $\dm$-sum of $M$.

Note that if $b=0$, we have
\begin{equation}
N_2(\th,0) = R(\th)\diamond R(\th).
\end{equation}
Also, according to Theorem 1.5.1 on pp. 24-25 and (1.4.7)-(1.4.8) on pp. 18 of \cite{Lon}, when $\lm=-1$ there hold
\bea
c_2 \not= 0 &{\rm if\;and\;only\;if}\;& \dim\ker(M_2(-1,c)+I)=1, \nn\\
c_2 = 0 &{\rm if\;and\;only\;if}\;& \dim\ker(M_2(-1,c)+I)=2. \nn\eea

In \cite{LaM} and \cite{Sli}, Lalonde, McDuff and Slimowitz
used slightly different normal forms by

\begin{equation}
N_1^-=\left(\matrix{1 & 0\cr
	-1  & 1\cr}\right), \qquad
N_1^+=\left(\matrix{1 & 0\cr
	1   & 1\cr}\right).
\end{equation}
It can be easily verified that
\begin{equation}
	N_1^-=J^TN_1(1,1)J,\qquad
	N_1^+=J^TN_1(1,-1)J.
\end{equation}

	Suppose the eigenvalues of a symplectic matrix $M\in\Sp(4)$ are
	\begin{equation}
	\sigma(M)=\{\lambda_1,{1\over \lambda_1},
                \lambda_2,{1\over\lambda_2}\},
	\end{equation}
	where
	$\mathrm{Im}(\lambda_i)>0$ (or $|\lambda_i|>1$ for $\lambda_i\in\R$).
	Moreover,
	let $\sigma_1$ and $\sigma_2$ represent the first two
	elementary symmetric functions of $\sigma(M)$.
	
	Let
	\begin{equation}
	\mu_i=\lambda_i+{1\over\lambda_i},\quad i=1,2.
	\end{equation}
	Then we have
	\begin{eqnarray}
	\mu_1+\mu_2&=&\sigma_1,
	\\
	\mu_1\mu_2&=&\sigma_2-2,
	\end{eqnarray}
	and consequently, $\mu_1$ and $\mu_2$ are the two roots of 
	the following quadratic polynomial
	\begin{equation}\label{quadratic.poly}
	x^2-\sigma_1x+(\sigma_2-2)=0.
	\end{equation}
	The discriminant of (\ref{quadratic.poly}) is given by
	\begin{eqnarray}
	\Delta
    =\sigma_1^2-4\sigma_2+8.
	\end{eqnarray}
	
	Then, the map $\Delta:\Sp(4)\rightarrow\R$ defined by
    $\Delta(M)=\sigma_1^2-4\sigma_2+8=4({\sigma_1^2\over4}+2-\sigma_2)$,
    is a multivariate polynomial
	with respect to the entries of the symplectic matrix.
    We have:

\hangafter 1
\hangindent 3.5em
\;\;(i) If $\Delta>0$, 
    the quadratic polynomial (\ref{quadratic.poly}) has two real roots,
	indicating that $M$ has two different pairs of 
    eigenvalues on $\U$ (or $\R$);

\hangafter 1
\hangindent 4em
\;(ii) If $\Delta=0$, the quadratic polynomial (\ref{quadratic.poly}) has one real root
    with multiplicity 2,
	indicating that $M$ has one pairs of  eigenvalues on $\U$ (or $\R$)
 with multiplicity 2.
 Furthermore, if $M$ is non-diagonalizable,
 it is a double-collision on $\U$ (or $\R$);

\hangafter 1
\hangindent 4em
(iii) If $\Delta<0$, 
    the quadratic polynomial (\ref{quadratic.poly}) has a pair of conjugate complex roots,
	indicating that $M$ has a quadruple of
    $4$ distinct
    eigenvalues in $\C\backslash(\U\cup\R)$.
	
    \noindent Note that in Remark 3.7(i) of \cite{LaM},
    Lalonde and McDuff discussed the properties of
    the matrices in $\mathcal{O}_{\mathcal{C}}$
    that satisfy the condition $\sigma_2>{\sigma_1^2\over4}+2$,
    while in $\mathcal{O}_{\mathcal{R}}$
    , matrices satisfy the condition $\sigma_2<{\sigma_1^2\over4}+2$.
    Additionally, in Lemma 2 on pp. 43 of \cite{Lon}, Long
    used $\psi=4A^2+2-B$ to study the perturbed paths,
    and the relation $\psi={1\over4}\Delta$ holds.

\begin{lemma}[pp. 43 of \cite{Lon}]\label{lemma:B_U}
	
	For $\om=e^{\th\sqrt{-1}}$ with
	 $\th\in(0,2\pi)\backslash\{\pi\}$ and $(b_1,b_2,b_3,b_4)\in\R^4$, 
	 consider the matrix $M=N_2(\th,b)$ and $M_\alpha=M\cdot\diag(R(\alpha),I_2)$
  for some small $\alpha$.
	 Then
	 
	 $1^\circ$ $M\in\Sp(4)$ if and only if
	  $(b_2-b_3)\cos\th+(b_1+b_4)\sin\th=0$.
	  In the following we always suppose $M\in\Sp(4)$.
	  
	  $2^\circ$ $\om$ and $\bar\om$ are double eigenvalues of $M$.
	  
	  $3^\circ$ $M\in\Sp(4)$, $\dim_{\C}\ker_{\C}(M-\om I)=1$ if and only if ${\rm tr}(J_2b)=b_2-b_3\ne0$.
	  
	  $4^\circ$ $M\in\Sp(4)$, $\dim_{\C}\ker_{\C}(M-\om I)=2$ if and only if ${\rm tr}(J_2b)=b_2-b_3=0$, and hence ${\rm tr}(b)=b_1+b_4=0$.
	  
	  $5^\circ$ $M\in\Sp(4)$, suppose $\dim_{\C}\ker_{\C}(M-\om I)=1$;
	  then there exists $\alpha_0>0$ small enough such that
	  $D_\om(M_\alpha)\ne0$ if $0<|\alpha|<\alpha_0$.
	  Let $\psi(\alpha):=\frac{1}{4}\Delta(M_\alpha)=4A^2+2-B$ for $\alpha\in\R$.
	  Then $\psi(0)=0$ and $\psi'(0)=(b_2-b_3)\sin\th$.
\end{lemma}

\begin{remark}\label{remark:B_U}
	If $b_2-b_3\ne0$, then $N_2(\th,b)\in\mathcal{B}_\mathcal{U}$;
	if $b_2-b_3=0$, then $N_2(\th,b)\in\mathcal{O}_\mathcal{U}$.
\end{remark}

	 
	 
	 
	 
	 

\begin{lemma}\label{lemma:B_U.perturbation}
	For $\th\in(0,2\pi)\backslash\{\pi\}$ and $(b_1,b_2,b_3,b_4)\in\R^4$, 
	consider the matrix $N=N_2(\th,b)\in\Sp(4)$ and $N_\alpha=N\cdot\diag(R(\alpha),R(\alpha))$.
    Then

    $1^\circ$
    $\Delta(N_\alpha)=8(b_2-b_3)\sin\th\cos\alpha\sin\alpha
+[16\sin^2\th+(b_1-b_4)^2+4b_2b_3]\sin^2\alpha$;

    $2^\circ$ If $N$ is non-trivial, i.e., $(b_2-b_3)\sin\th<0$,
    then there exists
    $\alpha_0>0$ small enough such that
    $\Delta(N_\alpha)>0$ if $\alpha\in(0,\alpha_0)$,
    and $\Delta(N_\alpha)<0$ if $\alpha\in(-\alpha_0,0)$;

    $3^\circ$ If $N$ is trivial, i.e., $(b_2-b_3)\sin\th>0$,
    then there exists
    $\alpha_0>0$ small enough such that
    $\Delta(N_\alpha)<0$ if $\alpha\in(0,\alpha_0)$,
    and $\Delta(N_\alpha)>0$ if $\alpha\in(-\alpha_0,0)$.
\end{lemma} 
\begin{proof}
	The proof of this lemma will be deferred to the Appendix.
\end{proof}

\begin{lemma}\label{lemma:B_1.perturbation}
	For $(c_1,c_2)\in\R^2$ where $c_2<0$, 
	consider the matrix $M^{\pm}_\alpha=M_2(\pm1,c)\cdot\diag(R(\alpha),R(\alpha))$.
    Then we have

    $1^\circ$
    $\Delta(M_2(1,c)\cdot \diag(R(\alpha),R(\alpha)))=
    4c_2\cos\alpha\sin\alpha+[(c_1+c_2)^2+4]\sin^2\alpha$.
    Hence there exists $\alpha_0>0$ small enough such that
    $\Delta(M^+_\alpha)<0$ if $\alpha\in(0,\alpha_0)$,
    and $\Delta(M^+_\alpha)>0$ if $\alpha\in(-\alpha_0,0)$;

    $2^\circ$
    $\Delta(M_2(-1,c)\cdot \diag(R(\alpha),R(\alpha)))=
    4c_2\cos\alpha\sin\alpha+[(c_1-c_2)^2+4]\sin^2\alpha$.
    Hence there exists $\alpha_0>0$ small enough such that
    $\Delta(M^-_\alpha)<0$ if $\alpha\in(0,\alpha_0)$,
    and $\Delta(M^-_\alpha)>0$ if $\alpha\in(-\alpha_0,0)$.
\end{lemma} 
\begin{proof}
	The proof of this lemma will be deferred to the Appendix.
\end{proof}

\begin{lemma}	
For $\lambda\in\R\backslash\{0,\pm1\}$ and $c=(c_1,c_2)\in\R^2$, 
	 consider the matrix $M_2(\lambda,c)$ and $M_\alpha=M_2(\lambda,c)\cdot\diag(R(\alpha),R(\alpha))$
  for some small $\alpha$.
	 Then we have
	 
	 $1^\circ$ $M_2(\lambda,c)\in\Sp(4)$.
	  
	  $2^\circ$ $\lambda$ and ${1\over\lambda}$ are double eigenvalues of $M$.
	  
	  $3^\circ$ $\dim_{\C}\ker_{\C}(M-\lambda I)=1$ and
   $\dim_{\C}\ker_{\C}(M-{1\over\lambda} I)=1$.
	  
	  $4^\circ$         $\Delta(M_2(\lambda,c)\cdot\diag(R(\alpha),R(\alpha)))=4c_2\cos\alpha\sin\alpha+[(c_1+\lambda c_2)^2+{4\over\lambda^2}]\sin^2\alpha$.
    Hence, if $c_2>0$,
    there exists $\alpha_0>0$ small enough such that
    $\Delta(M_\alpha)<0$ if $\alpha\in(0,\alpha_0)$,
    and $\Delta(M_\alpha)>0$ if $\alpha\in(-\alpha_0,0)$.
\end{lemma}
\begin{proof}
The proof of this lemma will be deferred to the Appendix.
\end{proof}

We also require certain basic normal form of symplectic matrices  
in $\Sp(6)$ which represent the triple-collision of the eigenvalues.
According to (35)-(37) on pp. 30 of \cite{Lon},
these normal forms are given by

\begin{eqnarray*}
N_3(\th,b)&=&\left(\matrix{
	\cos\th &   b_1   & -\sin\th & b_2      & 1     & 0\cr
	      0 & \cos\th &    0     & -\sin\th & 0     & 0\cr
	\sin\th &   b_3   &  \cos\th &      b_4 & 0     & 1\cr
	      0 & \sin\th &    0     &  \cos\th & 0     & 0\cr
          0 &    f_1  &    0     &     f_2  &\cos\th&-\sin\th\cr
          0 &    g_1  &    0     &     g_2  &\sin\th&\cos\th\cr
      }\right), \\
\tilde{N}_3(\th,b)&=&\left(\matrix{
	\cos\th &   b_1   & -\sin\th & b_2      & 0     & 1\cr
	0 & \cos\th &    0     & -\sin\th & 0     & 0\cr
	\sin\th &   b_3   &  \cos\th &      b_4 & 1     & 0\cr
	0 & \sin\th &    0     &  \cos\th & 0     & 0\cr
	0 &    f_1  &    0     &     f_2  &\cos\th&\sin\th\cr
	0 &    g_1  &    0     &     g_2  &-\sin\th&\cos\th\cr
}\right).
\end{eqnarray*}

	
	Suppose the eigenvalues of a symplectic matrix $M\in\Sp(6)$ are
	$$
	\sigma(M)=\{\lambda_1,\lambda_2,\lambda_3,{1\over \lambda_1},{1\over\lambda_2},{1\over\lambda_3}\},
	$$
	where
	$\mathrm{Im}(\lambda_i)>0$ (or $|\lambda_i|>1$ for $\lambda_i\in\R$).
	Moreover,
	let $\sigma_1,\sigma_2$ and $\sigma_3$ represent the first three
	elementary symmetric functions of $\sigma(M)$.
	
	Let
	\begin{equation}\label{mu}
	\mu_i=\lambda_i+{1\over\lambda_i},\quad i=1,2,3.
	\end{equation}
	Then we have
	\begin{eqnarray*}
	\mu_1+\mu_2+\mu_3&=&\sigma_1,
	\\
	\mu_1\mu_2+\mu_2\mu_3+\mu_3\mu_1&=&\sigma_2-3,
	\\
	\mu_1\mu_2\mu_3&=&\sigma_3-2\sigma_1,
	\end{eqnarray*}
	and thus, $\mu_1,\mu_2,\mu_3$ are the three roots of 
	the following cubic polynomial
	\begin{equation}\label{cubic.poly}
	x^3-\sigma_1x^2+(\sigma_2-3)x-(\sigma_3-2\sigma_1)=0.
	\end{equation}
	Let $\Sigma_1=\sigma_1,\Sigma_2=\sigma_2-3,\Sigma_3=\sigma_3-2\sigma_1$,
    then the discriminant of (\ref{cubic.poly}) is given by
	\begin{eqnarray}
	\Delta(\Sigma_1,\Sigma_2,\Sigma_3)&=&B^2-4AC
	\nonumber\\
	&=&3[-\sigma_1^2(\sigma_2-3)^2+27(\sigma_3-2\sigma_1)^2+4(\sigma_2-3)^3+4\sigma_1^3(\sigma_2-3)
	\nonumber\\
	&&\quad-18\sigma_1(\sigma_2-3)(\sigma_3-2\sigma_1)]
	\nonumber\\
	&=&27\sigma_3^2+4\sigma_1^3\sigma_3+4\sigma_2^3-\sigma_1^2\sigma_2^2-18\sigma_1\sigma_2\sigma_3
	\nonumber\\
	&&-[54\sigma_1\sigma_3+8\sigma_1^4+36\sigma_2^2-42\sigma_1^2\sigma_2]+108\sigma_2-9\sigma_1^2-108,
	\end{eqnarray}
	where
	\begin{eqnarray}
	A&=&-3(\sigma_2-{1\over3}\sigma_1^2-3),\label{A}
	\\
	B&=&9\sigma_3-15\sigma_1-\sigma_1\sigma_2,
	\\
	C&=&(\sigma_2-3)^2-3\sigma_1(\sigma_3-2\sigma_1). \label{C}
	\end{eqnarray}
	
	Then, the map $\Delta:\Sp(6)\rightarrow\R$ is a multivariate polynomial
	with respect to the entries of the symplectic matrix.
    We have
	
	\hangafter 1
    \hangindent 3.5em
    \;\;(i) If $\Delta>0$, the cubic polynomial (\ref{cubic.poly}) has a real root and a pair of conjugate complex roots. Consequently, $M$ has one pair of 
    eigenvalues on $\U$ (or $\R$),
    and a quadruple of
    $4$ distinct
    eigenvalues in $\C\backslash(\U\cup\R)$.
	
	\hangafter 1
    \hangindent 3.5em
    \;(ii) If $\Delta=0$, the cubic polynomial (\ref{cubic.poly}) has multiple roots. Consequently, $M$ has a double or triple collision.
	
	\hangafter 1
    \hangindent 3.5em
    (iii) If $\Delta<0$, the cubic polynomial (\ref{cubic.poly}) has three real roots. Consequently, $M$ has three pairs of 
    eigenvalues on $\U$ (or $\R$).
    
We have

\begin{lemma}\label{lemma:N_3}
	For $\om=e^{\th\sqrt{-1}}$ with
	$\th\in(0,2\pi)\backslash\{\pi\}$ and $b=(b_1,b_2,b_3,b_4)\in\R^4$, 
	consider the matrix $N_3(\th,b),\tilde{N}_3(\th,b)$ and $M_\alpha=N_3(\th,b)\cdot\diag(R(\alpha),R(\alpha),R(\alpha))$.
	Then
	
	$1^\circ$ $N_3(\th,b)\in\Sp(6)$ if and only if
	$\left(\matrix{f_1 & f_2\cr g_1 & g_2}\right)=
	\left(\matrix{\sin2\th & \cos2\th\cr -\cos2\th & \sin2\th}\right)$
	and
	$(b_2-b_3)\cos\th+(b_1+b_4)\sin\th=1$.
	
	$2^\circ$ $\tilde{N}_3(\th,b)\in\Sp(6)$ if and only if
	$\left(\matrix{f_1 & f_2\cr g_1 & g_2}\right)=
	\left(\matrix{\cos2\th & -\sin2\th\cr -\sin2\th & -\cos2\th}\right)$
	and
	$(b_2-b_3)\cos\th+(b_1+b_4)\sin\th=-1$.
    In the following we always suppose $N_3(\th,b),\tilde{N}_3(\th,b)\in\Sp(6)$.
	
	$3^\circ$ $\om$ and $\bar\om$ are triple eigenvalues of $N_3(\th,b)$ (or $\tilde{N}_3(\th,b)$).
	
	$4^\circ$ $\dim_{\C}\ker_{\C}(N_3(\th,b)-\om I)=1$ and $\dim_{\C}\ker_{\C}(\tilde{N}_3(\th,b)-\om I)=1$.
	
	$5^\circ$ Let $\Delta(\alpha)=\Delta(M_\alpha)$ for $\alpha\in\R$.
	Then $\Delta(0)=\Delta'(0)=0$ and $\Delta''(0)=10368\sin^6\th>0$.

    $6^\circ$ There exists $\alpha_0>0$ small enough such that
	$D_\om(M_\alpha)\ne0$ if $0<|\alpha|<\alpha_0$.
\end{lemma}
\begin{proof}
	The proof of this lemma will be deferred to the Appendix.
\end{proof}

\subsection{Families of $\Sp(2n)$  
in general position and their bifurcations}

Construction of the most general perturbations
of the normal forms of symplectic matrices
is facilitated by the theory of
versal deformation.

It has to do with the following abstract situation.
Let a Lie Group $G$ act on a smooth manifold $M$.
Two points of $M$ are considered equivalent 
if they belong to the same orbit, meaning they transform into each other under the action of this group.
A family with parameter space (base space) $V$
is a smooth mapping $V\rightarrow M$.
A {\bf deformation} of an element $x\in M$ is
a germ of a family $(V,0)\rightarrow(M,x)$ 
(where $0$ is the coordinate origin in $V\subseteq\R^k$).
One says that the deformation 
$\phi:(V,0)\rightarrow(M,x)$
is induced from the deformation
$\psi:(W,0)\rightarrow(M,x)$
under a smooth mapping of the base spaces
$v:(V,0)\rightarrow(W,0)$
if $\phi=\psi\circ v$.
Two deformations
$\phi,\psi:(V,0)\rightarrow(M,x)$
are called equivalent
if there is a deformation of the identity element
$g:(V,0)\rightarrow(G,id)$ such that
$\phi(v)=g(v)\psi(v)$ for any $v\in V$.

A deformation $\phi:(V,0)\rightarrow(M,x)$
is called {\bf versal} if any deformation
of the element $x$ is equivalent to
a deformation induced from $\phi$.
A versal deformation with the smallest base-space
dimension possible for a versal deformation 
is called {\bf miniversal}.

The germ of the manifold $M$ at the point $x$
is obviously a versal deformation for $x$,
but generally speaking it is not miniversal.

{\bf The Versality Theorem.}
{\it A deformation of the point $x\in M$ is versal
if and only if it is transversal to the orbit
$Gx$ of the point $x$.}

{\bf Corollary.}
{\it The number of parameters of a miniversal
deformation is equal to the codimension
of the orbit.}

Let $M$ once again be $sp(2n)$ (or $\Sp(2n)$)
and $G=\Sp(2n)$.
Now the meaning of the $G$-orbit is the
conjugacy class.
One says that the matrices of a given strata
are not encountered in general $l$-parameter
families if one can remove them by
an arbitrarily small perturbation.
The codimension of a given strata of $\Sp(2n)$
is the smallest number of parameters in
families in which matrices in this strata are
encountered unremovably.

Let $\phi:V\rightarrow M$ be a smooth map between two smooth manifold.
If $Q\subseteq M$ is a stratified submanifold of $M$
consisting of a finite union of disjoing smooth manifolds
(the strata), then the map $\phi$ is said to be transversal
to the stratification $Q$ if it is transversal to each stratum
of $Q$.
Then we have the following corollary:

\begin{corollary}
    In the space of all families of elements of $M=sp(2n)$ (or $\Sp(2n)$), the families transversal to all strata form a residual subset (see also \cite{Koc2,Rob}).
\end{corollary}
Families that are transversal to all strata are said to be in general position.

Recall that a family of elements of $M=sp(2n)$
is a map $\phi:V\rightarrow sp(2n)$
from the parameter space into the Lie algebra.
The bifurcation diagram of a family in general position is a decomposition of the parameter space into a finite union of 
submanifold by taking the inverse image of the strata of $sp(2n)$.
The codimension of a submanifold 
in the bifurcation diagram, within the parameter space of a family in general position, is equal to the codimension of the strata corresponding to this submanifold.

In a family in general position,
i.e., located in the open strata,
almost all
the matrices have simple eigenvalues and
the corresponding submanifolds of the parameter
space have zero codimension.
The exceptional parameter values corresponding 
to matrices with multiple eigenvalues have
nonzero codimension.
The singularities of the bifurcation diagrams
are obtained by analyzing the zero set of
discriminants of characteristic polynomials
of matrices for small values of the parameters
around zero.

By Corollary 1 on pp.~15 of \cite{ArG}
(We also refer to Section 5 of \cite{Koc} for more details),
we have

\begin{lemma}[\cite{ArG}, Corollary 1]\label{bifurcation.digram.of.c=1.2}

In one and two-parameter families of quadratic Hamiltonians one
encounters as irremovable only Jordan blocks 
of the following twelve types:
\begin{equation}
{\rm codimension\; 1:}\quad
(\pm a)^2,\;(\pm\sqrt{-1}a)^2,\;0^2
\end{equation}
(here the Jordan blocks are denoted by their determinants,
for example, $(\pm a)^2$ denotes a pair of Jordan blocks of order $2$
with eigenvalues $a$ and $-a$ respectively);
\begin{eqnarray}
{\rm codimension\; 2:}\quad
&&(\pm a)^3,\;(\pm\sqrt{-1}a)^3,\;(\pm a\pm\sqrt{-1}b)^2,\;0^4,
\nonumber\\
&&(\pm a)^2(\pm b)^2,\;(\pm\sqrt{-1}a)^2(\pm\sqrt{-1}b)^2,\;
(\pm a)^2(\pm\sqrt{-1}b)^2,\;(\pm a)^20^2,\;(\pm\sqrt{-1}a)^20^2
\nonumber\\
\end{eqnarray}
(the remaining eigenvalues are simple).
\end{lemma}

\begin{remark}
    The above lemma describes the classification of
    codimension 1 and codimension 2 stratum of $sp(2n)$.
    By the properties of the exponential map (c.f. \cite[Prop. 20.8]{Lee}),
    there exists a similar classification for $\Sp(2n)$.
\end{remark}

In Figure \ref{figure:bf_c1}, the bifurcation diagrams of generic deformations for the classes of codimension $1$ are presented in the order of their listing
in the statement 
of Lemma \ref{bifurcation.digram.of.c=1.2}.
Here each origin represents the class of codimension 1 strata.

\begin{figure}[ht]
	\centering
	\includegraphics[height=2.8cm]{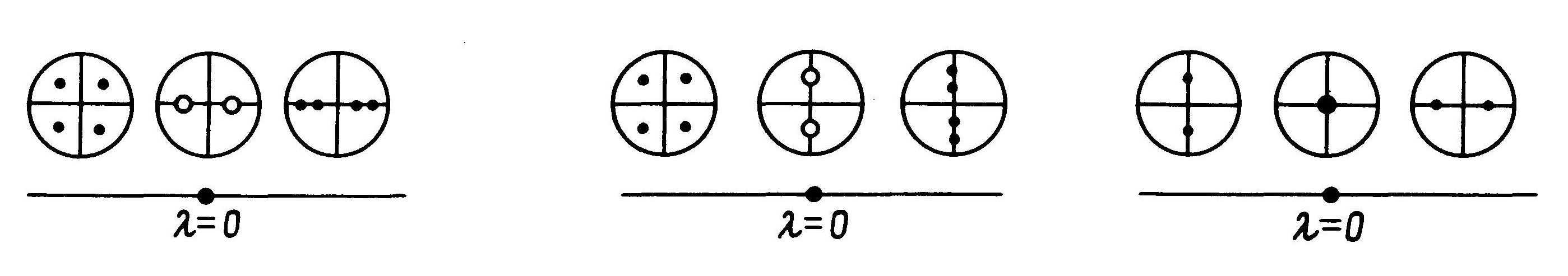}
	\vskip -0.5 cm
	\caption{Bifurcation diagrams of codimension 1 strata}
 \label{figure:bf_c1}
\end{figure}

In Figure \ref{figure:bf_c2}, the bifurcation diagrams of generic deformations for the classes of codimension $2$ are presented in the order of their listing
in the statement 
of Lemma \ref{bifurcation.digram.of.c=1.2}.
Here each origin represents the class of codimension 2 strata,
and the bold curves represent the classes of codimension 1 strata.

\begin{figure}[ht]
	\centering
	\includegraphics[height=11cm]{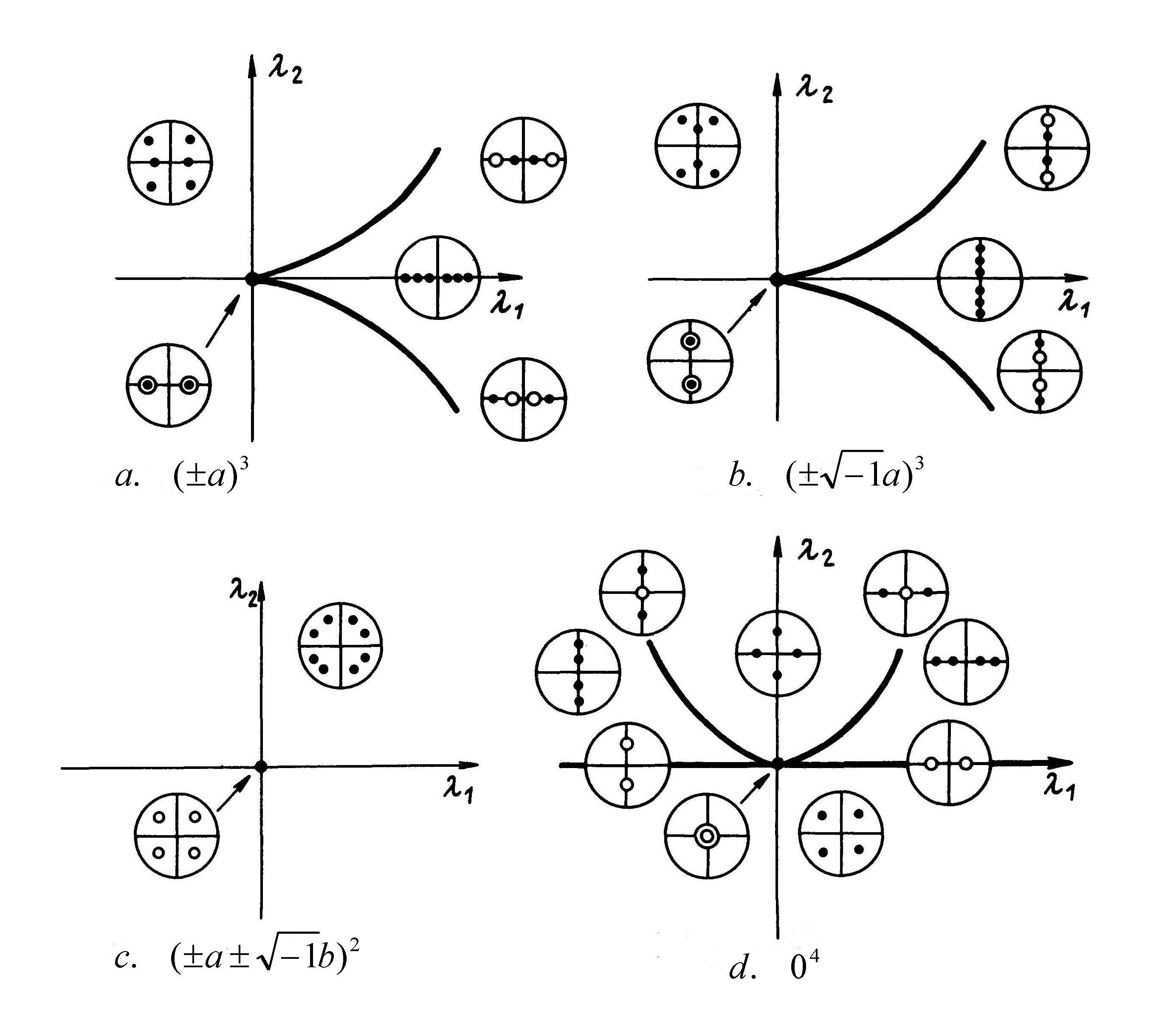}
	\vskip -0.5 cm
	\caption{Bifurcation diagrams of the first four classes of codimension 2 strata}
 \label{figure:bf_c2}
\end{figure}

Since the remaining five classes 
of codimension $2$
can be viewed as the direct product of two classes
of codimension $1$,
their bifurcation diagrams are the same, and can be represented by 
Figure \ref{figure:bf_c2_2}.
Here also the origin represents the class of codimension 2 strata,
and the  bold curves, except for the origin, represent the classes of codimension 1 strata.

\begin{figure}[ht]
	\centering
	\includegraphics[height=5cm]{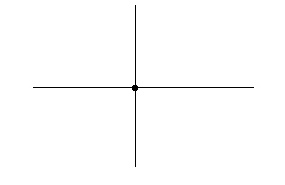}
	\vskip -0.5 cm
	\caption{Bifurcation diagrams of the last five classes of codimension 2 strata}
 \label{figure:bf_c2_2}
\end{figure}

Accordingly, in the parameter space of a family in general position,
every matrix in the codimension 2 strata can only be connected with
$0,2$ or $4$ curves which represent the matrices of codimension 1 strata.

\begin{lemma}\label{Lemma:cusp}
    The bifurcation diagram of general deformations for $N_3(\theta, b)$ (or $\tilde{N}_3(\theta, b)$) is a cusp of type $A_2: \lambda_2^2 \pm \lambda_1^3$ (c.f. \cite{Arn} and \cite{Ore}).
\end{lemma}

\begin{proof}
    Since ${\bf sp}(2n)$ is the Lie algebra of $\Sp(2n)$,
    we only need prove for ${\bf sp}(2n)$.
    The following matrix is given by Kocak (c.f. \cite{Koc}):
    \begin{equation}
        K(\beta;\lambda_1,\lambda_2)=\left(\matrix{
    	0 & 0 & 0            &|& -\beta  & 0 & -\lambda_2\cr
    	1 & 0 & \lambda_1    &|& 0 & -\beta  & 0\cr
    	0 & 1 & 0            &|& 0  & 0   & -\beta\cr
        -- & -- & --         &+& -- & -- & --\cr      
    	\beta & 0 &\lambda_2 &|& 0 & \lambda_1  & 0\cr
    	0 & \beta  &  0      &|&  1  & 0 & 0 \cr
    	0 & 0  & \beta       &|&  0  & 1 & 0 \cr
}\right)
    \end{equation}
    When $\lambda_1=\lambda_2=0$, $K(\beta;\lambda_1,\lambda_2)$
    is the normal form of the infinitesimal Symplectic matrix
    with triple eigenvalues $\pm\beta i$,
    and each eigenvalue has only one Jordan block.
    Moreovoer, $\exp(K(\beta;0,0))$ is similar to $N_3(\beta,b)$ or $\tilde{N}_3(\beta,b)$ for some $b$.

    The general deformation of $K(\beta;0,0)$ has two parameters
    $\lambda_1$ and $\lambda_2$.
    According to pp.143 of \cite{Koc},
    the versal deformation of the bifurcation diagram near $K(\beta;0,0)$ is characteried by 
    a polynomial equation $\Delta(\lambda_1,\lambda_2)=0$,
    where 
    \begin{eqnarray}
        \Delta(\lambda_1,\lambda_2)&=&27\lambda_2^4
        +(108\beta\lambda_1+432\beta^3)\lambda_2^3
        +(4\lambda_1^3+54\beta^2\lambda_1^2+864\beta^4\lambda_1+1728\beta^6)\lambda_2^2
        \nonumber\\
        &&
        +(16\beta\lambda_1^4-44\beta^3\lambda_1^3+432\beta^5\lambda_1^2)\lambda_2
        + 12\beta\lambda_1^5+11\beta^4\lambda_1^4+256\beta^6\lambda_1^3
        \nonumber\\
        &=&108\left[\Big(4\beta^3\lambda_2-{5\over6}\beta^2\lambda_1^2-\beta\lambda_1\lambda_2-{1\over2}\lambda_2^2-{1\over27}\lambda_1^3\Big)^2\right.
        \nonumber\\
        &&-\left.\Big(-{4\over3}\beta^2\lambda_1-2\beta\lambda_2+{1\over9}\lambda_1^2\Big)^3\right].
    \end{eqnarray}
    Here in the original computation of the second term of
    the discriminant on
    line 6, pp.~143, in the thesis \cite{Koc} of
    Kocak (1980), there is an error. It should be corrected to
    $(108\beta\lambda_1+432\beta^3)\lambda_2^3$.
    Thus $(0,0)$ is a cusp of the curve $\Delta(\lambda_1,\lambda_2)=0$
    with type $A_2: \lambda_2^2\pm \lambda_1^3$.
    Therefore, the  bifurcation diagram of general deformations for $N_3(\th,b)$
    (or $\tilde{N}_3(\th,b)$) is a cusp of the same type,
    see Figure \ref{figure:bf_c2} (b).
\end{proof}

\setcounter{equation}{0}
\section{Basic facts, definitions and some useful lemmas}
\label{sec:3}

\subsection{Basic facts of positive paths}



\begin{lemma}\label{Lemma:multiplication}
	Suppose $\{\gamma_1(t)\}, \{\gamma_2(t)\}, t \in [0,1]$ are two positive paths, then $\gamma_1 \gamma_2$ is a positive path.
\end{lemma}

\begin{proof}
Suppose
$$
{d\gamma_1\over dt}=JP_1(t)\gamma_1,\quad {d\gamma_2\over dt}=JP_2(t)\gamma_2
$$
where $P_i(t),i=1,2$ are positive definite for any $t\in[0,1]$.
Then we have
\begin{eqnarray}
{d(\gamma_1 \gamma_2)\over dt}&=&{d\gamma_1\over dt}\gamma_2
           +\gamma_1{d\gamma_2\over dt}  \nonumber\\
&=&JP_1(t)\gamma_1 \gamma_2+\gamma_1JP_2(t)\gamma_2  \nonumber\\
&=&J[P_1(t)+J^{-1}{\gamma_1}JP_2(t){\gamma_1}^{-1}]\gamma_1 \gamma_2  \nonumber\\
&=&J[P_1(t)+{\gamma_1}^{-T}P_2(t){\gamma_1}^{-1}]\gamma_1 \gamma_2.
\end{eqnarray}
Since $P_1(t)+{\gamma_1}^{-T}P_2(t){\gamma_1}^{-1}$ is positive definite for any $t\in[0,1]$,
thus $\gamma_1 \gamma_2$ is a positive path.
\end{proof}

\begin{remark}\label{rk:multiplication}
If $\gamma_1$ and $\gamma_2$ are non-negative paths, and at least one of them is positive, then $\gamma_1\gamma_2$ is also positive.
Specially, $\gamma_1 M_0$ and $M_0\gamma_1$ are positive paths for some constant symplectic matrix $M_0$.
\end{remark}

Recall 
$$
R(\theta)=\left(\matrix{\cos\theta&-\sin\theta\cr \sin\theta&\cos\theta}\right)
$$
and define $R_{2n}(\theta)={\rm diag}(\underbrace{R(\theta),\cdots,R(\theta)}_{n})$.
We have

\begin{lemma}\label{Prop:perturbation}
	If $\gamma$ is a positive path in $\Sp(2n)$, then there exists $\epsilon>0$ such that, for any $\th\in(-\epsilon,\epsilon)$,
	$\gamma R_{2n}(\th t)$ is also a positive path.
\end{lemma}

\begin{proof}
Suppose
\begin{equation}
{d\gamma\over dt}=JP_1(t)\gamma,
\end{equation}
where $P_1(t)$ are positive definite for any $t\in[0,1]$.
By ${d\over dt}R(\theta(t))=\dot{\theta}(t)J_2R(\theta(t))$,
we have ${d\over dt}R_{2n}(\theta(t))=\dot{\theta}(t)JR_{2n}(\theta(t))$.
Let now $\theta(t)=\theta t$. By Lemma \ref{Lemma:multiplication}, we have
\begin{equation}
	P=-J{d(\gamma(t)R_{2n}(\th t))\over dt}(\gamma(t)R_{2n}(\th t))^{-1}=P_1(t)+\th  \gamma(t)^{-T}\gamma(t)^{-1}.
\end{equation}
Since $P_1(t)$ and $\gamma(t)^{-T}\gamma(t)^{-1}$ are positive definite,
then for $\th$ small enough, $P$ ia also positive definite.
Therefore, $\{\gamma R_{2n}(\th t)\}$ is a family of positive paths 
for $\th\in(-\epsilon,\epsilon)$.
\end{proof}

We now give a simple proof of the following result which is stated in \cite{LaM}.

\begin{proposition}[\cite{LaM}, Proposition 1.1]\label{Prop:joint}
	 Any two elements in $\Sp(2n)$ may be jointed by a positive path.
\end{proposition}

\begin{proof}
Since $\Sp(2n)$ is path connected, there is a $C^1$ path $\{\gamma(t)\},t\in[0,1]$ which joint two given elements $M_0$ and $M_1$.
We define
\begin{equation}
\mu(t)=\gamma(t)R_{2n}(2k\pi t)
\end{equation}
for some $k\in\N$. Then $\{\mu(t)\},t\in[0,1]$ is also a $C^1$ path joint $M_0$ and $M_1$.

By ${d\over dt}R(\theta)=\dot{\theta}J_2R(\theta)$,
we have ${d\over dt}R_{2n}(\theta)=\dot{\theta}JR_{2n}(\theta)$,
and hence by Lemma \ref{Lemma:multiplication}, we have
\begin{eqnarray}
P=-J{d\mu\over dt}\mu^{-1}=
-J{d\gamma\over dt}\gamma^{-1}+2k\pi \gamma^{-T}\gamma^{-1}.
\end{eqnarray}
Since $\gamma^{-T}\gamma^{-1}$ is a positive definite matrix,
then for $k$ large enough, $P$ is also positive definite.
Therefore, $\mu$ is the required positive path.
\end{proof}

Although any two elements in $\Sp(2n)$ can be joint 
by a positive path, but the eigenvalues of such a path maybe have to rotate on $\U$.
Sometimes, we need to connect two elements in $\Sp(2n)$
within a proper subset of $\Sp(2n)$, such as
the truly hyperbolic set. 
Under such constraints, a positive path may not exist.
For example, we have

\begin{fact}\label{Prop:unconnecable.n=2}
    There does not exist a positive path which is entirely
    in the truly hyperbolic set of $\Sp(2)$ with endpoints $D({1\over2})$ and $D(2)$.
\end{fact}

Similar results also hold in higher dimensional space:
\begin{fact}\label{Prop:unconnecable}
    There exist two truly hyperbolic matrices
    such that connecting them by any positive truly
    hyperbolic path is impossible.
\end{fact}
The detailed proofs of Facts \ref{Prop:unconnecable.n=2} and \ref{Prop:unconnecable} will be deferred to Appendix \ref{proofs.facts}. We note that the proof of Fact \ref{Prop:unconnecable.n=2} is constructive.
However, the proof of Fact \ref{Prop:unconnecable} is difficult and relies on the main result of our paper, namely Theorem A.

\subsection{Generic homotopy and $C^1$-connectable homotopy}
Let $\{\alpha_t\}$ and $\{\beta_t\}$ (where $t\in[0,1]$) be two paths in $\mathrm{Sp}(2n)$. Recall that a continous map $H(s,t):[0,1]\times[0,1]\rightarrow\mathrm{Sp}(2n)$ is a homotopy from $\{\alpha_t\}$ to $\{\beta_t\}$ if $H(0,t)=\alpha_t$ and $H(1,t)=\beta_t$.
If $\{\alpha_t\}$ and $\{\beta_t\}$ are $C^1$-smooth, we typically assume that the homotopies in our paper are $C^0$ with respect to $s$ and $C^1$ with respect to $t$, unless we specifically emphasize. If $H(s,t)$ is such a homotopy, sometimes, as needed, we may suppose that it is also $C^1$ with respect to $s$ after perturbing it slightly (This is known as the Whitney Approximation Theorem, see, e.g., \cite[pp. 142, Theorem 6.29]{Lee}). If the paths $\{\alpha_t\}$ and $\{\beta_t\}$ are disjoint except at their endpoints\footnote{Sometimes, we have to consider the cases $\alpha_0=\beta_0=I$ or $\alpha_1=\beta_1=I$. In these cases, note that $I$ is a nonzero-codimention boundary point, but it is isolated in the cases that we consider. We can still obtain the same results.}, we can always view the homotopy, if necessary by perturbation, as an embedding  surface (with boundaries and possibly with cusps at the boundaries, such as when $H(s,0)=I$ or $H(s,1)=I$ for all $s$) in $\mathrm{Sp}(2n)$.

\begin{definition}\label{def:generic homotopy}
    A point $A$ in $\mathrm{Sp}(2n)$ is called a generic point if all of its eigenvalues have multiplicity 1. A path $\{\alpha_t\}$ (where $t\in[0,1]$, or $t\in(0,1]$, or $t\in[0,1)$, or $t\in(0,1)$) in $\mathrm{Sp}(2n)$ is called a generic path if all of its points are generic or lie on the codimension 1 boundary part of a generic region, and the codimension 1 boundary points are finite. If $\alpha_0=I$ (or $\alpha_1=I$, or $\alpha_0=\alpha_1=I$), $I$ is the only point that lies on the nonzero-codimension boundary near $0$ (resp., near $1$, near $0$ and $1$) with respect to $t$, and if the path $\{\alpha_t\}$ (where $t\in(0,1]$) (resp., $t\in[0,1)$, $t\in(0,1)$) is generic, we also call it generic. A homotopy $H(s,t)$ is called a generic homotopy if all points in $H([0,1]\times(0,1))$ only contain finite codimension 2 boundary points $\{P_j\}_{1\leq j\leq k}$ and no boundary points of codimention $\geq 3$\footnote{We note that the definition includes the cases where $H(s,0)=I$ or $H(s,1)=I$ for all $s$.}, and for every $s\in [0,1]$, except for finitely many $\{s_i\}_{i=1,2,\cdots,k'}\subset [0,1]$ with $k'\leq k$ and $\{H(s_i,t)\}$ passing  through some of the points $\{P_j\}_{1\leq j\leq k}$, $\{H(s,t)\}$ is a generic path.
\end{definition}

The existence of generic points and generic paths is ensured by the transversal regularity. Readers may refer to \cite{LaM,Sli}. In the following, we establish the existence of the generic homotopy in the context of general position.

\begin{lemma}\label{lem:generic homotopy}
    If a homotopy $H(s,t)$ satisfies that $\{H(0,t)\}$ and $\{H(1,t)\}$ are smooth generic paths, then there is a generic homotopy $\bar{H}(s,t)$ near $H(s,t)$ (i.e., $\bar{H}(s,t)$ is in a small neighborhood of $H(s,t)$) from $\{H(0,t)\}$ to $\{H(1,t)\}$.
\end{lemma}

\begin{proof}
  By the transversality theory of manifolds and Lemma \ref{bifurcation.digram.of.c=1.2}, we may suppose that $\tilde{H}(s,t)$ is in general position with respect to $H(s,t)$, meaning we perturb $H(s,t)$ slightly so that it transversely intersects the codimension 1 and 2 stratum, see Lemma \ref{bifurcation.digram.of.c=1.2}. $\tilde{H}(s,t)$ only contains finite 2-codimention boundary points, denoted as $P_i$ ($1\leq i\leq k$), and finitely many 1-submanifolds (with boundaries or not), denoted as $S_i$ ($1\leq i\leq m$), which arise from the intersections of codimension 1 strata with the 2-manifold $\tilde{H}(s,t)$. Here, the parameters $(s,t)\in[0,1]\times(0,1)$ if $H(s,0)=H(s,1)\equiv I$, $(s,t)\in[0,1]\times(0,1]$ if $H(s,0)\equiv I$, $(s,t)\in[0,1]\times[0,1)$ if $H(s,1)\equiv I$, and $(s,t)\in[0,1]\times[0,1]$ otherwise. If the 1-submanifold $S_i$ has boundaries, then the boundary point is either $P_j$ for some $j$ or a non-generic point of $\{H(0,t)\}$ or $\{H(1,t)\}$. 
  \begin{figure}[ht]
	\centering
	\includegraphics[height=6.5cm]{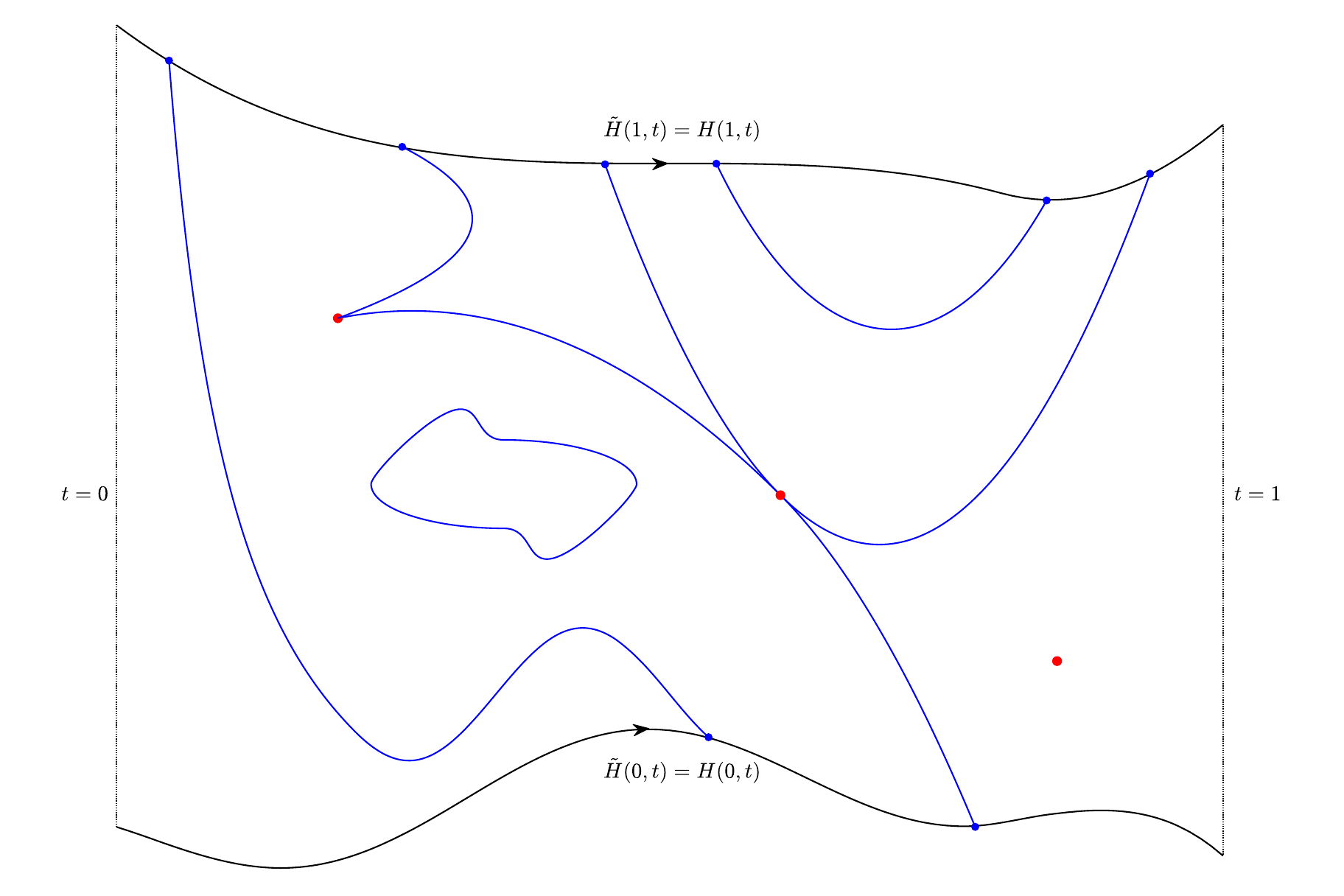}
	\vskip -0.0 cm
	\caption{The illustration of $\tilde{H}$}
 \label{figure:generic.homotopy}
\end{figure}
  
  Let us consider the parameters space $(s,t)\in[0,1]\times[0,1]$, and let $V(s,t)$ be the horizontal vector field on it. Denote $C=\cup_{i=1}^m\tilde{H}^{-1}\{S_i\}\subset [0,1]\times[0,1]$. The tangent lines of $S_i$ induced a vector field $W(s,t)$ on $C'=C\setminus \cup_{i=1}^k\tilde{H}^{-1}\{P_i\}$. Note that, by the definition of the generic path, $W(s,t)$ and $V(s,t)$ must be transversal to each other at the finite non-generic points of $\{H(0,t)\}$ and $\{H(1,t)\}$. It is easy to obtain a vector field $V'(s,t)$ on $[0,1]\times[0,1]$ by perturbing $V(s,t)$ such that $V'(s,t)$ and $W(s,t)$ are collinear at only finitely many points of $C'$.  After reparametrization, the integral curves of the new vector field $V'(s,t)$ define a smooth homeomorphism $v(s,t)$ of $[0,1]\times[0,1]$. Then the homotopy $\bar{H}=\tilde{H}\circ v$ is what we want.
\end{proof}

We will construct the positive homotopies piece by piece and then glue them together. Since the positive path is $C^1$ differentiable, the concatenation of two positive paths requires additional constraints.

Let	$\ga$ and $\mu$ be two positive paths.
	A {\bf positive homotopy with fixed end-points and end-speeds from $\ga$ to $\mu$} is a positive homotopy $H: [0,1]\times[0,1]\rightarrow\Sp(2n)$ satisfying the following conditions except $H(0,t)=\ga(t), H(1,t)=\mu(t)$:
	\begin{eqnarray}\label{fix.condition}
	H(s,0)&\equiv&\gamma(0),\quad 
	\partial_tH(s,0)\equiv\gamma'(0),
	\nonumber\\
	H(s,1)&\equiv&\gamma(1),\quad 
	\partial_tH(s,1)\equiv\gamma'(1).
	\end{eqnarray}
We denote this as $\ga \sim_{+, fix} \mu$. 


	Furthermore,
	a positive homotopy
	$H: [0,1]\times[0,1]\rightarrow\Sp(2n)$ from $\ga$ to $\mu$ is called {\bf positive homotopy with fixed end-segments from $\ga$ to $\mu$} if
	 there exists some $\epsilon>0$ such that
	 $H(\cdot,t)\equiv\ga(t)$ 
	 for any $t\in[0,\epsilon)\cup(1-\epsilon,1]$.
	 We denote it by $\ga\sim_{+,seg}\mu$.
It is evident that $\gamma \sim_{+, seg} \mu$ implies $\gamma \sim_{+, fix} \mu$. Futhermore, we can not answer the following question

\begin{question}
    Given two positive paths entirely within the hyperbolic set with the same endpoints (which are certainly homotopic), does there exist a positive homotopy
    with fixed end-points between them?
\end{question}
If we relax the condition of a positive homotopy from fixed endpoints to allowing the endpoints to lie in conjugacy classes, then the answer to the question above is positive (see Theorem \ref{Thm:remove.hyperbolic.collisions} in Section 4). In fact, in the intermediate steps of our constructive proof, we only require the endpoints to lie in the same conjugacy classes. 
It is a reason why we define the following $C^1$-connectable positive homotopy instead of the positive homotopy with fixed end-points and end-speeds.


\begin{definition}\label{pos.connectable.homotopy}
	Let $\ga$ and $\mu$ be two (positive) paths.
	A {\bf $C^1$-connectable (positive) homotopy} is a (positive) homotopy
	$H: [0,1]\times[0,1]\rightarrow\Sp(2n)$
    from $\ga$ to $\mu$ that satisfies the following conditions: there exist continuous matrices $X_s, Y_s \in \Sp(2n)$ for $s \in [0,1]$, with $X_0 = Y_0 = I$, such that:
	\begin{eqnarray}\label{C1.connectable.homotopy}
	H(s,0)&=&X_s^{-1}\gamma(0)X_s,\quad 
	\partial_tH(s,0)=X_s^{-1}\gamma'(0)X_s,
	\nonumber\\
	H(s,1)&=&Y_s^{-1}\gamma(1)Y_s,\quad 
	\partial_tH(s,1)=Y_s^{-1}\gamma'(1)Y_s.
	\end{eqnarray}
	We denote it by $\ga\sim_{con}\mu$ ($\ga\sim_{+,con}\mu$).
\end{definition}
Certainly, a necessary condition of $\ga\sim_{con}\mu$ is given by
\begin{equation}\label{C1.condition}
\left\{
\begin{array}{ll}
    X_1^{-1}\gamma(0)X_1=\mu(0),
    &X_1^{-1}\gamma'(0)X_1=\mu'(0);
    \\
    Y_1^{-1}\gamma(1)Y_1=\mu(1),
    &Y_1^{-1}\gamma'(1)Y_1=\mu'(1),
\end{array}
\right.
\end{equation}
for some matrices $X_1,Y_1\in\Sp(2n)$.
We refer to (\ref{C1.condition}) as the {\bf $C^1$-connectable conditions}.

Let $X_s\equiv Y_s\equiv I$ in (\ref{C1.connectable.homotopy}). Then it is easy to see that $\ga\sim_{+,fix}\mu$ implies $\ga\sim_{+,con}\mu$.) For positive path $\gamma$ and $\gamma_i=Z_i^{-1}\gamma Z_i$ where $Z_i\in \Sp(2n)\; (i=0,1)$,
the positive homotopy $H(s,t)=Z(s)^{-1}\gamma(t)Z(s)$ is
$C^1$-connectable, where $Z(s),s\in[0,1]$ is a continuous path with $Z(0)=Z_0$ and $Z(1)=Z_1$.\bigskip

From Lemma 2.3 in \cite{LaM}, we have the following result
\begin{lemma}
Given two  positive homotopies $H_1(s,t)$ ($t_0\leq t\leq t_1$) and $H_2(s,t)$ ($t_1\leq t\leq t_2$) where $0\leq t_0<t_1<t_2\leq 1$,  which satisfy that 
\begin{itemize}
\item $H_1(s,t)$  (resp. $ H_2(s,t)$) is $C^1$ w.r.t. to $t$ and right (resp. left) continuous at $t_1$;
\item  if we write $\frac{\partial H_1(s,t)}{\partial t}|_{t=t_1^+}=JA_sH_1(s,t_1)$, then there is $X_s\in \mathrm{Sp}(2n)$ such that $H_2(s,t_1)=X_s^{-1}H_1(s,t_1)X_s$ and $\frac{\partial H_2(s,t)}{\partial t}|_{t=t_1^-}=J(X_s^TA_sX_s)H_2(s,t_1)$ for every $s\in[0,1]$.
\end{itemize}
Let $\bar{H}(s,t)=X_s^{-1}H_1(s,t)X_s$ when $t_0\leq t\leq t_1$ and $\bar{H}(s,t)=H_2(s,t)$ when $t_1\leq t\leq t_2$. Then $\bar{H}(s,t)$ ($t_0\leq t\leq t_2$) is a  positive homotopy.
\end{lemma}

If we cut $\ga$ and $\mu$ into several smaller paths
$\ga_1,\ga_2,\ldots,\ga_n$ and $\mu_1,\mu_2,\cdots,\mu_n$, respectively,
by time intervals $0=t_0<t_1<\cdots<t_{n-1}<t_n=1$.
Moreover, we suppose that $\ga_i\sim_{+,con}\mu_i$ for $1\le i\le n$,
i.e., there exits $C^1$-connectable positive homotopies
$H_i: [0,1]\times[t_{i-1},t_i]\rightarrow\Sp(2n)$
which satisfying
\begin{eqnarray}
	H_i(s,t_{i-1})&=&(X_s^{(i)})^{-1}\gamma(t_{i-1})X_s^{(i)},\quad 
	\partial_tH(s,t_{i-1})=(X_s^{(i)})^{-1}\gamma'(t_{i-1})X_s^{(i)},
	\nonumber\\
	H_i(s,t_i)&=&(Y_s^{(i)})^{-1}\gamma(t_i)Y_s^{(i)},\qquad\quad
	\partial_tH(s,t_i)=(Y_s^{(i)})^{-1}\gamma'(t_i)Y_s^{(i)},
\end{eqnarray}
for $i=1,2,\cdots,n$.
We denote $Y_s^0=I$. Then we can glue them together:
$$H(s,t)=[\Pi_{j=1}^i(Y_s^{(j-1)})^{-1}X_s^{(j)}]H_i(s,t)[\Pi_{j=1}^i(Y_s^{(j-1)})^{-1}X_s^{(j)}]^{-1}
    \quad \mathrm{if}\;\, t\in[t_{i-1},t_i].$$
Now $H$ is  a positive homotopy, and hence $\ga\sim_{+,con}\mu$.

\subsection{Smoothing Lemma and the connectedness of the set of
the positive paths in the same conjugacy class}

We have the following simple but useful formula. 

\begin{lemma}\label{lem:conj}
Suppose that $\{\gamma_t\}_{t\in [a,b]}$, $\{X_t\}_{t\in [a,b]}\subset \mathrm{Sp}(2n)$ satisfy $\frac{d }{dt}\gamma_t=JP_t\gamma_t$ and $\frac{d }{dt}X_t=JY_tX_t$ where $P_t$ is a positive definite matrix and $Y_t$ is a real symmetry matrix. We have
$$\frac{d}{dt}(X_t^{-1}\gamma_tX_t)=J[X_t^T(P_t-Y_t+(\gamma_t^{-1})^TY_t\gamma_t^{-1})X_t](X^{-1}_t\gamma_tX_t). $$ In particular, when  $X_{t_0}=I$ where $t_0\in [a,b]$, the equation becomes $$\frac{d}{dt}(X_t^{-1}\gamma_tX_t)|_{t=t_0}=J[(P_{t_0}-Y_{t_0}+(\gamma_{t_0}^{-1})^TY_{t_0}\gamma_{t_0}^{-1})]\gamma_{t_0}.$$ 
\end{lemma}

\begin{proof}

\begin{eqnarray}\label{eq:conj}
\frac{d}{dt}(X_t^{-1}\gamma_tX_t)&=&-X^{-1}_{t}\frac{d X_t}{dt}X^{-1}_{t}\gamma_{t}X_{t}+X_{t}^{-1}\gamma_{t}\frac{d X_t}{dt}+X^{-1}_{t}\frac{d \gamma_t}{dt}X_{t}\nonumber\\
&=&-X^{-1}_{t}JY_{t}X_{t}(X^{-1}_{t}\gamma_{t}X_{t})+X_{t}^{-1}\gamma_{t}JY_{t}X_{t}+X^{-1}_{t}JP_{t}\gamma_{t}X_{t}\nonumber\\
&=&-JX^{T}_{}Y_{t}X_{t}(X^{-1}_{t}\gamma_{t}X_{t})+JX_{t}^{T}(\gamma_{t}^{-1})^TY_{t}\gamma_{t}^{-1}X_{t}(X^{-1}_{t}\gamma_{t}X_{t})\nonumber\\&&\qquad\qquad +JX^{T}_{t}P_{t}X_{t}(X^{-1}_{t}\gamma_{t}X_{t})\nonumber\\
&=&J[X_{t}^T(P_{t}-Y_{t}+(\gamma_{t}^{-1})^TY_{t}\gamma_{t}^{-1})X_{t}](X^{-1}_{t}\gamma_{t}X_{t}). 
\end{eqnarray}

\end{proof}

As a corollary, we immediately obtain the following result, although we already know it through other method, e.g. \cite{CLW}.

\begin{lemma}
Suppose that $\{X_t\}_{t\in [a,b]}\subset\mathrm{Sp}(2n)$ is $C^1$ w.r.t. $t$ and $A\in\Sp(2n)$. If $\mathrm{Spec}(A)\cap \mathcal{U}\neq \emptyset$, then the path $\{X_t^{-1}AX_t\}$ is not a positive path.
\end{lemma}
\begin{proof}
From Formula (\ref{eq:conj}), we have
\begin{eqnarray*}
\frac{d}{dt}(X_t^{-1}AX_t)=J[X_t^T(-Y_t+(A^{-1})^TY_tA^{-1})X_t](X^{-1}_tAX_t).
\end{eqnarray*}
Obviously, $X_t^T(-Y_t+(A^{-1})^TY_tA^{-1})X_t$ is a real symmetry matrix. If $Av=\lambda v$ for some $\lambda\in\mathbf{C}$ with $|\lambda|=1$ and $0\neq v\in \mathbf{C}^{2n}$, we have that
$$\langle (-Y_t+(A^{-1})^TY_tA^{-1})v,v\rangle
=-\langle Y_t v, v\rangle +\langle (A^{-1})^TY_tA^{-1})v, v\rangle=(|\lambda|^2-1)\langle Y_t v,v\rangle=0.$$ 
\end{proof}

Recall $\pi:\Sp(2n)\to\mathcal{C}onj(\Sp(2n)$ denote
the projection:
$$
    \pi(A)=\bigcup_{X}
    \Big\{
    XAX^{-1}:\;X\in\Sp(2n)
    \Big\}\in\mathcal{C}onj(\Sp(2n)).
$$
The following lemma, as stated by Lalonde and McDuff in \cite{LaM}, presents a general result on positive paths.
It is noteworthy that we require a modified version of (iii) of this lemma, denoted by (iii$'$).

\begin{lemma}[\cite{LaM}, Lemma 2.1]\label{lem:c0c1}\quad
\begin{description}
\item[(i)] The set of positive paths is open in the $C^1$ topology.
\item[(ii)] Any piecewise positive path may be $C^0$ approximated by a positive path.
\item[(iii)] Let $\{\alpha_t\}_{t\in[0,1]},\{\beta_t\}_{t\in[0,1]}$ be two positive paths with the same initial point
$\alpha_0 =\beta_0$. Then, given $\epsilon>0$, there is $\delta>0$ and a $C^0$-small deformation of $\{\alpha_t\}$ to a positive path $\{\bar{\alpha}_t\}$ which coincides with $\{\beta_t\}$ for $t\in[0,\delta]$ and with $\{\alpha_t\}$ for $t>\epsilon$.
\item[(iii$'$)] Let $\{\alpha_t\}_{t\in[0,1]},\{\beta_t\}_{t\in[0,1]}$ be two positive paths with the same initial point $\alpha_0 =\beta_0$ and $\pi(\alpha_t)=\pi(\beta_t)$ for all $t$.  Then, given $\epsilon>0$, there is $\delta>0$ and a $C^0$-small deformation of $\{\alpha_t\}$ to a positive path $\{\bar{\alpha}_t\}$ which coincides with $\{\beta_t\}$ for $t\in[0,\delta]$ and with $\{\alpha_t\}$ for $t>\epsilon$, and satisfies that $\pi(\alpha_t)=\pi(\bar{\alpha}_t)$ for all $t$. That is the deformation occurs within the fibers $\cup_{t\in[0,1]}\pi^{-1}(\pi(\alpha_t))$.
\end{description}
\end{lemma}




Now, we provide further details of the proof of Proposition 2.4 in \cite{LaM}, which is also crucial for the proof of Lemma \ref{lem: smoothing} below since the idea of its proof roots in the proof of Proposition 2.4. In \cite{LaM}, to streamline the proof of Proposition 2.4, they utilized Lemma \ref{lem:c0c1}(iii). However, Lemma \ref{lem:c0c1}(iii) represents a $C^0$-small deformation, and through their proof, the deformation may not necessarily remain within the conjugacy class. To establish Proposition 2.4, it is necessary to adjust Lemma \ref{lem:c0c1}(iii) so that the deformation remains within the fibers, i.e., Lemma \ref{lem:c0c1}(iii$'$). Refer to the end of this subsection for the proof of (iii$'$). 

Let $\|\cdot\|$ be a norm on the space of matrices (e.g., $\|A-B\|=\max_{i,j}|a_{i,j}-b_{i,j}|$ if $A=(a_{i,j})$ and $B=(b_{i,j})$). Let $G=\mathrm{Sp}(2n)$. For every $A\in G$, recall that $\mathcal{P}_A\subset T_AG$ is the positive cone on $A$: $\mathcal{P}_A=\{JPA \mid P=P^T, P>0\}$.

\begin{proposition}[\cite{LaM}, Proposition 2.4]\label{prop:lift}
Let $\alpha=\{\alpha_t\}\subset \mathrm{Sp}(2n)$ be a generic positive path joining two generic points $\alpha_0$ and $\alpha_1$\footnote{It is not necessary that $\alpha_0$ or $\alpha_1$ is non-generic; in particular, it is posible that $\alpha_0=\beta_0=I$ or $\alpha_1=\beta_1=I$.}. Then the set of positive paths in $\mathrm{Sp}(2n)$ that lift $a=\{\pi(\alpha_t)\}\subset \textrm{Conj}(\Sp(2n))$ is path-connected.
Moreover, if two smooth lifts $\alpha$ and
$\beta$ of $a$ satisfy the $C^1$-connectable condition, then $\alpha\sim_{+,con}\beta$.
\end{proposition}

\begin{proof}
We follow the argument of the proof of Proposition 2.4 in \cite{LaM} and initiate the deformation within the fibers by utilizing our modified Lemma \ref{lem:c0c1}(iii$'$). 

 Suppose $\{\beta_t\}$ is another path which is a lift of $a$ and  satisfies $\pi(\alpha_t)=\pi(\beta_t)$ for every $t$. (In fact, it is not necessary that $\{\beta_t\}$ is disjoint from $\{\alpha_t\}$, as shown in the proof below, although it is required in the proof in \cite{LaM}.) Recall the finite sequence of times $0=t_0<t_1<t_2<\cdots<t_k<t_{k+1}=1$ where $\{t_i\}_{i=1,\cdots,k}$ are the non-generic points of $\{\alpha_t\}$ at which $\{\alpha_t\}$ crosses a codimension 1 stratum. Let $Y:=\pi^{-1}(a)$ and $Y_i:=\{\pi^{-1}(\pi(\alpha_t)):t_i\leq t\leq t_{i+1}\}$. Recall the vector fields $\xi_\alpha$ and $\xi_\beta$ defined on $Y_i$ in the proof of \cite[Proposition 2.4]{LaM} (using the notations $\xi_A$ and $\xi_B$ there).  We recall that $\xi_\alpha$ (resp. $\xi_\beta$) has the form $JP_\alpha(q)q$ (resp. $JP_\beta(q)q$) where $q\in Y$, $P_\alpha(q)$ (resp. $P_\beta(q)$) $\in \mathcal{P}_q$ and $\pi_*(JP_\alpha(q)q)=\pi_*(\frac{d \alpha_t}{dt})$ (resp. $\pi_*(JP_\beta(q)q)=\pi_*(\frac{d \beta_t}{dt})$) when $q\in\pi^{-1}\{\pi(\alpha_t)\}$.
In the proof of \cite[Proposition 2.4]{LaM}, we note that we should require that $\xi_\alpha$ and $\xi_\beta$ are uniformly bounded under the norm $\|\cdot\|$ in the following sense: there exists $N>0$ such that $\|P_\alpha(q)\|\leq N$ and $\|P_\beta(q)\|\leq N$ for all $q\in Y$. Otherwise, the vector fields $\xi_\alpha$ and $\xi_\beta$ are not necessary integrable since $Y$ is not compact (the integral curves may go to $\infty$ along the fibers). 


Fix some $0\leq i\leq k$. Let $\{V_s(q)\}\subset T_{q}Y_i$ be a family of vector fields of the form $(1-s)\xi_\alpha+s\xi_\beta$, where the parameters $s\in[0,1]$ and $q\in Y_i$. Obviously, $V_s(q)$ is smooth w.r.t. $s$ and $q$. For any $q\in Y_i$, we write $T_qG:=T_q\pi^{-1}(q)\oplus (T_q\pi^{-1}(q))^{\bot}$. By the argument of the proof in \cite[Proposition 2.4]{LaM}, $\frac{d \alpha_t}{dt}\not\in T_{\alpha_t}\pi^{-1}(\alpha_t)$ and $\frac{d \beta_t}{dt}\not\in T_{\beta_t}\pi^{-1}(\beta_t)$. Then the vector field $V_s$ is integrable for every $s$. Let $\Gamma^i_s(t)$ be the integral curve of $V_s$ (on $Y_i$) defined by $\frac{d \Gamma^i_s(t)}{dt}=V_s(\Gamma^i_s(t))$. Let $\{\delta_i(s)\}_{s\in [0,1]}$ be a path from $\alpha_{t_i}$ to $\beta_{t_i}$ (recall that the fiber $\pi^{-1}\{\pi(\alpha_{t_i})\}$ is path connected). Note that $V_0(\alpha_t)=\frac{d \alpha_t}{dt}$ and $V_1(\beta_t)=\frac{d \beta_t}{dt}$. By the dependence of the smoothness of solutions of ODE on the initial value and parameters, the family of solutions $\{\Gamma^i_s(t)\}_{s\in[0,1],t_i\leq t\leq t_{i+1}}$ of the equation
\[
\left\{
    \begin{array}{l}
\frac{d \Gamma^i_s(t)}{dt}=V_s(\Gamma^i_s(t))\\ \Gamma^i_s(t_i)=\delta_i(s) 
\end{array}\right.
\]
gives a homotopy from $\{\alpha_t\}_{t_i\leq t\leq t_{i+1}}$ to $\{\beta_t\}_{t_i\leq t\leq t_{i+1}}$.

We will complete the proof by induction. We start with $Y_0$ and the path $\{\delta_0(s)\}_{s\in [0,1]}$, where $\{\delta_0(s)\}_{s\in [0,1]}$ is a path from $\alpha_{0}$ to $\beta_{0}$, and we set $\delta_0(s)\equiv I$ if $\alpha_0=\beta_0=I$. Solving the ODE above, we obtain a homotopy $\{\Gamma^0_s(t)\}_{s\in[0,1],0\leq t\leq t_{1}}$ from $\{\alpha_t\}_{0\leq t\leq t_{1}}$ to $\{\beta_t\}_{0\leq t\leq t_{1}}$. Next, we consider $Y_1$ and $\delta_1(s):=\Gamma^0_s(t_1)$ to obtain $\{\Gamma^1_s(t)\}_{s\in[0,1],t_1\leq t\leq t_{2}}$. Assuming we have completed the $(i-1)$-th step, we let $\delta_{i}(s):=\Gamma^{i-1}_s(t_{i})$. We then consider $Y_{i}$ and $\delta_{i}(s)$ to obtain $\{\Gamma^i_s(t)\}_{s\in[0,1],t_i\leq t\leq t_{i+1}}$. By induction, we obtain $(k+1)$ homotopies $\{\Gamma^i_s(t)\}_{s\in[0,1],t_i\leq t\leq t_{i+1}}$ ($i=0,\cdots,k$). Finally, we construct our homotopy by concatenating them together.

Furthermore, if $\alpha$ and
$\beta$ satisfy the $C^1$-connectable condition, that is, there are $X_1,Y_1\in \Sp(2n)$ such that 
$$
\left\{
\begin{array}{ll}
    X_1^{-1}\alpha(0)X_1=\beta(0),
    &X_1^{-1}\alpha'(0)X_1=\beta'(0),
    \\
    Y_1^{-1}\alpha(1)Y_1=\beta(1),
    &Y_1^{-1}\alpha'(1)Y_1=\beta'(1).
\end{array}
\right.
$$
For every $X\in \Sp(2n)$, we may choose $\xi_\alpha$ and $\xi_\beta$ satisfy that
$\xi_\alpha|_{X^{-1}\alpha(0)X}=\xi_\beta|_{X^{-1}\alpha(0)X}=X^{-1}\alpha'(0)X$, $\xi_\alpha|_{X^{-1}\alpha(1)X}=\xi_\beta|_{X^{-1}\alpha(1)X}=X^{-1}\alpha'(1)X$.
So, according to the process of the proof, we obtain $\alpha\sim_{+,con}\beta$.

\end{proof}

Suppose $a:[0,1]\rightarrow \mathcal{C}onj(\Sp(2n))$ is a path. We define $a(t_0)$, where $t_0\in (0,1)$, as a smooth point of $a$ if there exists a lift path $\gamma: [0,1]\rightarrow \mathrm{Sp}(2n)$ of $a$ (i.e., $\pi(\gamma(t))=a(t)$ for all $t\in[0,1]$) such that there exists $\epsilon>0$ satisfying the condition that the sub-path $\gamma|_{[t_0-\epsilon, t_0+\epsilon]}$ is smooth in the space $\mathrm{Sp}(2n)$. Similarly, we can define the smoothness of $a(0)$ and $a(1)$. A path $a$ in $\mathcal{C}onj(\mathrm{Sp}(2n))$ is said to be smooth if every point of $a$ is smooth.  We have the following smoothing lemma:

\begin{lemma} [Smoothing Lemma]\label{lem: smoothing}
Let $\gamma:[0,1]\rightarrow \mathrm{Sp}(2n)$ be a 2-piecewise positive path which is not $C^1$ at the time $t_0\in(0,1)$, i.e., $\lim_{t\rightarrow t_0^-}\frac{\gamma(t)-\gamma(t_0)}{t-t_0}\not=\lim_{t\rightarrow t_0^+}\frac{\gamma(t)-\gamma(t_0)}{t-t_0}$. If $\gamma(t_0)$ is a generic point and $\pi(\gamma(t_0))$ is a smooth point of $\pi(\gamma)$, then for any $\epsilon>0$, there is a $C^1$-smooth positive path $\bar{\gamma}$ from $\gamma(0)$ to $\gamma(1)$ such that $\pi(\gamma(t))=\pi(\bar{\gamma}(t))$ for all $t\in [0,1]$ and $d_{C^0}(\gamma,\bar{\gamma})<\epsilon$. 
\end{lemma}
\begin{proof}

Since the problem is local, without loss of generality, we may suppose that there is $\delta>0$ small enough such that $[t_0-\delta,t_0+\delta]\subset (0,1)$ and $\gamma(t)$ is generic for all $t\in[t_0-\delta,t_0+\delta]$. We extend the positive path $\gamma|_{[0,t_0]}$ (resp., $\gamma|_{t\in[t_0,1]}$) a little bit to a $C^1$-smooth positive path $\alpha: [0,t_0+1/j_0]\rightarrow \mathrm{Sp}(2n)$ (resp., $\beta: [t_0-1/j_0,1]\rightarrow \mathrm{Sp}(2n)$) for some $j_0>2/\delta$ which satisfies that $\alpha(t)=\gamma(t)$ when $t\in[0,t_0]$ and $\pi(\alpha(t))=\pi(\gamma(t))$ when $t\in[t_0,t_0+1/j_0]$ (resp., $\beta(t)=\gamma(t)$ when $t\in[t_0,1]$ and $\pi(\beta(t))=\pi(\gamma(t))$ when $t\in[t_0-1/j_0,t_0]$). It is possible by the definition of smoothness and Lemma \ref{lem:c0c1}(i). Furthermore, we can also request that $\|\frac{d \alpha(t)}{dt}\|,\|\frac{d \beta(t)}{dt}\|<C$ for some constant $C>0$, where $C$ only depends on $\gamma$.

\begin{figure}[ht]
	\centering
	\includegraphics[height=7.5cm]{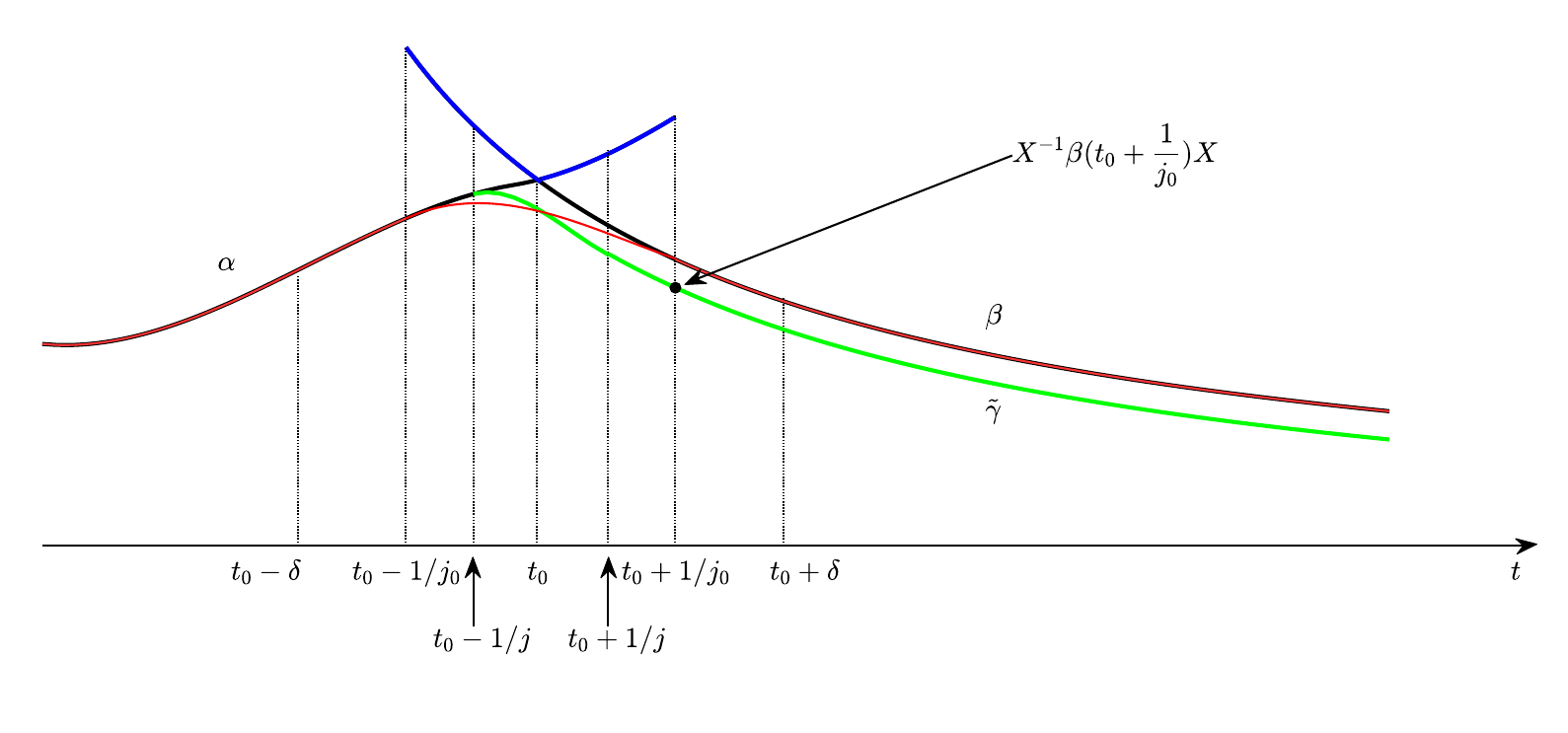}
	\vskip -1 cm
	\caption{Smoothing lemma.}
    \label{psi0}
\end{figure}

Let $a=\pi(\{\gamma(t)\}_{t\in [t_0-1/j_0, t_0+1/j_0]})$. We consider the submanifold $Y=\pi^{-1}(a)=\{\pi^{-1}(\pi(\gamma(t))): t_0-1/j_0\leq t\leq t_0+1/j_0\}$. Let $\xi_{\alpha}$ and $\xi_{\beta}$ be the vector fields on $Y$ which are sections of the intersection of the positive cone field with $TY$ that project to the tangent vector field of $a$, and such that $\{\alpha(t)\}_{t\in  [t_0-1/j_0,t_0+1/j_0]}$ and $\{\beta(t)\}_{t\in  [t_0-1/j_0,t_0+1/j_0]}$ are respectively integral curves of  $\xi_{\alpha}$ and $\xi_{\beta}$. This is possible because this intersection forms an open convex cone, as explained in the proof of \cite[Prop. 2.4]{LaM}.

We consider the following family of vector fields $\{\xi_j\}$ on $Y$: For $j\in\mathbf{N}$ with $j>j_0$, $t\in[t_0-1/j,t_0+1/j]$ and $E\in\pi^{-1}\{\gamma(t)\}$, we let $\xi_j|_{E}=(1-\frac{t-(t_0-1/j)}{2/j})\xi_\alpha|_{E}+(\frac{t-(t_0-1/j)}{2/j})\xi_{\beta}|_{E}$. For $t\in [t_0-1/j_0, t_0-1/j]$ and $E\in\pi^{-1}\{\gamma(t)\}$, let $\xi_j|_{E}=\xi_\alpha|_{E}$. For $t\in [t_0+1/j, t_0+1/j_0]$ and $E\in\pi^{-1}\{\gamma(t)\}$, let $\xi_j|_{E}=\xi_{\beta}|_{E}$. Obviously, $\{\xi_j\}$ are continuous for every $j>j_0$. 

For the given $\epsilon$, we can assume that the $C^0$-distance between the integral curve $\{\tau(t)\}$ (where $t\in[t_0-1/j_0,t_0+1/j_0]$) of $\xi_j$ starting at $\gamma(t_0-1/j_0)$ and ending at $\gamma(t_0+1/j_0)$, and the sub-path $\gamma|_{[t_0-1/j_0,t_0+1/j_0]}$, is less than $\epsilon/2$ when $j$ is sufficiently large.. 

We now suppose that $\tau(t_0+1/j_0)=X^{-1}\beta(t_0+1/j_0)X$ for some $X\in\mathrm{Sp}(2n)$ which is close to $I$. Let $\tilde{\gamma}$ be the concatenation of the path $\gamma|_{[0,t_0-1/j_0]}$, the integral curve $\tau$, and the path $X^{-1}\beta|_{[t_0+1/j_0,1]}X$. In a small neighborhood of $I$, we choose a smooth path $X:[0,1]\rightarrow \mathrm{Sp}(2n)$ from $I$ to $X$ that satisfies that $X(t)\equiv I$ when $t\in [0, t_0-1/j_0]$ and $X(t)\equiv X$ when $t\in[t_0+1/j_0,1]$, and $\|X'(t)\|$ is small enough (it is possible when we let $j$ be large enough).  By Formula (\ref{eq:conj}), the path $\bar{\gamma}(t)=X(t)^{-1}\tilde{\gamma}(t)X(t)$ is the positive path that we want.
\end{proof}

Similar to the proof of Lemma \ref{lem: smoothing}, we provide the proof of Lemma \ref{lem:c0c1}(iii$'$).

\begin{proof}[Proof of Lemma \ref{lem:c0c1}(iii$'$)]
Step 1. For any $0<\epsilon<1$, by (\ref{eq:conj}), it is easy to prove that there exist $0<\delta'<<1$ and a $C^0$-small deformation of $\{\alpha_t\}$ to a 2-piecewise positive path $\{\tilde{\alpha}_t\}$ which coinsides with $\{\beta_t\}$ for $t\in[0,\delta']$ and with $\{\alpha_t\}$ for $t>\epsilon$ (Obviously, $t=\delta'$ is the only cusp of $\tilde{\alpha}_t$). More precisely, we suppose that $\frac{d\alpha_t}{dt}=JP_t\alpha_t$. We can take $\delta'$ small enough (compared to $\epsilon$) so that $X$ is close to $I$ where $X$ satisfies that $\beta(\delta')=X^{-1}\alpha(\delta')X$. In a neighborhood of $I$ we choose a smooth path $\{X_t\}_{\delta'\leq t\leq 1}$ from $X$ to $I$. Suppose that $\frac{d X_t}{dt}=JY_tX_t$. We consider the path $\{X_t^{-1}\alpha_t X_t\}$. Decreasing $\delta'$ if necessary, we may further require that $\{X_t\}$ in (\ref{eq:conj}) (here we let $\gamma_t=\alpha_t$) satisfies that $\|Y_t\|<<\|P_t\|$ for $t\in[\delta', \epsilon]$ and $X_t=I$ when $t\geq \epsilon$. Let $\tilde{\alpha}$ be the concatenation of the path $\beta|_{[0,\delta']}$ and the path $\{X_t^{-1}\alpha_tX_t\}_{\delta'\leq t\leq1}$.\\

\begin{figure}[ht]
	\centering
	\includegraphics[height=6.0cm]{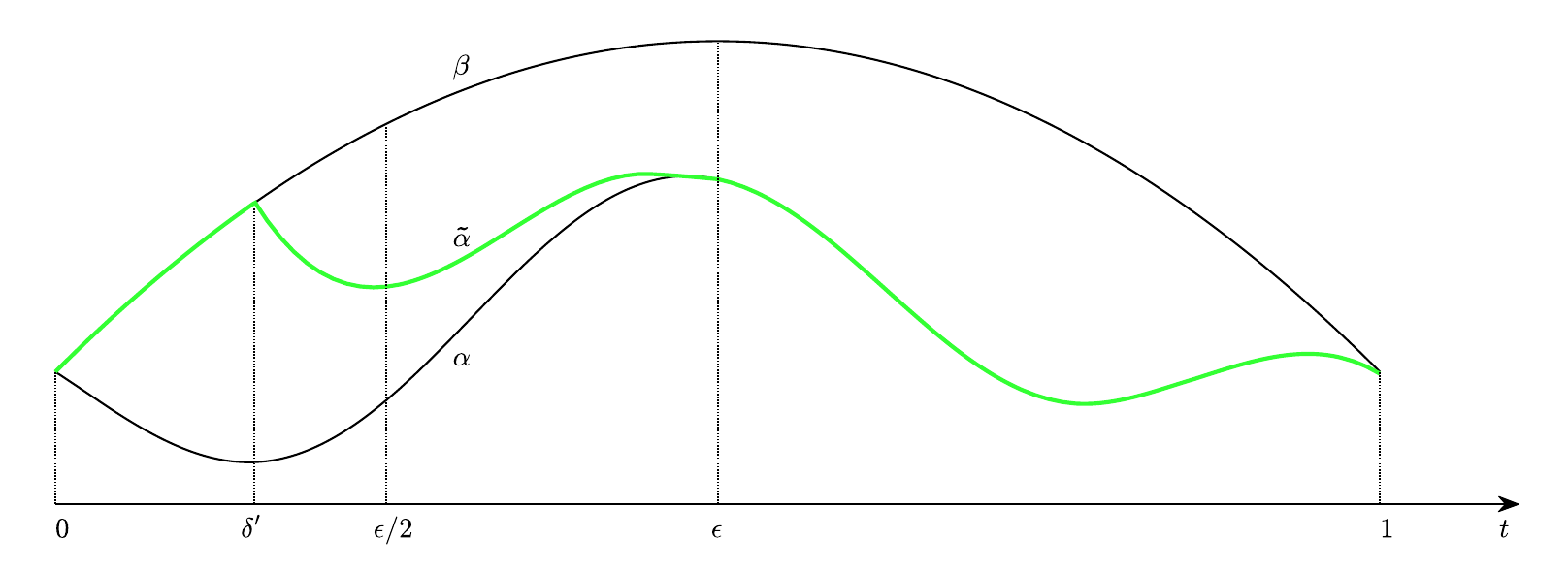}
	\vskip -0.5 cm
	\caption{Step 1}
    \label{step1}
\end{figure}

Step 2. (Smoothing) Fix $j_0> 2/\delta'$. We choose a $C^1$-positive path $\{\hat{\alpha}_t\}_{t\in[\delta'-1/j_0, 1]}$ which satisfies that $\hat{\alpha}_t=\tilde{\alpha}_t$ for $t\in[\delta', 1]$, $\pi(\hat{\alpha}_t)=\pi(\alpha_t)$ for $t\in[\delta'-1/j_0,\delta']$, and $\|\frac{d \hat{\alpha}_t}{dt}\|<C$ for some constant $C$, where $C$ only depends on $\{\alpha_t\}$ and $\{\beta_t\}$. Since the problem is local and the choice of $\delta'$ is arbitrary, w.l.o.g., we may suppose that $\{\alpha_t\}_{t\in [\delta'-1/j_0, \frac{1}{2}\epsilon]}$ is generic. Let $a=\pi(\{\alpha_t\}_{t\in [\delta'-1/j_0, \frac{1}{2}\epsilon]})$. We consider the submanifold $Y=\{\pi^{-1}(\pi(\alpha_t)): \delta'-1/j_0\leq t\leq \frac{1}{2}\epsilon\}$. Let $\xi_{\hat{\alpha}}$ and $\xi_\beta$ be the vector fields on $Y$ which are sections of the intersection of the positive cone field with $TY$ which project to the tangent vector field of $a$, and that $\{\hat{\alpha}_t\}_{t\in [\delta'-1/j_0, \frac{1}{2}\epsilon]}$ and $\{\beta_t\}_{t\in  [\delta'-1/j_0, \frac{1}{2}\epsilon]}$ are respectively integral curves of $\xi_{\hat{\alpha}}$ and $\xi_\beta$. (This is possible because this intersection is an open convex cone, see the proof of Prop. 2.4, on the page 367 in \cite{LaM}.) We consider the following family of vector fields $\{\xi_j\}$ on $Y$: For $j\in\mathbf{N}$ with $j>j_0$, $t\in[\delta'-1/j,\delta'+1/j]$ and $E\in\pi^{-1}\{\alpha_t\}$, we let $\xi_j|_{E}=(1-\frac{t-(\delta'-1/j)}{2/j})\xi_\beta|_{E}+(\frac{t-(\delta'-1/j)}{2/j})\xi_{\hat{\alpha}_t}|_{E}$. For $t\in [\delta'-1/j_0, \delta'-1/j]$ and $E\in\pi^{-1}\{\alpha_t\}$, let $\xi_j|_{E}=\xi_\beta|_{E}$. For $t\in [\delta'+1/j, \frac{1}{2}\epsilon]$ and $E\in\pi^{-1}\{\alpha_t\}$, let $\xi_j|_{E}=\xi_{\hat{\alpha}}|_{E}$. Obviously, $\{\xi_j\}$ are continuous for every $j>j_0$. Let $\delta=\delta'-1/j_0$. When $j$ is large enough, similar to the last two paragraphs in the proof of Smoothing Lemma, we may get a positive path $\bar{\alpha}_t$ that satisfies the conditions.

\begin{figure}[ht]
	\centering
	\includegraphics[height=7.0cm]{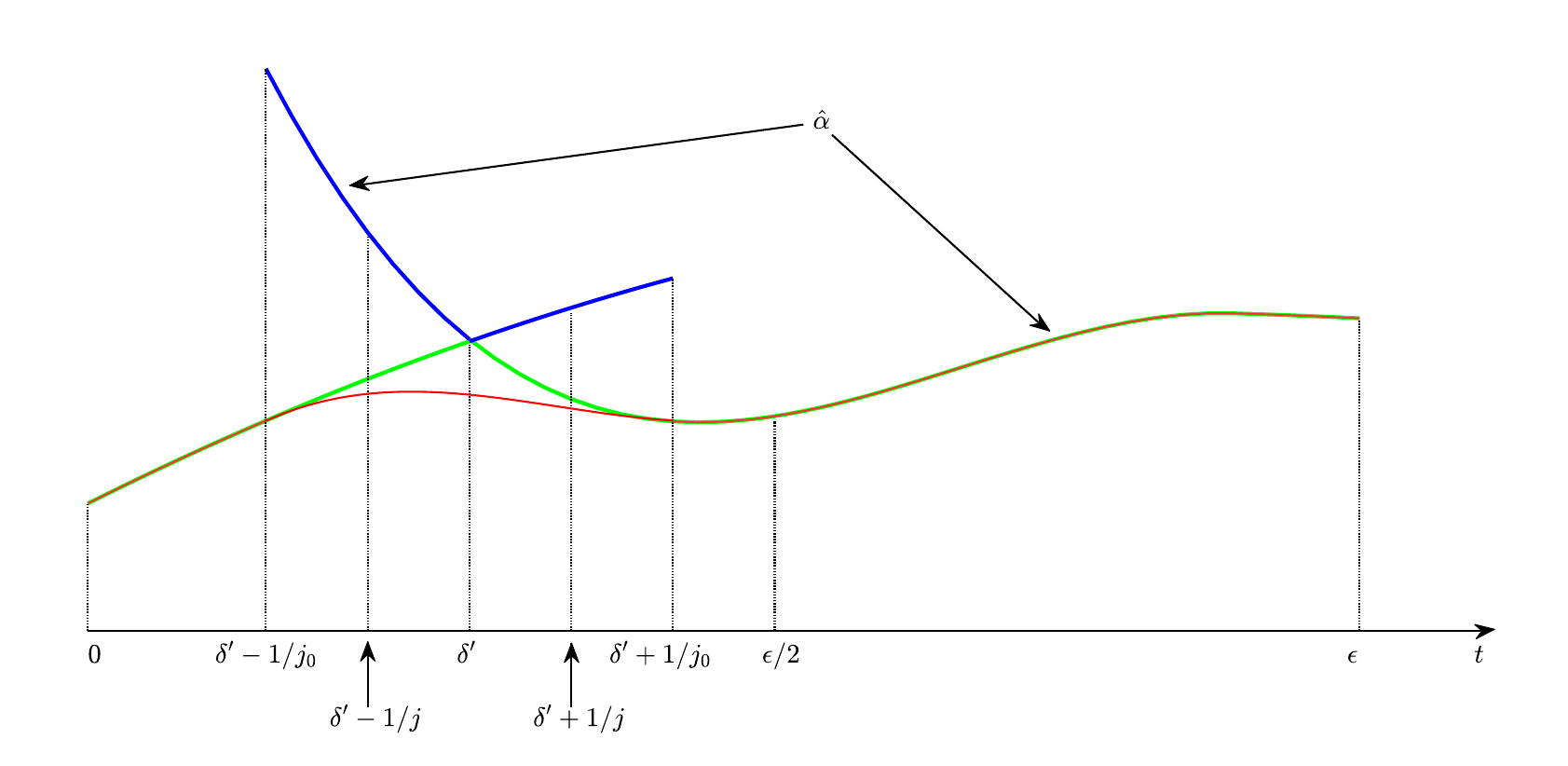}
	\vskip -0.5 cm
	\caption{Step 2}
    \label{step2}
\end{figure}
\end{proof}

\begin{remark}
    In Step 2 of the above proof, to achieve a smooth modification of the path, we require an additional constraint: the path must remain within the same conjugacy class as the original path. Therefore, the proof differs slightly from that in the Smoothing Lemma.
\end{remark}


\subsection{Speed Changing Lemmata and Positive Decomposition Lemma}

The next lemma demonstrates that if we adjust the time scale of a positive path, the adjusted path remains positive, and both paths are positively homotopic.
Let $\sim$ and $\sim_+$ mean homotopic and positively homotopic,
respectively.

\begin{lemma}[Total Speed Changing Lemma]\label{Lemma:change.the.time}

	If $\gamma(t),t\in[0,1]$ is a positive path, then for any differential function $\tau(\cdot)$ with respect to $t$ such that $\tau(0)=0,\tau(1)=1$ and $\tau'(t)>0,\tau\in[0,1]$,
	$\gamma({\tau(t)})$ is a positive path with respect to $t$,
	and $\gamma\sim_+\gamma(\tau(\cdot))$.
	Moreover, if $\tau'(0)=\tau'(1)=1$, we have
	$\gamma\sim_{+,con}\gamma(\tau(\cdot))$.
\end{lemma}

\begin{proof}
By the positiveness of $\gamma(t)$, we have that
$$
P_1(t)=-J{d\gamma(t)\over dt}\gamma^{-1}
$$
is positive definite,
and hence
\begin{equation}
P_2=-J{d\gamma_{\tau(t)}\over dt}\gamma^{-1}=\tau'(t)P_1(\tau(t))
\end{equation}
is positive definite.
Therefore, ${d\gamma_{\tau(t)}\over dt}=JP_2\gamma_{\tau(t)}$ is a positive path with respect to $t$.

Similarly, $H(s,t)=\gamma((1-s)t+s\cdot \tau(t))$ is the required positive homotopy between $\gamma$ and $\gamma(\tau(\cdot))$.
\end{proof}

%
%
%

\begin{proposition}\label{Prop:positive.decomposition}
    If $A\in\Sp(2n)$ is conjugate to an element
    which preserves the splitting 
    $\R^{2n}=\R^{2k}\oplus\R^{2(n-k)}$,
    then $\pi_*$ maps the subcone in $\mathcal{P}_A$
    formed by vectors preserving this splitting onto the entire image $\pi_*(\mathcal{Q})$,
    where $\mathcal{Q}\subseteq\mathcal{P}_A$ consists of vectors preserving the splitting within the conjugacy class.
    
    In other words, for a given positive path $\gamma$
    with $\gamma(0)=A$, such that 
    $X(t)^{-1}\gamma(t)X(t)=\diag(\gamma_1,\gamma_2)$
    hold as $t$ near $0$
    for some path $\{X(t)\}\subset\Sp(2n)$ and 
    $\gamma_1\subset\Sp(2k),\gamma_2\subset\Sp(2(n-k))$,
    then we have another decomposition
    $\bar{X}(t)^{-1}\gamma(t)\bar{X}(t)=\diag(\bar\gamma_1,\bar\gamma_2)$
    holds as $t$ near $0$
    for some path $\{\bar{X}(t)\}\subset\Sp(2n)$ and positive paths
    $\bar\gamma_1\subset\Sp(2k),\bar\gamma_2\subset\Sp(2(n-k))$.
    Such a decomposition is denoted by
    {\bf positive decomposition}.
\end{proposition}

\begin{proof}
    Without loss of generality, we can suppose
    $A=\diag(A_1,A_2)$ where $A_1\in\Sp(2k)$ and
    $A_2\in\Sp(2(n-k))$.
    
    Let $\gamma$ be a $C^1$ path such that $\gamma(0)=A$ 
	and $\gamma'(0)=JPA$ for some positive definite matrix $P$.
	There exists a positive path $\tilde\gamma=\diag(\gamma_1,\gamma_2)$ 
    such that $\tilde\gamma(0)=A$ 
	and $\tilde\gamma'(0)=J\tilde{P}A$ for some symmetric matrix $\tilde P$ (not necessarily positive definite),
	and satisfying:
	\begin{equation}\label{same.conj.1}
	\tilde\gamma(t)=X(t)^{-1}\gamma(t)X(t),
    \quad\forall t\in(-\epsilon,\epsilon)
	\end{equation}
 	for some path $\{X(t)\}\subset\Sp(2n)$ with $X(0)=I_{2n}$ and some $\epsilon>0$.	
	Suppose that $\frac{d }{dt}X(t)\Big|_{t=0}=JY$ where $Y$ is a real symmetry matrix. 
    By Lemma \ref{lem:conj}, we have
    \begin{equation}
	\tilde{P}=P+[A^{-T}YA^{-1}-Y].
	\end{equation}
  
    On the other hand, 
    since $\tilde\gamma=\diag(\gamma_1,\gamma_2)$,
    there exists two symmetric matrices $\tilde{P}_1,\tilde{P}_2$ such that
    $\tilde\gamma_1'(0)=J\tilde{P}_1A_1$,
    $\tilde\gamma_2'(0)=J\tilde{P}_2A_2$;
    and hence $\tilde{P}=\diag(\tilde{P}_1,\tilde{P}_2)$.
    Letting $P=\left(\matrix{P_{11}& P_{12}\cr P_{12}^T& P_{22}}\right)$ and
    $Y=\left(\matrix{Y_{11}& Y_{12}\cr Y_{12}^T& Y_{22}}\right)$,
    we then have
    \begin{equation}
	\left(\matrix{\tilde{P}_1& O\cr O& \tilde{P}_2}\right)
    =\tilde{P}
    =\left(\matrix{P_{11}& P_{12}\cr P_{12}^T& P_{22}}\right)
    +\left(\matrix{A_{1}^{-T}Y_{11}A_{1}^{-1}-Y_{11}& A_{1}^{-T}Y_{12}A_{2}^{-1}-Y_{12}\cr A_{2}^{-T}Y_{12}^TA_{1}^{-1}-Y_{12}^T& A_{2}^{-T}Y_{22}A_{2}^{-1}-Y_{22}}\right).
	\end{equation}
Thus $P_{12}+[A_{1}^{-T}Y_{12}A_{2}^{-1}-Y_{12}]=O$ holds.

Now let $\bar{Y}=\left(\matrix{O& Y_{12}\cr Y_{12}^T& O}\right)$,
and $\{\bar{X}(t)\}\subset\Sp(2n)$ be a symplectic matrix path
    with $\bar{X}(0)=I_{2n}$ and $\frac{d }{dt}\bar{X}(t)\Big|_{t=0}=J\bar{Y}$.
    Hence, $\mu=\bar{X}(t)^{-1}\gamma(t)\bar{X}(t)$
    is in the conjugacy class of $\gamma$
    and $\mu'(0)=J\bar{P}A$ for some symmetric matrix $\bar{P}$.
    At last, by Lemma \ref{lem:conj},
    we have
    \begin{equation}
	\bar{P}=P+[A^{-T}\bar{Y}A^{-1}-\bar{Y}]
    =\left(\matrix{P_{11}& P_{12}\cr P_{12}^T& P_{22}}\right)
    +\left(\matrix{O& A_{1}^{-T}Y_{12}A_{2}^{-1}-Y_{12}\cr A_{2}^{-T}Y_{12}^TA_{1}^{-1}-Y_{12}^T& O}\right)
    =\left(\matrix{P_{11}& O\cr O& P_{22}}\right)
	\end{equation}
    which is positive definite.
\end{proof}

\begin{lemma}[Positive Decomposition Lemma]\label{Lm:positive.decomposition}

	Let $\gamma$ be a generic positive path in $\Sp(2n)$ such that
	\begin{equation}\label{decomposition}
	\gamma(t)=X(t)^{-1}\left(\matrix{ \gamma_1(t)& O\cr O& \gamma_2(t)}\right)X(t),
    \quad t\in[t_0,t_1]
	\end{equation}
	for some paths
    $\{X(t)\}\subset\Sp(2n)$,$\gamma_1\subset\Sp(2k),\gamma_2\subset\Sp(2(n-k))$.
	Then there exists positive paths $\mu_1$ and $\mu_2$
    such that $\pi(\gamma_1)=\pi(\mu_1)$ and
    $\pi(\gamma_2)=\pi(\mu_2)$.
\end{lemma}

\begin{proof}
Without loss of generality, we assume $[t_0,t_1]=[0,1]$. By Proposition \ref{Prop:positive.decomposition}
(also by Lemma B.1 of \cite{CLW} if $\sigma(\gamma_1(t))\cap\sigma(\gamma_2(t))=\emptyset$),
for any $t\in[0,1]$, there exists $\epsilon(t)>0$ such that
$\gamma_1,\gamma_2$ in the decomposition (\ref{decomposition}) are
positive on interval $I_t=(t-\epsilon(t),t+\epsilon(t))\cap[0,1]$.
Here we $C^1$-extend the time interval of $\gamma$ from $[0,1]$
to $(-\ep_0,1+\ep_0)$ for some $\ep_0>0$.

    Since $\cup_{t\in[0,1]}I_t=[0,1]$,
    there exists finitely many intervals, says $I_i,i=1,2,\ldots,m$
    such that
	$\cup_{i=1}^m(s_i-\epsilon(s_i),s_i+\epsilon(s_i))=[0,1]$,
	Therefore, $[0,1]$ can be partitioned into finitely many
	closed sub-intervals $\tilde{I}_i\subset(s_i-\epsilon(s_i),s_i+\epsilon(s_i))$
	such that on each of these intervals, we have decomposition
    \begin{equation}
	\gamma(t)=(X^{(i)}(t))^{-1}\left(\matrix{ \gamma_1^{(i)}(t)& O\cr O& \gamma_2^{(i)}(t)}\right)X^{(i)}(t),
    \quad t\in \tilde{I}_i,
	\end{equation}
    where $\gamma_1^{(i)}$ and $\gamma_2^{(i)}$ are positive on time interval $\tilde{I}_i$ for $i=1,2,\cdots,m$.
 
	Then we glued these positive paths together, up to a conjugation (by (\ref{decomposition}))
	between the adjoint connected points,
	we obtain two piece wise positive paths $\tilde\gamma_1$ and
    $\tilde\gamma_2$ which satisfy
	$\pi(\gamma_i)=\pi(\tilde\gamma_i),i=1,2$.
    Now using Smoothing Lemma, we can find two positive paths
    $\mu_1$ and $\mu_2$ such that $\pi(\mu_i)=\pi(\tilde\gamma_i)=\pi(\gamma_i),i=1,2$.
    Therefore, $\diag(\mu_1,\mu_2)$ is a positive path which satisfying $\pi(\diag(\mu_1,\mu_2))=\pi(\diag(\gamma_1,\gamma_2))=\pi(\gamma)$.
\end{proof}

\begin{lemma}[Part Speed Changing Lemma]\label{Lemma:change.part.time}
Let $\gamma$ be a generic positive path in $\Sp(2n)$ satisfying the condition
(\ref{decomposition}).
If $\tau$ is a monotonic function such that $\tau(t_0)=t_0,\tau(t_1)=t_1,\tau'(t_0)=\tau'(t_1)=1$ and $\tau'(t)>0$ for $t\in[t_0,t_1]$.
Then
\begin{equation}
	\mu(t)=\left(\matrix{\mu_1(t)& O\cr O& \mu_2(\tau(t))}\right),
\end{equation}
where $\mu_1(t)$ and $\mu_2(t)$ are given in Lemma \ref{Lm:positive.decomposition}, is also a generic positive path,
and we have $\gamma\sim_{+,con}\mu$.
\end{lemma}

\begin{proof}
Using the same notations as in the proof
of Lemma \ref{Lm:positive.decomposition},
by Proposition \ref{prop:lift}, we have $\gamma\sim_{+,con}\diag(\mu_1,\mu_2)$.
By changing the time variable via Lemma \ref{Lemma:change.the.time},
we have $\mu_2\sim_{+,con}\mu_2(\tau(\cdot))$, and hence
\begin{equation}
\mu(t)=\left(\matrix{\mu_1(t)& O\cr O& \mu_2(\tau(t))}\right)
\end{equation}
is a positive path, and $\diag(\gamma_1,\gamma_2)\sim_{+,con}\mu$.
Accordingly, we have $\gamma\sim_{+,con}\mu$.
\end{proof}

\begin{remark}
Using the Part Speed Changing Lemma, we can move collisions over specific time intervals.
As a result, we can partition the eigenvalues into several groups
where collisions do not occur between different groups,
ensuring that collisions only happen within each group.
Combined with the Positive Decomposition Lemma,
we can decompose the original positive path into several
positive paths within lower-dimensional symplectic spaces.
\end{remark}


	



If the loops have no collisions of eigenvalues except for $\pm1$, they can be decomposed into positive loops in $\Sp(2)$; thus, they are positively homotopic.  

\begin{theorem}\label{Thm:pos.homotopy.of.loops.has.no.collisions}
	Let $\gamma$ and $\mu$ be two positive loops that start from the identity. Suppose both loops have no collisions of eigenvalues except for $\pm1$. Then, if $\gamma \sim \mu$, it implies $\gamma \sim_+ \mu$.
\end{theorem}

\begin{proof}
	By assumption, there exists the following decomposition:

    \begin{equation}\label{Sp_2n.decomposition}
    \gamma(t)=X_t^{-1}\left(\matrix{ \gamma_1(t)& O& \cdots& O\cr 
                                      O& \gamma_2(t)& \cdots& O\cr
                                      \cdots& \cdots& \cdots& \ldots\cr
                                      O& O& \cdots& \gamma_n(t)}
                                      \right)X_t.
    \end{equation}
    Here $\gamma_i\subset\Sp(2), i=1,2,\cdots,n$ 
    and $\{X_t\}$ are smooth symplectic paths.
    Moreover, by Lemma \ref{Lm:positive.decomposition},
    Lemma \ref{Lemma:change.part.time} and induction,
    we can assume $\gamma_i\subset\Sp(2), i=1,2,\cdots,n$ are positive paths.

    Since homotopy implies positive homotopy for positive paths in $\Sp(2)$, $\gamma_i\sim_+ e^{2\pi k_itJ_2}$ where
    $2k_i$ is the Maslov-type index of $\gamma_i$.
    Then the Maslov-type index of $\gamma$ is $2k_1+2k_2+\cdots+2k_n$.
    
    Similarly, we have
    \begin{eqnarray}
    \mu(t)&=&Y_t^{-1}\diag(\mu_1(t),\mu_2(t),\cdots,\mu_n(t))Y_t
    \nonumber\\
    &\sim_+&\diag(\mu_1(t),\mu_2(t),\cdots,\mu_n(t))
    \nonumber\\
    &\sim_+&\diag(e^{2\pi l_1tJ_2},e^{2\pi l_2tJ_2},\cdots,e^{2\pi l_ntJ_2}),
    \end{eqnarray}
    where $\{Y_t\}$ is a proper symplectic path, $\mu_i(t)$ is a positive path,
    and $2l_i$ is the Maslov-type index of $\mu_i$
    for $i=1,2,\cdots,n$.
    Now the Maslov-type index of $\mu$ is $2l_1+2l_2+\cdots+2l_n$.

    $\gamma\sim\mu$ implies the same Maslov-type indices of $\gamma$
    and $\mu$,
    i.e., $2(k_1+k_2+\cdots+k_n)=2(l_1+l_2+\cdots+l_n)$.
    By Lemma 4.7 of \cite{Sli} and induction, we have
    \begin{equation}
        \diag(e^{2\pi k_1tJ_2},e^{2\pi k_2tJ_2},\cdots,e^{2\pi k_ntJ_2})
        \sim_+\diag(e^{2\pi l_1tJ_2},e^{2\pi l_2tJ_2},\cdots,e^{2\pi l_ntJ_2})
    \end{equation}
    Therefore, we have $\gamma\sim_+\mu$.
\end{proof}

\setcounter{equation}{0}
\section{The positive homotopy within the truly hyperbolic set}
\label{sec:4}

\subsection{No constraints on positive paths in $\mathcal{C}onj(\Sp(2n)$ within the truly hyperbolic set}\label{sec:4.1}

An important observation is that the movement of eigenvalues on $\U$ along positive paths is constrained, but the movement of eigenvalues on $\C\backslash(\{0\}\cup\U)$ is unconstrained (see Proposition D in Section 1.1). Such an observation inspires us to consider the (positive) paths entirely located in the truly hyperbolic set (the definition is given in Section \ref{subsec:2.1}); and for brevity, we refer to such a path as a {\bf (positive) truly hyperbolic path}. Now, we proceed to prove Proposition D.

\begin{proof}[Proof of Proposition D]
	Let $\gamma$ be a $C^1$ path such that $\gamma(0)=A$ 
	and $\gamma'(0)=JQA$ for some symmetric matrix $Q$.
	To demonstrate the surjectivity of $\pi_*(\mathcal{P}_A)$ at $A$,
	it suffices to establish the existence of a positive path $\mu$ such that $\mu(0)=A$ 
	and $\mu'(0)=JPA$ for some positive definite matrix $P$,
	and satisfying:
	\begin{equation}\label{same.conj}
	X(t)^{-1}\gamma(t)X(t)-\mu(t)=o(t),
	\end{equation}
	for some path $\{X(t)\}\in\Sp(2n)$ with $X(0)=I_{2n}$.
	
	Suppose that $\frac{d }{dt}X(t)\Big|_{t=0}=JY$ where $Y$ is a real symmetry matrix. 
	Expanding the first-order term of (\ref{same.conj}), we have
$$
	(I-tJY)(I+tJQ)A(I+tJY)-(I+tJP)A=o(t),
$$
	which implies
$$
	P=Q+J^{-1}[AJY-JYA]A^{-1}=Q+[A^{-T}YA^{-1}-Y].
$$
	
	We now use the normal forms of $A$ (c.f. \cite{Lon}) to find 
 an appropriate $Z$ such that
	$A^{-T}ZA^{-1}-Z$ is positive definite, 
	and hence $Y=kZ$ for $k>0$ large enough 
	will guarantee the positive difinite of $P$.

    Note that the standard symplectic matrix in \cite{Lon} is different from
    the one we have previously used.
   However, it's crucial to emphasize that the proof of this theorem doesn't hinge on the specific forms of the standard symplectic matrix.
For the sake of convenience, in this proof only,
we utilize the original normal forms without providing further justification.
	Referring to Theorem 1 on pp. 34 and Theorem 2 on pp. 36 of \cite{Lon},
	there exists $T\in\Sp(2n)$ such that 
$$
	T^{-1}AT=B_1\diamond\cdots\diamond B_m,
$$
	where $B_i$ takes the form
$$
    M_k(\lambda)=\left(\matrix{A_k(\lambda) & 0\cr 0 & A_k(\lambda)^{-T}}\right),
$$
or the from
$$
    N_{2k}(\rho,\th)=\left(\matrix{A_{2k}(\rho,\theta) & 0\cr 0 & A_{2k}(\rho,\theta)^{-T}}\right),
$$
	where
	\begin{eqnarray}
	A_k(\lambda)&=&\left(\matrix{
		\lambda &   \epsilon   & 0 & \cdots      & 0     & 0\cr
		0 & \lambda &    \epsilon     & \cdots & 0     & 0\cr
		0 &   0   &  \lambda &    \cdots & 0     & 0\cr
		\cdot & \cdot &    \cdot     &  \cdots & \cdot  & \cdot\cr
		0 &    0  &    0     &     \cdots  &\lambda&\epsilon\cr
		0 &    0  &    0     &     \cdots  &0      &\lambda\cr
	}\right), \quad\lambda>1,
 \nonumber\\
	A_{2k}(\rho,\th)&=&\left(\matrix{
		\rho{R}(\th) &   \epsilon{I}_2   & 0 & \cdots   & 0   & 0\cr
		0 & \rho{R}(\th) & \epsilon{I}_2  & \cdots & 0     & 0\cr
		0 &  0  &  \rho{R}(\th) & \cdots & 0    & 0\cr
		\cdot & \cdot &    \cdot     & \cdots & \cdot   & \cdot\cr
		0 &  0  &  0   & \cdots  &\rho{R}(\th)&\epsilon{I}_2\cr
		0 &  0  &  0   & \cdots  &0&\rho{R}(\th)\cr
	}\right),\quad\rho>1.\nonumber
	\end{eqnarray}
	Here, note that we substitute $\epsilon$ and $\epsilon{I}_2$ for $1$ and $I_2$, respectively,
in the Jordan block of $A_k(\lambda)$ and $A_{2k}(\rho,\theta)$.
	
If $B_i=M_k(\lambda)$, 
 then letting $Z_i=\diag(-I_k,I_k)$,
	we have
	\begin{eqnarray}
	B_i^{-T}Z_iB_i^{-1}-Z_i&=&
	\left(\matrix{A_k(\lambda)^{-T} & 0\cr 0 & A_k(\lambda)}\right)
	\left(\matrix{-I_k & 0\cr 0 & I_k}\right)
	\left(\matrix{A_k(\lambda)^{-1} & 0\cr 0 & A_k(\lambda)^{T}}\right)
	-\left(\matrix{-I_k & 0\cr 0 & I_k}\right)
	\nonumber\\
	&=&
	\left(\matrix{-A_k(\lambda)^{-T}A_k(\lambda)^{-1}+I_k & 0\cr 
		0 & A_k(\lambda)A_k(\lambda)^{T}-I_k}\right).\nonumber
	\end{eqnarray}
Since $-A_k(\lambda)^{-T}A_k(\lambda)^{-1}+I_k=A_k(\lambda)^{-T}[A_k(\lambda)^{T}A_k(\lambda)-I_k]A_k(\lambda)^{-1}$
	and
	\begin{eqnarray}
	A_k(\lambda)A_k(\lambda)^{T}-I_k&=&
	\left(\matrix{
		\lambda^2-1+\epsilon^2 &   \lambda\epsilon   & 0 & \cdots      & 0     & 0\cr
		\lambda\epsilon& \lambda^2-1+\epsilon^2& \lambda\epsilon& \cdots& 0& 0\cr
		0& \lambda\epsilon& \lambda^2-1+\epsilon^2& \cdots& 0 & 0\cr
		\cdot & \cdot &    \cdot     &  \cdots & \cdot  & \cdot\cr
		0 & 0& 0&\cdots& \lambda^2-1+\epsilon^2& \lambda\epsilon\cr
		0 &   0  &  0  &  \cdots  &\lambda\epsilon &\lambda^2-1\cr
	}\right)
    \nonumber
	\end{eqnarray}
is positive definite for $\epsilon$ small enough,
	it can be conclude that $B_i^{-T}Z_iB_i^{-1}-Z_i$ is positive definite.
	Similarly, if $B_i=N_{2k}(\rho,\th)$, letting $Z_i=\diag(-I_{2k},I_{2k})$,
	we also have that $B_i^{-T}Z_iB_i^{-1}-Z_i$ is positive definite.
	
	Now the proof is complete.	
\end{proof}

\begin{lemma}\label{Lemma:hyperbolic.positive.path}
	If $\gamma$ is a $C^1$ truly hyperbolic path (not necessarily positive) in $\Sp(2n)$,
	there exists a positive truly hyperbolic path $\mu$ such that
	$\pi(\gamma)=\pi(\mu)$.
\end{lemma}
\begin{proof}
	Suppose $\gamma'(t)=JQ(t)\gamma(t)$.
	For any $s\in[0,1]$,
	since $\pi_*(\mathcal{P}_A)$ is surjective at $\gamma(s)$ 
	by Proposition D,
	there exists a symmetric matrix $Y_s$ such that
	$Q(s)+[\gamma(s)^{-T}Y_s\gamma(s)^{-1}-Y_s]>0$.
	Letting $X_s(t)=e^{tJY_s}$ and $\gamma_s(t)=X_s(t)^{-1}\gamma(t)X_s(t)$,
	then $\pi(\gamma)=\pi(\gamma_s)$; 
    and hence, by Lemma \ref{lem:conj}, we have
	\begin{equation}
	-J{d\gamma_s(t)\over dt}\gamma_s(t)^{-1}=
	X_s(t)^{T}[Q(t)+(\gamma(t)^{-T}Y_s\gamma(t)^{-1}-Y_s)]X_s(t),
	\end{equation}
	which is positive definite if $t=s$,
	and hence it is also positive definite on interval $I_s=(s-\epsilon(s),s+\epsilon(s))\cap[0,1]$	
	for some $\epsilon(s)>0$ which 
    may depends on $s$.
	
	Since $\cup_{s\in[0,1]}I_s=[0,1]$,
	there exist finitely many $s_i,i=1,2,\cdots,m$ such that
	$\cup_{i=1}^m I_{s_i}=[0,1]$,
	Therefore, $[0,1]$ can be partitioned into finitely many
	closed sub-intervals $I_i\subset I_{s_i}$
	such that $\gamma|_{I_i}$ is a positive path.
	Then, by gluing these positive paths together, allowing for conjugation
	between the adjoint endpoints,
	we obtain a piece wise positive path $\mu$.
	Furthermore, $\pi(\gamma)=\pi(\mu)$.
	
	Finally, by the Smoothing Lemma (Lemma \ref{lem: smoothing}), we can find a $C^1$ positive
	path in the same conjugation class $\pi(\mu)$.
	
\end{proof}

\subsection{The positive homotopy
of truly hyperbolic paths}


We have shown the existence of generic homotopies in Lemma \ref{lem:generic homotopy}. Additionally, we demonstrate the following existence of positive homotopies within the hyperbolic region.

\begin{lemma}\label{Lemma:positive.homotopy}
Given a generic homotopy $H(s,t)$ ($t_0\leq t\leq t_1$) within the
truly hyperbolic region, where $0\leq t_0<t_1\leq 1$,  which is $C^0$ with respect to $s$ and $C^1$ with respect to $t$. If $\{H(0,t)\}$ and $\{H(1,t)\}$ are positive paths, 
then there exists a positive homotopy $\bar{H}(s,t)$ ($t_0\leq t\leq t_1$) from $\{H(0,t)\}$ to $\{H(1,t)\}$.
Moreover, if $H(s,t)$ is a $C^1$-connectable homotopy, then the positive homotopy $\bar{H}(s,t)$ is also $C^1$-connectable.
\end{lemma}

\begin{proof}
As before, by the Whitney Approximation Theorem and a similar argument in the proof of Lemma \ref{lem:generic homotopy}, we may assume that the homotopy is also $C^1$ with respect to $s$. By hypothesis, $\{H(0,t)\}$ and $\{H(1,t)\}$ are positive paths. For any $r\in (0,1)$, thanks to Lemma \ref{Lemma:hyperbolic.positive.path} above, there exists a positive path $\gamma_{r}(t)$ such that $\pi(\gamma_{r}(t))=\pi(H(r,t))$. 
We define $\gamma_{r}(t)=(X_{t}^{r})^{-1}H(r,t)X_{t}^{r}$ (without loss of generality, we may assume that $X_{t_0}^{r}=I$). 
According to Lemma \ref{lem:c0c1}(i), there is $\epsilon_{r}>0$ small enough such that for every $s\in (r-\epsilon_{r},r+\epsilon_{r})$, $(X_{t}^{r})^{-1}H(s,t)X_{t}^{r}$ is also a positive path. 
For $s=0$ and $s=1$, there are $\epsilon_0$ and $\epsilon_1$ such that when $s\in[0,\epsilon_0)\cup(\epsilon_1,1]$, $\{H(s,t)\}$ is a positive path. 
In these cases, $X^0_t=X^1_t=I$ for all $t_0\leq t\leq t_1$. 
Given  $H(s,t)$ ($t_0\leq t\leq t_1$), there are $s_0=0<s_1<\cdots<s_{k-1}<s_k=1$ and $\{\epsilon_{s_i}\}_{0\leq i\leq k}$ satisfying that

\begin{itemize}
\item $0<s_1-\epsilon_{s_1}<\epsilon_0<s_2-\epsilon_{s_2}<s_1+\epsilon_{s_1}<s_3-\epsilon_{s_3}<s_2+\epsilon_{s_2}<\cdots<1-\epsilon_1<s_{k-1}+\epsilon_{s_{k-1}}<1$;

\item  for every $s\in (s_i-\epsilon_{s_i},s_i+\epsilon_{s_i})$, $\{(X_{t}^{s_i})^{-1}H(s,t)X_{t}^{s_i}\}$ is  a positive path.

\end{itemize}

Since $H(s,t)$ is a generic homotopy, by Definition \ref{def:generic homotopy}, for every open set $(s_{i+1}-\epsilon_{s_{i+1}}, s_i+\epsilon_{s_i})$ ($0\leq i\leq k-1$), we can choose $s'_i$ in it such that $\{H(s'_i,t)\}$ is a generic path. For every $0\leq i\leq k-1$, we consider $\{(X_{t}^{s_i})^{-1}H(s'_i,t)X_{t}^{s_i}\}$ and $\{(X_{t}^{s_{i+1}})^{-1}H(s'_i,t)X_{t}^{s_{i+1}}\}$. 
According to Proposition \ref{prop:lift}, there exists a positive homotopy $H_i(s,t)$ from $\{(X_{t}^{s_i})^{-1}H(s'_i,t)X_{t}^{s_i}\}$ to $\{(X_{t}^{s_{i+1}})^{-1}H(s'_i,t)X_{t}^{s_{i+1}}\}$. 
The homotopy $\bar{H}(s,t)$ is then constructed as the concatenation of these homotopies
\begin{eqnarray}
H(s,t) (0\leq s\leq s'_0), H_0, (X_{t}^{s_1})^{-1}H(s,t)X_{t}^{s_1}  (s'_0\leq s \leq s'_1), H_1, \cdots, & &\nonumber\\
(X_{t}^{s_{k-1}})^{-1}H(s,t)X_{t}^{s_{k-1}} (s'_{k-2}\leq s \leq s'_{k-1}), H_{k-1}, H(s,t) (s'_{k-1}\leq s\leq 1)\nonumber
\end{eqnarray}

 From the construction, it easy to see that  $\bar{H}(s,t)$ is a positive homotopy from $\{H(0,t)\}$ to $\{H(1,t)\}$.

Furthermore, if $H(s,t)$ is a $C^1$-connectable homotopy, that is, there exist two paths of matrices $\{X_s\}, \{Y_s\} \subset \Sp(2n)$ where $s \in [0,1]$, with $X_0 = Y_0 = I$, such that:
	\begin{eqnarray*}
	H(s,t_0)&=&X_s^{-1}H(0,t_0)X_s,\quad 
	\partial_tH(s,t_0)=X_s^{-1}\partial_tH(0,t_0)X_s,
	\nonumber\\
	H(s,t_1)&=&Y_s^{-1}H(0,t_1)Y_s,\quad 
	\partial_tH(s,t_1)=Y_s^{-1}\partial_tH(0,t_1)Y_s.
	\end{eqnarray*}

We can first find two smooth paths of matrices $\{\tilde{X}_s\}_{s\in [0,1]}, \{\tilde{Y}_s\}_{s\in [0,1]} \subset \Sp(2n)$ with $\tilde{X}_0 = \tilde{Y}_0 = I$ and $\tilde{X}_1=X_1, \tilde{Y}_1=Y_1$ that are close to the paths $\{X_s\}$ and $\{Y_s\}$, respectively.

We will find a homotopy $\tilde{H}(s,t)$ in a small neighborhood of $H(s,t)$ satisfying that
\begin{equation}\label{C1.connectable.homotopy.proof1}
\tilde{H}(s,t_0)=\tilde{X}_s^{-1}H(0,t_0)\tilde{X}_s,\quad 
	\tilde{H}(s,t_1)=\tilde{Y}_s^{-1}H(0,t_1)\tilde{Y}_s\end{equation}
 and 
\begin{eqnarray}\label{C1.connectable.homotopy.proof2}
	\partial_t\tilde{H}(s,t_0)=\tilde{X}_s^{-1}\partial_t H(0,t_0)\tilde{X}_s,\quad \tilde{H}(0,t)=H(0,t),
	\nonumber\\
	\partial_t\tilde{H}(s,t_1)=\tilde{Y}_s^{-1}\partial_tH(0,t_1)\tilde{Y}_s,\quad \tilde{H}(1,t)=H(1,t).
	\end{eqnarray}
In fact, for some small $\ep>0$,
letting
\begin{eqnarray*}
    P(t)&=&J^{-1}\partial_t H(0,t)\cdot H(0,t)^{-1},\\
 Q(t)&=&J^{-1}\partial_t H(1,t)\cdot H(1,t)^{-1},
\end{eqnarray*}
for $(s,t)\in[0,1]\times[t_0,t_0+\epsilon]$,
we define $\tilde{H}(s,t)$ as the following
\begin{eqnarray*}
\left\{\matrix{
	\hat{X}_s^{-1}H(0,t)\hat{X}_s,
        &s\le{1\over2},t\in[t_0,t_0+(1-2s)\epsilon],\cr 
	\hat{X}_s^{-1}e^{JP(t_0+(1-2s)\epsilon)[t-t_0-(1-2s)\epsilon]}H(0,t_0+(1-2s)\epsilon)\hat{X}_s,
        &s\le{1\over2},t\in[t_0+(1-2s)\epsilon,t_0+\ep],\cr
    \hat{X}_s^{-1}\hat{X}_1H(1,t)\hat{X}_1^{-1}\hat{X}_s,
        &s\ge{1\over2},t\in[t_0,t_0+(2s-1)\epsilon],\cr 
    \hat{X}_s^{-1}\hat{X}_1e^{JQ(t_0+(2s-1)\epsilon)[t-t_0-(2s-1)\epsilon]}H(1,t_0+(2s-1)\epsilon)\hat{X}_1^{-1}\hat{X}_s,
        &s\ge{1\over2},t\in[t_0+(2s-1)\epsilon,t_0+\ep].}\right.
\end{eqnarray*}
Similarly, we can define $\tilde{H}(s,t)$ on the range $[0,1]\times[t_1-\ep,t_1]$. Using this approach, we can easily obtain a $C^0$ homotopy $\tilde{H}(s,t)$ defined on $[0,1]\times[t_0,t_1]$ that satisfies (\ref{C1.connectable.homotopy.proof1}) and (\ref{C1.connectable.homotopy.proof2}).

 Let $$A=\left(\bigcup_{s\in[0,1], t\in[t_0,t_0+\frac{\epsilon}{2}]\cup[t_1-\frac{\epsilon}{2},t_1]}\tilde{H}(s,t)\right)\cup\left(\bigcup_{t\in[t_0,t_1]}(\tilde{H}(0,t)\cup\tilde{H}(1,t))\right).$$ 
 By the Whitney Approximation Theorem (see e.g. \cite[pp. 141, Theorem 6.26]{Lee}) and the argument of the proof of Lemma \ref{lem:generic homotopy}, we can get a smooth generic homotopy $\hat{H}(s,t)$ from $\{H(0,t)\}$ to $\{H(1,t)\}$ close to $\tilde{H}(s,t)$ (and hence close to $H(s,t)$) such that $\hat{H}|_A=\tilde{H}|_A$. After the same argument as above, we can obtain a positive homotopy $\bar{H}(s,t)$ that is also $C^1$-connectable.
\end{proof}

\begin{theorem}
\label{Thm:remove.hyperbolic.collisions}
    Any two positive truly hyperbolic generic paths
    within the same connected component of the truly hyperbolic set
    are positively homotopic.
    Moreover, if these two paths satisfy the $C^1$-connectable condition,
    they are $C^1$-connectable homotopic.
\end{theorem}
\begin{proof}For any two truly hyperbolic generic paths, we may construct a homotopy between the two paths based on Proposition D. 
Furthermore, we can perturb it to a generic homotopy by Lemma \ref{lem:generic homotopy}. 
Finally, it immediately follows from Lemma \ref{Lemma:positive.homotopy}.
\end{proof}

According to Theorem \ref{Thm:remove.hyperbolic.collisions}, any two positive truly hyperbolic generic paths that satisfy the $C^1$-connectable condition are $C^1$-connectable positive homotopic. Furthermore, given a positive truly hyperbolic generic path $\gamma$, if there exists another positive truly hyperbolic path $\mu$ with no collisions such that the $C^1$-connectable conditions between $\gamma$ and $\mu$ hold, then we have $\gamma\sim_{+,con}\mu$. Thus, the collisions of such $\gamma$ can be completely removed in the sense of $C^1$-connectable positive homotopy.



\setcounter{equation}{0}
\section{Positive paths in $\Sp(4)$}
\label{sec:5}

\subsection{The structure of elements in $\Sp(4)$}
\label{subsec:5.1}

For convenience,
we will list all the open regions 
and its codimension 1 boundaries 
of $\mathcal{C}onj(\Sp(4))$ which we will used later
(c.f. \cite{LaM} and \cite{Sli}).
A general element of $\Sp(4)$ lies in one of the following open regions:

\hangafter 1
\hangindent 4em
\;\;\;(i) $\mathcal{O}_{\mathcal{C}}$,
consisting of all matrices with $4$ distinct eigenvalues in
$\C\backslash(\U\cup\R)$;
one conjugacy class for each quadruple;

\hangafter 1
\hangindent 4em
\;\;(ii) $\mathcal{O}_{\mathcal{U}}$,
consisting of all matrices with eigenvalues on $\U\backslash\{\pm1\}$
where each eigenvalues has multiplicity 1 
(or multiplicity $2$ with non-zero splitting number);
four (or two) conjugacy classes for each quadruple
corresponding to the splitting numbers;

\hangafter 1
\hangindent 4em
\;(iii) $\mathcal{O}_{\mathcal{R}}$,
consisting of all matrices whose eigenvalues 
have multiplicity $1$ and lie on
$\R\backslash\{0,\pm1\}$;
and a conjugate pair of eigenvalues on
one conjugacy class for each quadruple;

\hangafter 1
\hangindent 4em
\;\;(iv) $\mathcal{O}_{\mathcal{U,R}}$,
consisting of all matrices with  $4$ distinct eigenvalues, one pair on $\U\backslash\{\pm1\}$
and the other on $\R\backslash\{0,\pm1\}$;
two conjugacy classes for each quadruple
corresponding to the possible splitting numbers
of the pair on $\U\backslash\{\pm1\}$.

\noindent The codimension $1$ part of the boundaries of the above
regions are:

\hangafter 1
\hangindent 4em
\;\;\;(v) $\mathcal{B}_{\mathcal{U}}$,
consisting of all non-diagonalizable matrices whose spectrum is a pair of conjugate points 
$\U\backslash\{\pm1\}$
each of multiplicity $2$ and splitting number $0$;
two conjugacy class for each quadruple:
$\mathcal{B}_{\mathcal{U}}^-$ containing those matrices
from which positive paths enter $\mathcal{O}_{\mathcal{C}}$ 
and $\mathcal{B}_{\mathcal{U}}^+$ containing those matrices
from which positive paths enter $\mathcal{O}_{\mathcal{U}}$;

\hangafter 1
\hangindent 4em
\;\;(vi) $\mathcal{B}_{\mathcal{R}}$,
consisting of all non-diagonalizable matrices whose spectrum is a pair of distinct points 
$\lambda,{1\over\lambda}\in\R\backslash\{0,\pm1\}$
each of multiplicity $2$;
one conjugacy class for each quadruple;

\hangafter 1
\hangindent 4em
\;(vii) $\mathcal{B}_{\mathcal{U},1}$ (rep. $\mathcal{B}_{\mathcal{U},-1}$),
consisting of all non-diagonalizable matrices with eigenvalues $\{\lambda,{\bar\lambda},1,1\}$
(or eigenvalues $\{\lambda,{\bar\lambda},-1,-1\}$)
with $\lambda\in\U\backslash\{\pm1\}$;
two conjugacy class for each quadruple,
corresponding to $N_1^-$ 
and $N_1^+$;


\hangafter 1
\hangindent 4em
(viii) $\mathcal{B}_{\mathcal{R},1}$ (rep. $\mathcal{B}_{\mathcal{R},-1}$),
consisting of all non-diagonalizable matrices with eigenvalues $\{\lambda,{1\over\lambda},1,1\}$
(or eigenvalues $\{\lambda,{1\over\lambda},-1,-1\}$)
with $\lambda\in\R\backslash\{0,\pm1\}$;
two conjugacy class for each quadruple,
corresponding to $N_1^-$ 
and $N_1^+$.

\noindent The codimension $2$ part of the boundaries of the above
region is:

\hangafter 1
\hangindent 4em
\;\;(ix) $\mathcal{B}_{1}$ (rep. $\mathcal{B}_{-1}$),
consisting of all  matrices with eigenvalues $\{1,1,1,1\}$ (or eigenvalues $\{-1,-1,-1,-1\}$)
which has only one Jordan block;
two conjugacy class for each quadruple.

\noindent In addition, there are two important strata of higher codimension:

\hangafter 1
\hangindent 4em
\;\;\;(x) $\mathcal{B}_{\mathcal{U,D}}$,
consisting of all diagonalizable matrices with a
conjugate pair of eigenvalues on $\U\backslash\{\pm1\}$,
each of multiplicity two with $0$ splitting number;
one conjugacy class for each quadruple;

\hangafter 1
\hangindent 4em
\;\;(xi) $\mathcal{B}_{\mathcal{R,D}}$,
consisting of all diagonalizable matrices with a
conjugate pair of eigenvalues on $\R\backslash\{0,\pm1\}$,
each of multiplicity two;
one conjugacy class for each quadruple.

\noindent Here in (ix), all the codimension $2$ part
in $\Sp(4)$ are $\mathcal{B}_{1}$
and $\mathcal{B}_{-1}$ which follows by
Lemma \ref{bifurcation.digram.of.c=1.2}.


\subsection{The constraints
of the collisions of positive paths in $\Sp(4)$}

	
	

As emphasized in introduction, a key observation is the constraints
of the collisions,
which raised from the positivity of paths in $\Sp(4)$.

\begin{lemma}\label{toplogical.constraint.of.Sp4}
	Let $\gamma$ be a generic positive path in $\Sp(4)$. We have
	
	(i) if $\gamma$ leaves $\mathcal{O}_\mathcal{U}$, enters $\mathcal{O}_\mathcal{C}$
	at time $t^*$,
	then after the previous collision (if exist),
	there must exist a pair of eigenvalues of the path which comes from $\{1,1\}$ or $\{-1,-1\}$ at time $t_0 < t^*$;
	
	(ii) if $\gamma$ leaves $\mathcal{O}_\mathcal{C}$, enters $\mathcal{O}_\mathcal{U}$
	at time $t^*$, then before the next collision (if exist), there must exist a pair of eigenvalues of the path which enters $\{1,1\}$ or $\{-1,-1\}$ at time $t_0 > t^*$.
\end{lemma}

Now, suppose $\gamma$ leaves $\mathcal{O}_\mathcal{U}$ and enters $\mathcal{O}_\mathcal{C}$ at time $t^*$. Then, there must exist a pair of eigenvalues of the path that comes from $N_1^+$ or $N_{-1}^+$ at time $t_0 < t^*$. Up to a small perturbation, we can assume there is only one pair of such eigenvalues. When $t\in(t_0, t^*)$, the two pairs of eigenvalues do not collide or pass through $\pm1$. In other words, there exists an $\epsilon > 0$ small enough such that
\begin{equation}
\gamma(t)\in\left\{\matrix{
	\mathcal{O}_{\mathcal{U},\mathcal{R}},&t\in[t_0-\epsilon,t_0),\cr 
	\mathcal{B}_{\mathcal{U},1},&t=t_0,\cr
    \mathcal{O}_{\mathcal{U}},&t\in(t_0,t^*),\cr 
    \mathcal{B}_{\mathcal{U}},&t=t^*,\cr
    \mathcal{O}_{\mathcal{C}},&t\in(t^*,t^*+\epsilon].\cr }\right.
\end{equation}

\begin{figure}[ht]
	\centering
	\includegraphics[height=8.0cm]{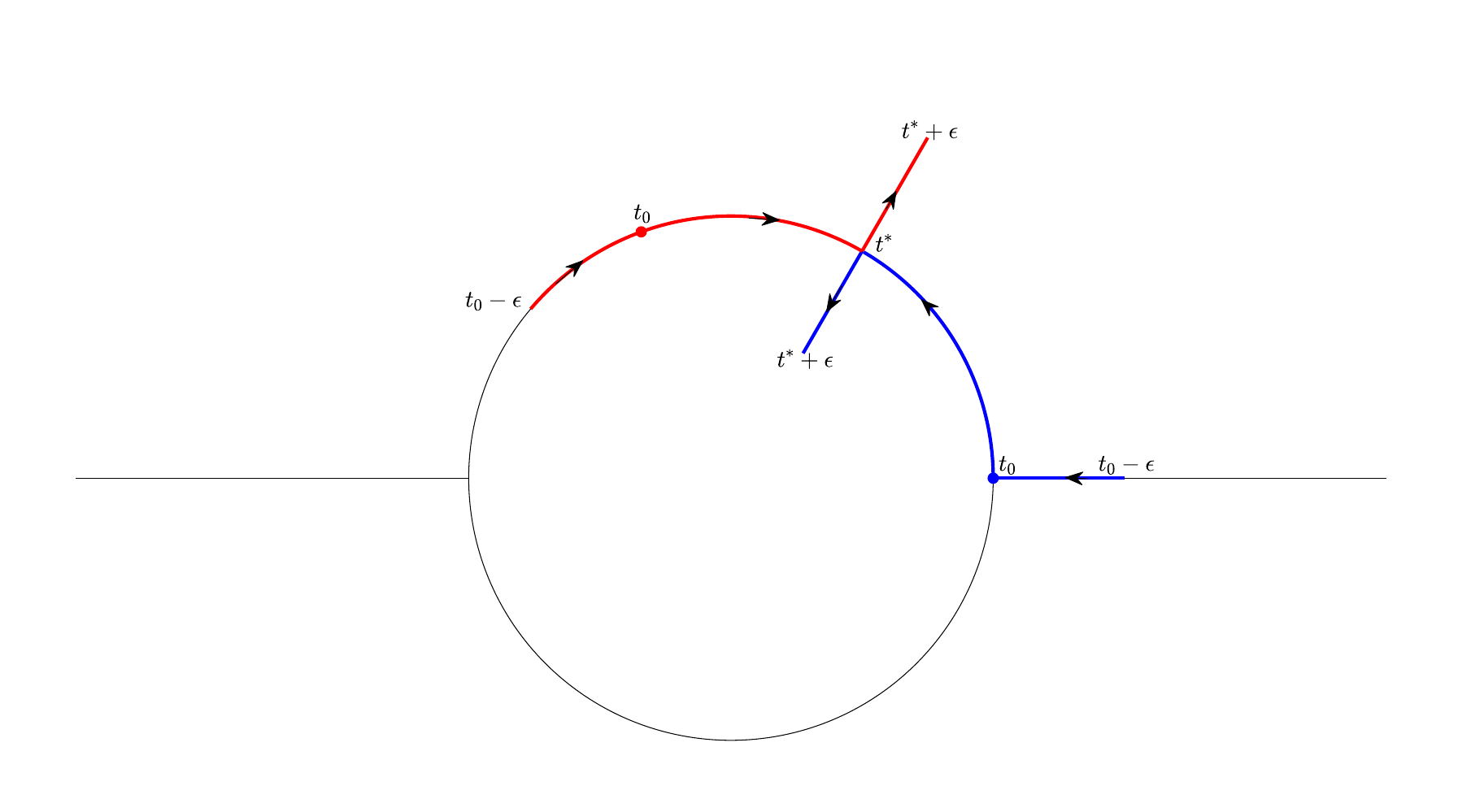}
	\vskip -0.5 cm
	\caption{An elementary collide-out path. 
    The red curve represents a pair of eigenvalues that comes from $N_1^+$. The blue curve represents the another pair of eigenvalues that remain on the unit circle for $t\in[t_0-\epsilon,t^*)$.}
 \label{collide-out.path}
\end{figure}
Then $\mu=\gamma|_{[t_0-\epsilon,t^*+\epsilon]}$ is a positive path connects $\gamma({t_0-\epsilon})\in\mathcal{O}_{\mathcal{U},\mathcal{R}}$ 
and $\gamma({t^*+\epsilon})\in\mathcal{O}_{\mathcal{C}}$.
We define

\begin{definition}\label{elementary.collide.path}
	An {\bf elementary collide-out path} is a symplectic path in $\Sp(4)$
	that follows the trajectory:
	\begin{equation}
	\mathcal{O}_{\mathcal{U},\mathcal{R}}
	\rightarrow\mathcal{B}_{\mathcal{U},1}
	\rightarrow\mathcal{O}_{\mathcal{U}}
	\rightarrow\mathcal{B}_{\mathcal{U}}
	\rightarrow\mathcal{O}_{\mathcal{C}}.
	\end{equation}
	
	Similarly, an {\bf elementary collide-in path} is a symplectic path in $\Sp(4)$
	that follows the trajectory:
	\begin{equation}
	\mathcal{O}_{\mathcal{C}}
	\rightarrow\mathcal{B}_{\mathcal{U}}
	\rightarrow\mathcal{O}_{\mathcal{U}}
	\rightarrow\mathcal{B}_{\mathcal{U},1}
	\rightarrow\mathcal{O}_{\mathcal{U},\mathcal{R}}.
	\end{equation}
\end{definition}

\begin{lemma}\label{Th:positive.homotopy.of.two.elementary.path}
 Any two elementary collide-out (collide-in) path with the same conjugate collision points
 which satisfy the $C^1$-connectable conditions are positively homotopic.
\end{lemma}

\begin{proof}
Suppose $\gamma$ and $\mu$ be two elementary collide-out paths
that satisfy the $C^1$-connectable condition.
Moreover, by changing the time scale, 
we suppose the two paths are collide at $t={1\over2}$,
and $\mu({1\over2})=X_0^{-1}\gamma({1\over2})X_0$ for some $X_0\in\Sp(4)$.

Letting $\tilde\gamma(t)=X_0^{-1}\gamma({t})X_0$,
then we have $\gamma\sim_{+,con}\tilde\gamma$.
Since $\mu({1\over2})=\tilde\gamma({1\over2}$),
by Lemma \ref{lem:c0c1}(iii), we can suppose that
$\tilde\gamma(t)=\mu(t)$ for $t\in({1\over2}-\epsilon,{1\over2}+\epsilon)$
with some $\epsilon>0$, please see Figure \ref{figure:tilde.gamma}.
\begin{figure}[ht]
	\centering
	\includegraphics[height=8cm]{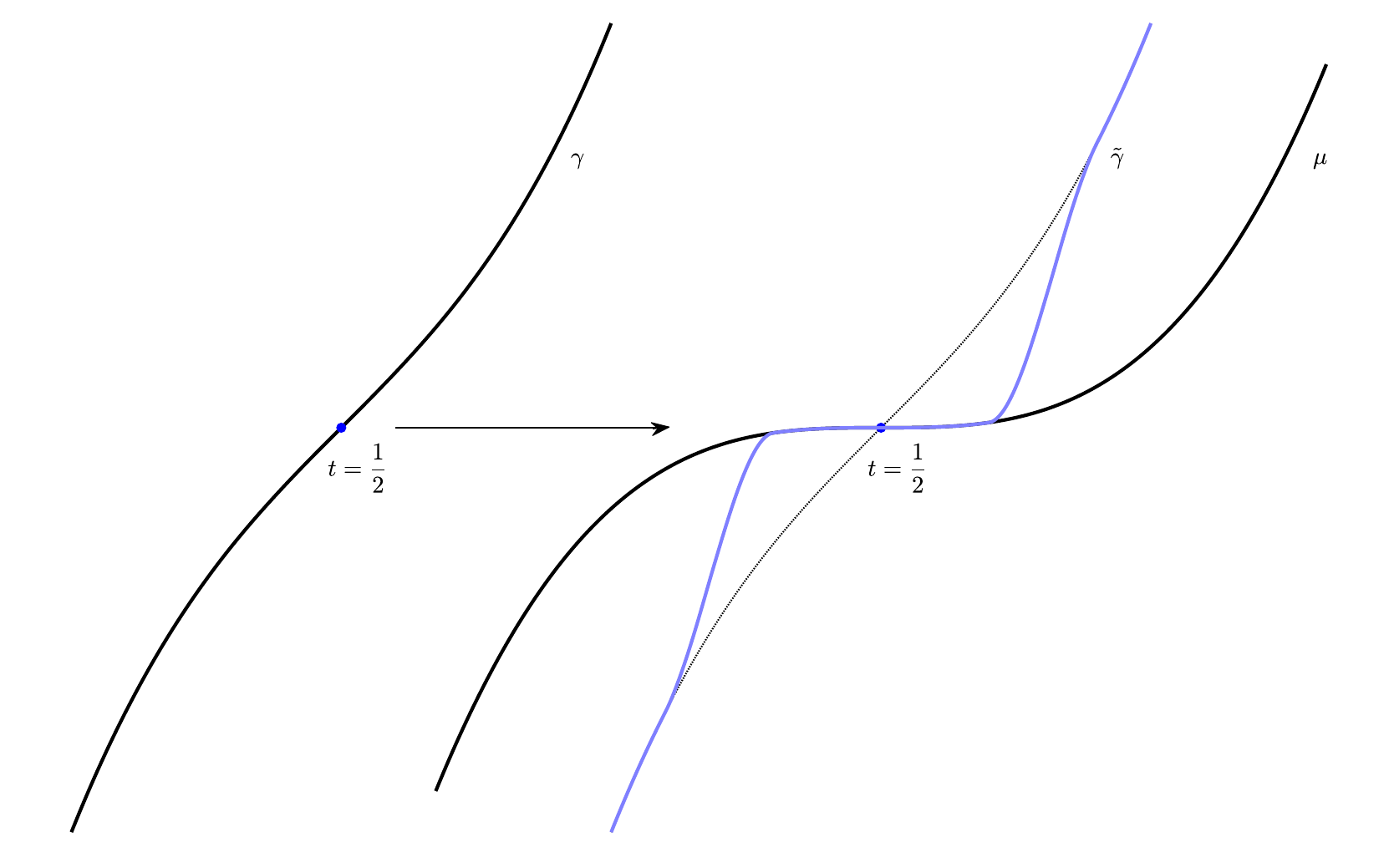}
	\vskip -0.5 cm
	\caption{The illustration of $\tilde{\gamma}$}
 \label{figure:tilde.gamma}
\end{figure}

Now the sub-paths $\tilde{\gamma}|_{[0,{1\over2}-{\epsilon\over2}]}$
and $\mu|_{[0,{1\over2}-{\epsilon\over2}]}$ have no collisions,
and satisfy the $C^1$-connectable conditions.
By Lemma \ref{Lm:positive.decomposition},
there exist positive decompositions
$\tilde{\gamma}=Y_t^{-1}\diag(\tilde\gamma_1,\tilde\gamma_2)Y_t$
and $\mu=Z_t^{-1}\diag(\mu_1,\mu_2)Z_t$
for $t\in[0,{1\over2}-{\epsilon\over2}]$.
Here $\tilde\gamma_1,\tilde\gamma_2$ and $\mu_1,\mu_2$
are four positive paths in $\Sp(2)$.
By changing the speed of the eigenvalues of $\tilde\gamma_1$ and $\mu_1$,
we have $\pi(\tilde\gamma_1)=\pi(\mu_1)$;
and hence $\tilde\gamma_1\sim_{+,con}\mu_1$.
Similarly, we have
$\tilde\gamma_2\sim_{+,con}\mu_2$.
Therefore, $\tilde{\gamma}|_{[0,{1\over2}-{\epsilon\over2}]}\sim_{+,con}\mu|_{[0,{1\over2}-{\epsilon\over2}]}$

Note that $\tilde\gamma|_{[{1\over2}+{\epsilon\over2},1]}$
and $\mu|_{[{1\over2}+{\epsilon\over2},1]}$
are two positive truly hyperbolic paths with $C^1$-connectable conditions,
by Theorem \ref{Thm:remove.hyperbolic.collisions},
we have $\tilde\gamma|_{[{1\over2}+{\epsilon\over2},1]}\sim_{+,con}\mu|_{[{1\over2}+{\epsilon\over2},1]}$.

At last, by gluing the three positive homotopies, 
we have $\gamma\sim_{+,con}\mu$.
\end{proof}

\subsection{Move the collisions along the unit circle $\U$}

The following theorem demonstrates that, for a given elementary collide-in (or collide-out) path, its single collision can be moved along an open subset of $\U\backslash\{\pm1\}$ through a $C^1$-connectable positive homotopy.

To describe more precisely,
given an elementary collide-out path $\gamma(t),t\in[0,1]$ (the argument for elementary collide-in path is similar),
we can suppose that
$\gamma$ leaves $\mathcal{O}_\mathcal{U}$, enters $\mathcal{O}_\mathcal{C}$
at time ${1\over2}$.
Let $\Delta(\gamma(t)):=\sigma_1(\gamma(t))^2-4\sigma_2(\gamma(t))+8$ be  the discriminant of collisions of $\gamma$.
    Then we have
	\begin{equation}
	\Delta(\gamma(t))\left\{\matrix{
		>0&t\in[0,{1\over2})\cr
		=0&t={1\over2}\cr
		<0&t\in({1\over2},1]\cr }\right.
	\end{equation}
Furthermore,
	we suppose
	\begin{eqnarray}
	\sigma(\gamma(0))&=&\{e^{\sqrt{-1}\th_0},e^{-\sqrt{-1}\th_0},
    \lambda_0,{1\over\lambda_0}\},
	\\
	\sigma(\gamma({1\over2}))&=&\{e^{\sqrt{-1}\th^*},e^{-\sqrt{-1}\th^*},
	e^{\sqrt{-1}\th^*},e^{-\sqrt{-1}\th^*}\},
	\end{eqnarray}
	where $0<\th^*<\th_0<\pi$ (the case $0<\th_0<\th^*<\pi$ is similar) and $\lambda_0>1$ (see Figure \ref{collide-out.path}).

 Note that in Figure \ref{collide-out.path},
 the pair of eigenvalues that represented
 by the red color
 travels clockwise.
 Thus, for $t\in[0,t_0)$,
 the normal form of
 $\gamma(t)$ is given by $\diag(R(2\pi-\th(t)),D(\lambda(t))$ where $\th(0)=\th_0,\lambda(0)=\lambda_0$ and
 $\th(t)$
 is the angle of the red pair of eigenvalues of $\gamma(t)$ lying in the upper half-plane.

\begin{lemma}\label{Th:move.collisions.on.U}
Using the notations above,
for any $\th\in(0,\th_0)$,
there exists an en elementary collide-out path $\mu$ which collides at
$\{e^{\sqrt{-1}\th},e^{-\sqrt{-1}\th}\}$,
such that $\gamma\sim_{+,con}\mu$ holds.
\end{lemma}
\begin{proof}
Let $\Theta \subseteq (0,\theta_0)$ be the set of angles at which an elementary collide-out path $\mu$ collides, and such that $\gamma \sim_{+, con}\mu$. Then $\theta^* \in \Theta$. For any $\theta \in \Theta$, according to Lemma \ref{lem:c0c1}(i), we can locally perturb the collision of the positive path on $\theta \in (\theta - \epsilon, \theta + \epsilon)$ for some $\epsilon > 0$; thus, $\Theta$ is an open subset of $(0,\theta_0)$. To prove $\Theta$ is the whole interval $(0,\theta_0)$, it is sufficient to show that $\Theta$ is also a closed subset of $(0,\theta_0)$.

Now suppose an elementary collide-out path $\xi$ which collide at $\th_\xi\in(0,\th_0)$
    can be positively homotopic to a familty of positive paths $\eta^{\theta}$ where two pairs of eigenvalues collide at $\th\in(\th_\xi-\alpha^*,\th_\xi+\alpha^*)$
    where $0<\alpha^*<\min\{\th_\xi,\th_0-\th_\xi\}$.
    Then we will prove that $\xi$ can be positively homotopic to a path $\eta$ where two pairs of eigenvalues collide at the angle $\th_\xi-\alpha^*$ (or $\th_\xi+\alpha^*$).
    First, we need to consider whether
    the elementary collide-out path that collide at $\th_\xi-\alpha^*$ exists.
    To finish the lemma, we need the
    following claim whose proof will be
    provided afterward.

{\bf  Claim}: There exists an elementary collide-out path with a single collision occurring at
eigenvalues
$\{e^{\sqrt{-1}\th},e^{-\sqrt{-1}\th}\}$
for $\th\in(0,\th_0)$,
such that the $C^1$-connectable condition
between such a path and $\xi$ holds.

Armed with this claim,
    suppose $\eta$ is an elementary collide-out path that connects some conjugate matrices of $\xi(0), \xi(1)$, respectively, with its only collision at $e^{\sqrt{-1}(\th_\xi-\alpha^*)}$
    such that $\xi$ and $\eta$ satisfy the
    $C^1$-connectable condition.
    By perturbation (see Lemma \ref{lem:c0c1}), there exists $\epsilon>0$ such that, for any $\alpha\in(\th_\xi-\alpha^*-\epsilon,\th_\xi-\alpha^*+\epsilon)$,
    $\eta$ can be homotopic to an elementary collide-out path that collides at $e^{\sqrt{-1}\alpha}$
    via a $C^1$-connectable positive homotopy.
    Fix an angle $\alpha_0\in(\th_\xi-\alpha^*,\th_\xi-\alpha^*+\epsilon)$,
    suppose $\tilde\eta$ is such a path that collides at $e^{\sqrt{-1}\alpha_0}$.
    
    On the other hand, $\xi$ can also be positively homotopic to a positive path $\tilde\xi$ that only collides at $e^{\sqrt{-1}\alpha_0}$.
   According to Lemma \ref{Th:positive.homotopy.of.two.elementary.path},
    $\tilde\xi$ and $\tilde\eta$ are positively homotopic.
    By transitivity, $\xi\sim_{+,con}\eta$.

    Finally, we have $\Theta=(0,\theta_0)$.
    In other words,
    we can move the collission
    to any angle in the global interval
    $(0,\th_0)$.
Therefore, we only need to proof the claim.

\begin{proof}[Proof of the Claim.]
For $\epsilon>0$ small enough,	$\mu(t)=N_2(\th,b)\cdot\diag(R(t-{1\over2}),R(t-{1\over2})$ is a positive path
	for $t\in[{1\over2}-\epsilon,{1\over2}+\epsilon]$.
	By $3^\circ$ of Lemma \ref{lemma:B_U.perturbation}, 
	$\Delta(\mu(t))>0$ when $t<{1\over2}$, 
    and $\Delta(\mu(t))<0$ when $t>{1\over2}$;
    and hence $\mu$ has type
	$\mathcal{O}_{\mathcal{U}}
	\rightarrow\mathcal{B}_{\mathcal{U}}
	\rightarrow\mathcal{O}_{\mathcal{C}}$.
	Since $\gamma(0)\in\mathcal{O}_{\mathcal{U,R}}$ and $\mu({1\over2}-\epsilon)\in\mathcal{O}_{\mathcal{U}}$,
    for $\epsilon>0$ small enough,
	there exists a positive path $\gamma_1(t),t\in[0,{1\over2}-\epsilon]$, which connects $X_1^{-1}\gamma(0)X_1$
	and $\mu({1\over2}-\epsilon)$ for some $X_1\in\Sp(4)$.
 In fact, let $X_0$ be a matrix such that $X_0\mu({1\over2}-\ep)X_0^{-1}=\diag(R(2\pi-\th_{1,\ep}),R(\th_{2,\ep}))$ for some $\th_{1,\ep}\in(\th,\th+\ep)$ and
 $\th_{2,\ep}\in(\th-\ep,\th)$
 where $\ep<\min\{\th,\th_0-\th\}$ is sufficiently small.
 Furthermore, let
 $\th(t)$ be a strictly decreasing function with respect to $t$ such that
 $\th(0)=\th_0,\th({1\over2}-\ep)=\th_{1,\ep}<\th_0$
 and $\gamma_h$ be a positive path in $\Sp(2)$ that connects some matrix that conjugates to $D(\lambda_0)$ and
 $R(\th_{2,\ep})$ itself.
 Then $\gamma_1(t)=X_0^{-1}\diag(R(2\pi-\th(t),\gamma_h(t)))X_0$ for $t\in[0,{1\over2}-\ep]$.
 Furthermore, according to the Smoothing Lemma (Lemma \ref{lem: smoothing}),
 with appropriate perturbations near the end-points of $\gamma_1$,
 we can assume
	\begin{eqnarray}
	\gamma_1'(0)&=&X_1^{-1}\gamma'(0)X_1,
	\\
	\gamma_1'({1\over2}-\epsilon)&=&\mu'({1\over2}-\epsilon).
	\end{eqnarray}	
	By a similar argument, there exists a positive path $\gamma_2(t),t\in[{1\over2}+\epsilon,1]$, which connects 
	$\mu({1\over2}+\epsilon)$ and $X_2^{-1}\gamma(1)X_2$ for some $X_2\in\Sp(4)$, such that
	\begin{eqnarray}		\gamma_2'({1\over2}+\epsilon)&=&\mu'({1\over2}+\epsilon),
	\\
	\gamma_2'(1)&=&X_2^{-1}\gamma'(1)X_2.
	\\
	\end{eqnarray}
    Here $\gamma_2$ is entirely in $\mathcal{O}_{\mathcal{C}}$.
 
	Now $\tilde\gamma=\gamma_1*\mu*\gamma_2$ is a positive path
	with the type $\mathcal{O}_{\mathcal{U,R}}
	\rightarrow\mathcal{B}_{\mathcal{U}}
	\rightarrow\mathcal{O}_{\mathcal{C}}$.
	Moreover, we have $\tilde\gamma({1\over2})=\mu({1\over2})=N_2(\th,b)$,
	and
	\begin{eqnarray}	
	\tilde\gamma(0)=X_1^{-1}\gamma(0)X_1,&\quad&
	\tilde\gamma'(0)=X_1^{-1}\gamma'(0)X_1,
	\\
	\tilde\gamma(1)=X_2^{-1}\gamma(1)X_2,&\quad&
	\tilde\gamma'(1)=X_2^{-1}\gamma'(1)X_2.
	\end{eqnarray}
Thus $\tilde\gamma$ is the required
positive path.  
\end{proof}
We have completed the proof of Lemma \ref{Th:move.collisions.on.U}.
\end{proof}

\subsection{Move the collisions from the unit circle $\U$ to the real line $\R$}

The following theorem demonstrates that, for a given elementary collide-in (or collide-out) path, its single collision can be moved from $\U\backslash\{\pm1\}$ to $\R\backslash\{0,\pm1\}$ through a $C^1$-connectable positive homotopy.


\begin{theorem}\label{Th:transfer.collisions}
Any elementary collide-out (collide-in) path can be positively homotopic to a positive path with only one collision on the real line via a $C^1$-connectable positive homotopy.
\end{theorem}

\begin{proof}[Proof of Theorem \ref{Th:transfer.collisions}]

The proof will come in several steps.

{\bf  Step 1: Construct a positive path with a single collision occurring at the eigenvalue $\pm1$}

    We only consider the case when one pair of eigenvalues
    pass through $\{1,1\}$,
    and the other case is similar.
 Such a construction is similar to that of
 the claim in the proof of Lemma \ref{Th:move.collisions.on.U}.

	For $\epsilon>0$ small enough,	$\mu(t)=M_2(1,c)\cdot\diag(R(t-{1\over2}),R(t-{1\over2})$ is a positive path
	for $t\in[{1\over2}-\epsilon,{1\over2}+\epsilon]$.
	By $1^\circ$ of Lemma \ref{lemma:B_1.perturbation}, 
	$\Delta(\mu(t))>0$ when $t<{1\over2}$, 
    and $\Delta(\mu(t))<0$ when $t>{1\over2}$;
    and hence $\mu$ has type
	$\mathcal{O}_{\mathcal{U,R}}
	\rightarrow\mathcal{B}_{\pm1}
	\rightarrow\mathcal{O}_{\mathcal{C}}$.
	Since $\gamma(0),\mu({1\over2}-\epsilon)\in\mathcal{O}_{\mathcal{U,R}}$,
    for $\epsilon>0$ small enough,
    there exists a positive path $\gamma_1(t),t\in[0,{1\over2}-\epsilon]$, 
    entirely in $\mathcal{O}_{\mathcal{U,R}}$,
    which connects $X_1^{-1}\gamma(0)X_1$
	and $\mu({1\over2}-\epsilon)$ for some $X_1\in\Sp(4)$.
 Furthermore, according to the Smoothing Lemma (Lemma \ref{lem: smoothing}),
 with appropriate perturbations near the end-points of $\gamma_1$,
 we can assume
	\begin{eqnarray*}
	\gamma_1'(0)=X_1^{-1}\gamma'(0)X_1,\qquad	\gamma_1'({1\over2}-\epsilon)=\mu'({1\over2}-\epsilon).
	\end{eqnarray*}	
Meanwhile, there exists a positive path $\gamma_2(t),t\in[{1\over2}+\epsilon,1]$, 
entirely in $\mathcal{O}_{\mathcal{C}}$,
which connects 
	$\mu({1\over2}+\epsilon)$ and $X_2^{-1}\gamma(1)X_2$ for some $X_2\in\Sp(2)$, such that
	\begin{eqnarray*}		\gamma_2'({1\over2}+\epsilon)=\mu'({1\over2}+\epsilon),\qquad
	\gamma_2'(1)=X_2^{-1}\gamma'(1)X_2.
	\end{eqnarray*}
 
	Now $\tilde\gamma=\gamma_1*\mu*\gamma_2$ is a positive path
	with the type $\mathcal{O}_{\mathcal{U,R}}
	\rightarrow\mathcal{B}_{\pm1}
	\rightarrow\mathcal{O}_{\mathcal{C}}$.
	Moreover, we have $\tilde\gamma({1\over2})=\mu({1\over2})=M_2(1,c)$,
	and
 	\begin{eqnarray*}	
	\qquad\tilde\gamma(0)=X_1^{-1}\gamma(0)X_1,\,
	\tilde\gamma'(0)=X_1^{-1}\gamma'(0)X_1;\quad
	\tilde\gamma(1)=X_2^{-1}\gamma(1)X_2,\,
	\tilde\gamma'(1)=X_2^{-1}\gamma'(1)X_2.
	\end{eqnarray*}
	 Thus $\tilde\gamma$ is the required
	 positive path.
	
	{\bf  Step 2: Perturb the collision at $\{1,1,1,1\}$}

   Referring to Lemma \ref{bifurcation.digram.of.c=1.2} and the subsequent argument, the bifurcation diagram for general deformations of $M_2(1, c)$ is illustrated in Figure \ref{figure:bf_c2} (d). Additional insights can be gained from Figure \ref{bifurcation_diagram_of_zero4}.
    
    \begin{figure}[ht]
	\centering
	\includegraphics[height=10.5cm]{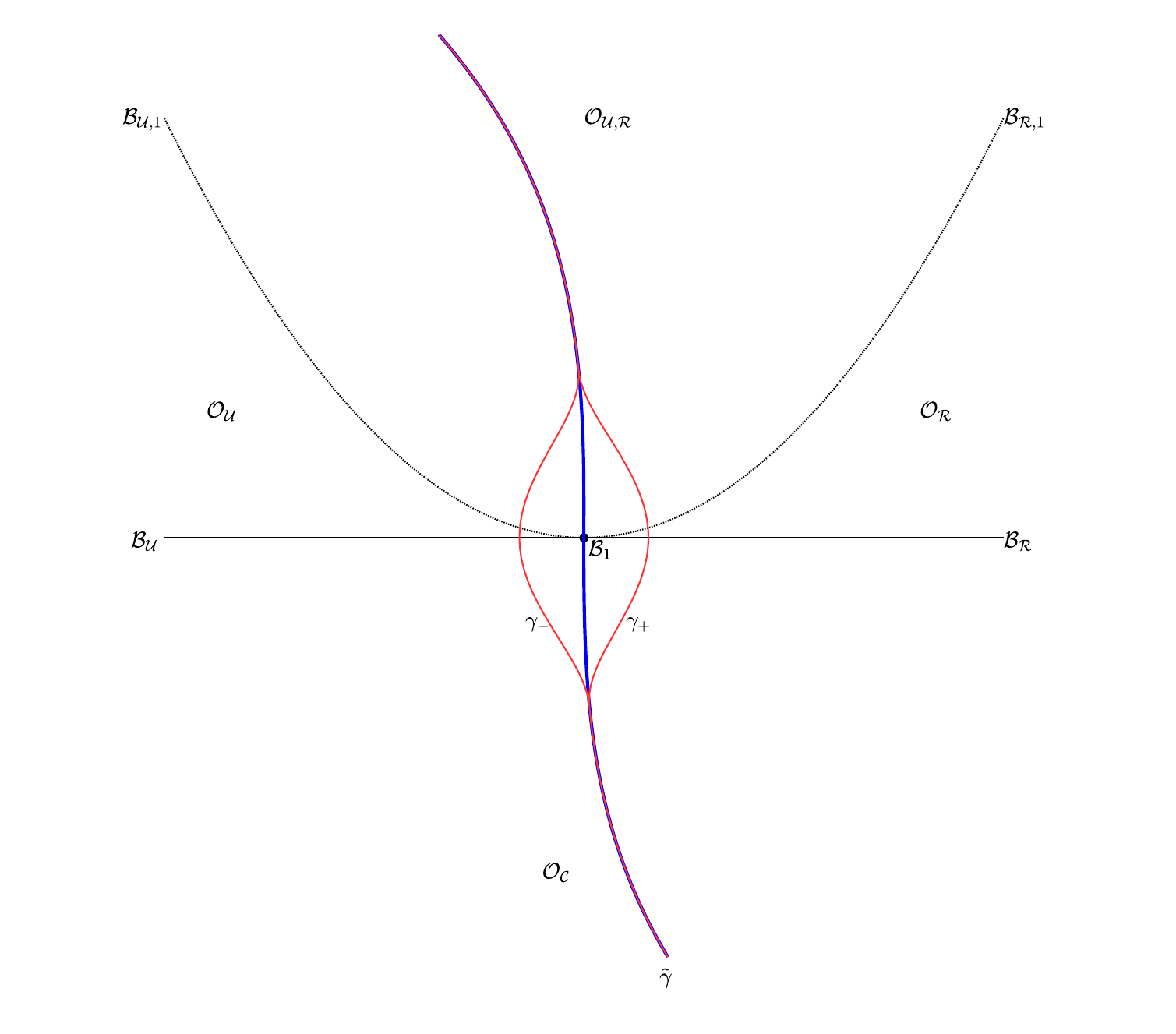}
	\vskip -0.8 cm
	\caption{ Perturb the path near $\mathcal{B}_1$.}
    \label{bifurcation_diagram_of_zero4}
    \end{figure}
	Then we perturb $\tilde\gamma$ near the point $M_2(1,c)\in\mathcal{B}_1$, as depicted in Figure \ref{bifurcation_diagram_of_zero4}.
    Here $\gamma_+$ has only one collision on $\R\backslash\{\pm1\}$,
    while $\gamma_-$ has only one collision on $\U\backslash\{\pm1\}$.
    Note that $\gamma_+,\gamma_-$, and $\tilde\gamma$
    are positive homotopy with fixed end-segments (see Formula (\ref{fix.condition}) and below),
    which implies that $\gamma_+\sim_{+,con}\tilde\gamma\sim_{+,con}\gamma_-$.

  {\bf Step 3: Move the collision from $\U$ to $\R$}

    Note that
	the positive path $\gamma_-$ in Step 2 has the same type as $\gamma$.
    Using Lemma \ref{Th:move.collisions.on.U} to $\gamma_-$, 
    there exists a positive path $\hat{\gamma}$ such that $\tilde\gamma\sim_{+,con}\hat\gamma$;
    and $\gamma,\hat\gamma$ have the same conjugate collision points.
    Then by Lemma \ref{Th:positive.homotopy.of.two.elementary.path},
    $\gamma\sim_{+,con}\hat\gamma$;
    and hence $\gamma\sim_{+,con}\gamma_+$.
    Note that $\gamma_+$ is a positive path with a single collision
    on $\R$.
    Thus, we move the collision from $\U$ to $\R$ via a proper 
    $C^1$-connectable positive homotopy.
\end{proof}

For a positive path $\gamma$, if one pair of eigenvalues travels as 
$\R^+\rightarrow\{1,1\}\rightarrow\U$,
another pair of eigenvalues comes from $\R^-\rightarrow\{-1,-1\}\rightarrow\U$,
and then they collide at $\mathcal{B}_{\mathcal{U}}$,
moving from $\mathcal{O}_{\mathcal{U}}$ to $\mathcal{O}_{\mathcal{C}}$.
Hence, we can not only move the collision to $\R^+$, but also to $\R^-$.
From another point of view, for such a path,
we can move the collision from $\R^+$ to $\R^-$ (or from $\R^-$ to $\R^+$) through $\U$.
For more precise discussions,
we need to determine whether the eigenvalues in 
$\R\backslash\{0\}$ belongs to $\R^+$ or $\R^-$.
For this purpose, we utilize slightly different symbols to denote the subsets of
the regions as listed in Section 5.1.
For example, $\mathcal{O}_{\mathcal{R^-,R^+}}$ represent
the subset of $\mathcal{O}_{\mathcal{R}}$,
which consists of all matrices with one pair
of eigenvalues in $\R^+$ and another pair of eigenvalues in $\R^-$; 
$\mathcal{B}_{\mathcal{R}^-,1}$ represents
the subset of $\mathcal{B}_{\mathcal{R},1}$,
which consists of all matrices with one pair
of eigenvalues in $\R^-$ and another pair of eigenvalues being $\{1,1\}$; 
and so on.
Indeed, we have

\begin{lemma}\label{Th:transfer.R+.to.R-}
(i) Suppose the positive path $\gamma$ travels as follows: 
	\begin{equation}
	\mathcal{O}_{\mathcal{R^-,R^+}}
	  \rightarrow\mathcal{B}_{\mathcal{R^-},1}
    \rightarrow\mathcal{O}_{\mathcal{R^-,U}}
      \rightarrow\mathcal{B}_{-1,\mathcal{U}}
    \rightarrow\mathcal{O}_{\mathcal{U}}
	\rightarrow\mathcal{B}_{\mathcal{U}}
	\rightarrow\mathcal{O}_{\mathcal{C}},
	\end{equation}
see Figure \ref{Sp4_type2}.
Then $\gamma$ can be positively homotopic to a positive path
that has only one collision on $\R^+$, as depicted in Figure \ref{Sp4_type3}.
It can also be positively homotopic to a positive path
that has only one collision on $\R^-$,
see Figure \ref{Sp4_type4}.

(ii) Any positive path travels as follows:
	\begin{equation}
	\mathcal{O}_{\mathcal{R^+,R^-}}
	  \rightarrow\mathcal{B}_{1,\mathcal{R^-}}
    \rightarrow\mathcal{O}_{\mathcal{U,R^-}}
      \rightarrow\mathcal{B}_{-1,\mathcal{R^-}}
    \rightarrow\mathcal{O}_{\mathcal{R^-,R^-}}
	\rightarrow\mathcal{B}_{\mathcal{R}}
	\rightarrow\mathcal{O}_{\mathcal{C}}
	\end{equation}
 see Figure \ref{Sp4_type3},
 can be positively homotopic to a positive path with type
 \begin{equation}
	\mathcal{O}_{\mathcal{R^+,R^-}}
	  \rightarrow\mathcal{B}_{\mathcal{R^+},-1}
    \rightarrow\mathcal{O}_{\mathcal{R^+,U}}
      \rightarrow\mathcal{B}_{\mathcal{R^+},1}
    \rightarrow\mathcal{O}_{\mathcal{R^+,R^+}}
	\rightarrow\mathcal{B}_{\mathcal{R}}
	\rightarrow\mathcal{O}_{\mathcal{C}},
	\end{equation}
 see Figure \ref{Sp4_type4}.
\end{lemma}

\begin{figure}[ht]
	\centering
	\includegraphics[height=6.5cm]{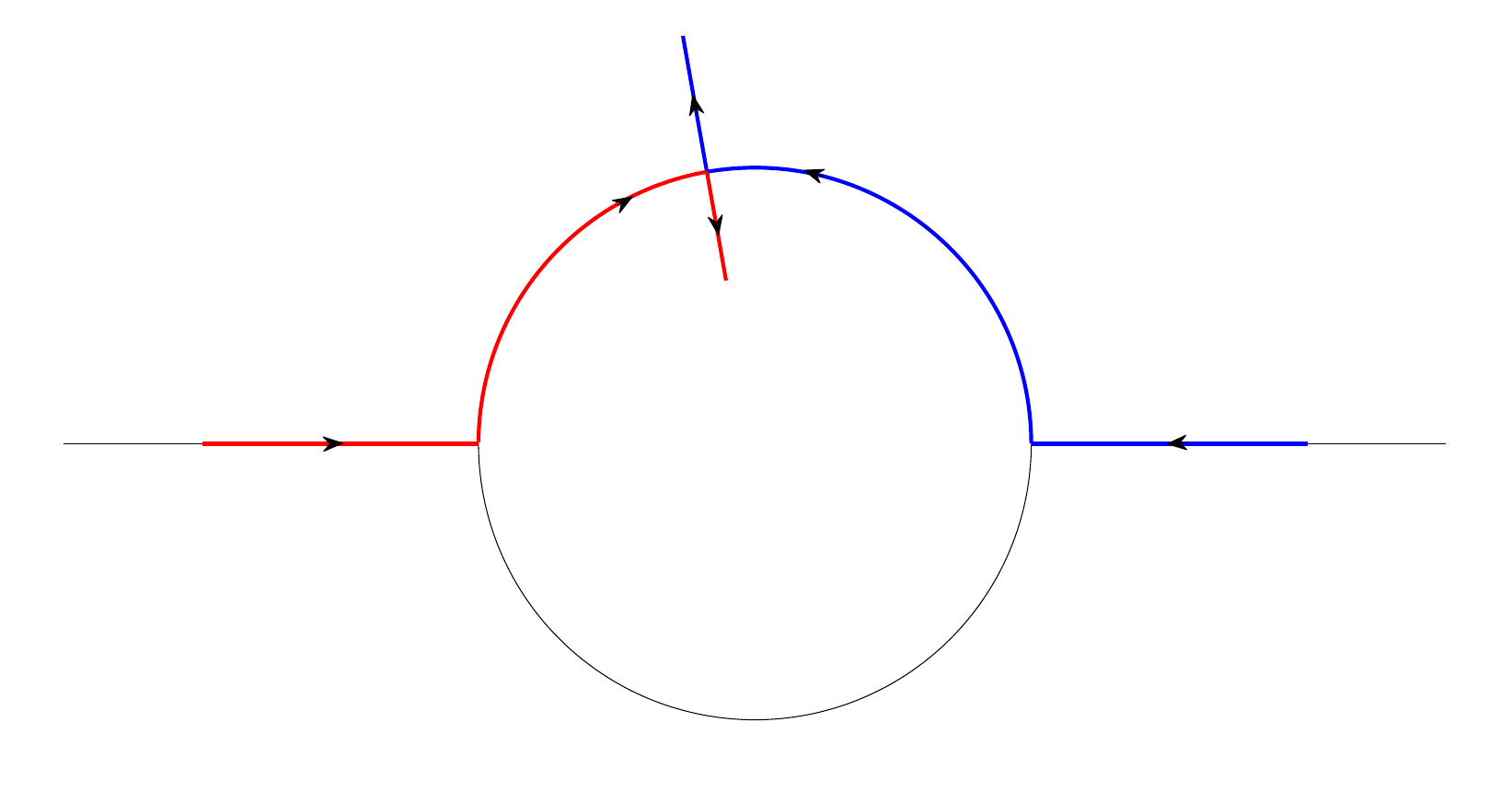}
	\vskip -0.8 cm
	\caption{A path with type $\mathcal{O}_{\mathcal{R^-,R^+}}
	  \rightarrow\mathcal{B}_{\mathcal{R^-},1}
    \rightarrow\mathcal{O}_{\mathcal{R^-,U}}
      \rightarrow\mathcal{B}_{-1,\mathcal{U}}
    \rightarrow\mathcal{O}_{\mathcal{U}}
	\rightarrow\mathcal{B}_{\mathcal{U}}
	\rightarrow\mathcal{O}_{\mathcal{C}}$. The red curve represents a pair of eigenvalues that goes to $N_{-1}^-$. The blue curve represents the another pair of eigenvalues that goes to $N_{1}^-$.}
    \label{Sp4_type2}
\end{figure}

\begin{figure}[ht]
	\centering
	\includegraphics[height=6.0cm]{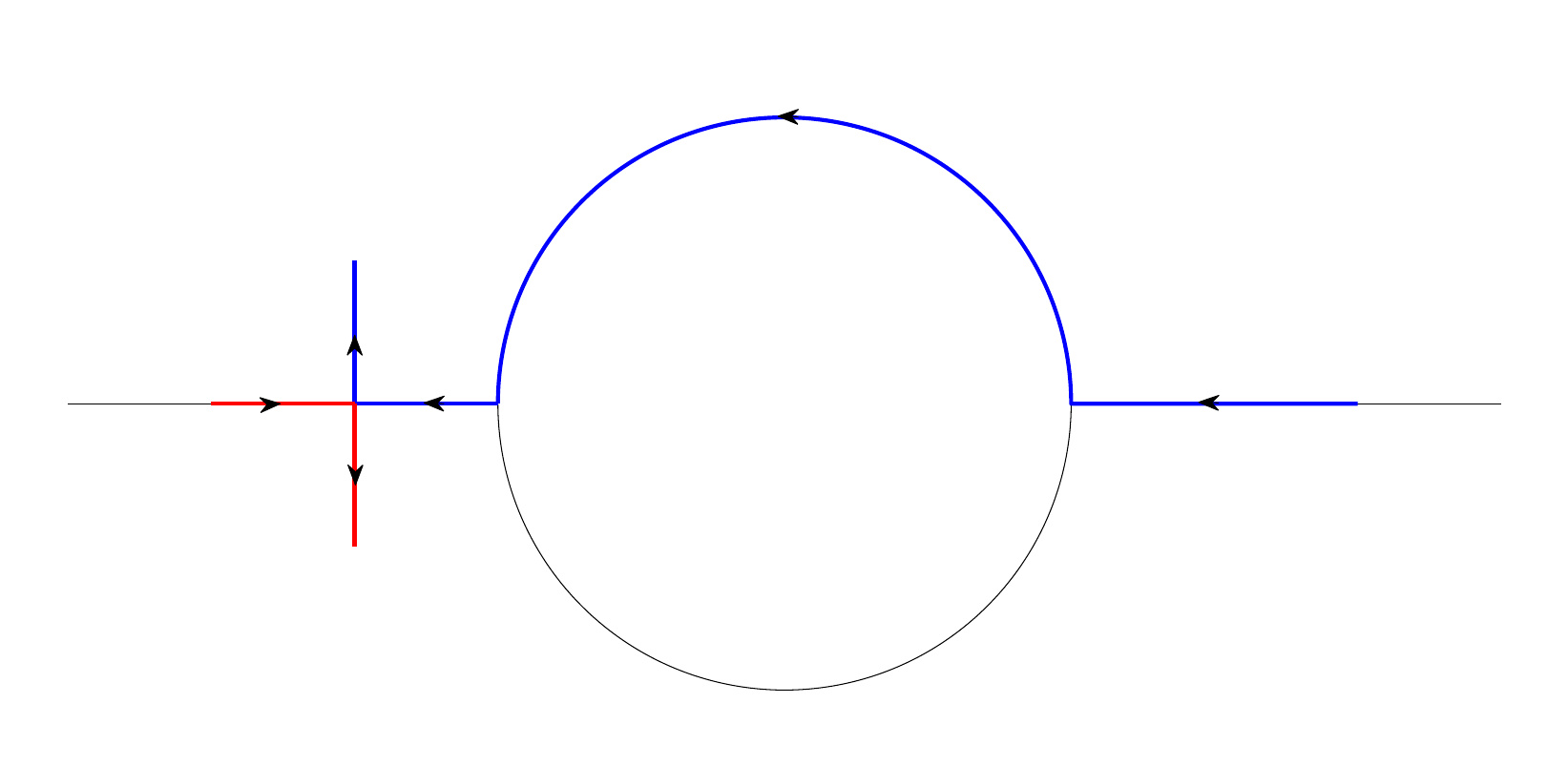}
	\vskip -0.8 cm
	\caption{A path with type $\mathcal{O}_{\mathcal{R^+,R^-}}
	  \rightarrow\mathcal{B}_{1,\mathcal{R^-}}
    \rightarrow\mathcal{O}_{\mathcal{U,R^-}}
      \rightarrow\mathcal{B}_{-1,\mathcal{R^-}}
    \rightarrow\mathcal{O}_{\mathcal{R^-,R^-}}
	\rightarrow\mathcal{B}_{\mathcal{R}}
	\rightarrow\mathcal{O}_{\mathcal{C}}$. Each curve
    with distinct colors represents the trajectory of
    a pair of eigenvalues.}
    \label{Sp4_type3}
\end{figure}

\begin{figure}[ht]
	\centering
	\includegraphics[height=6.0cm]{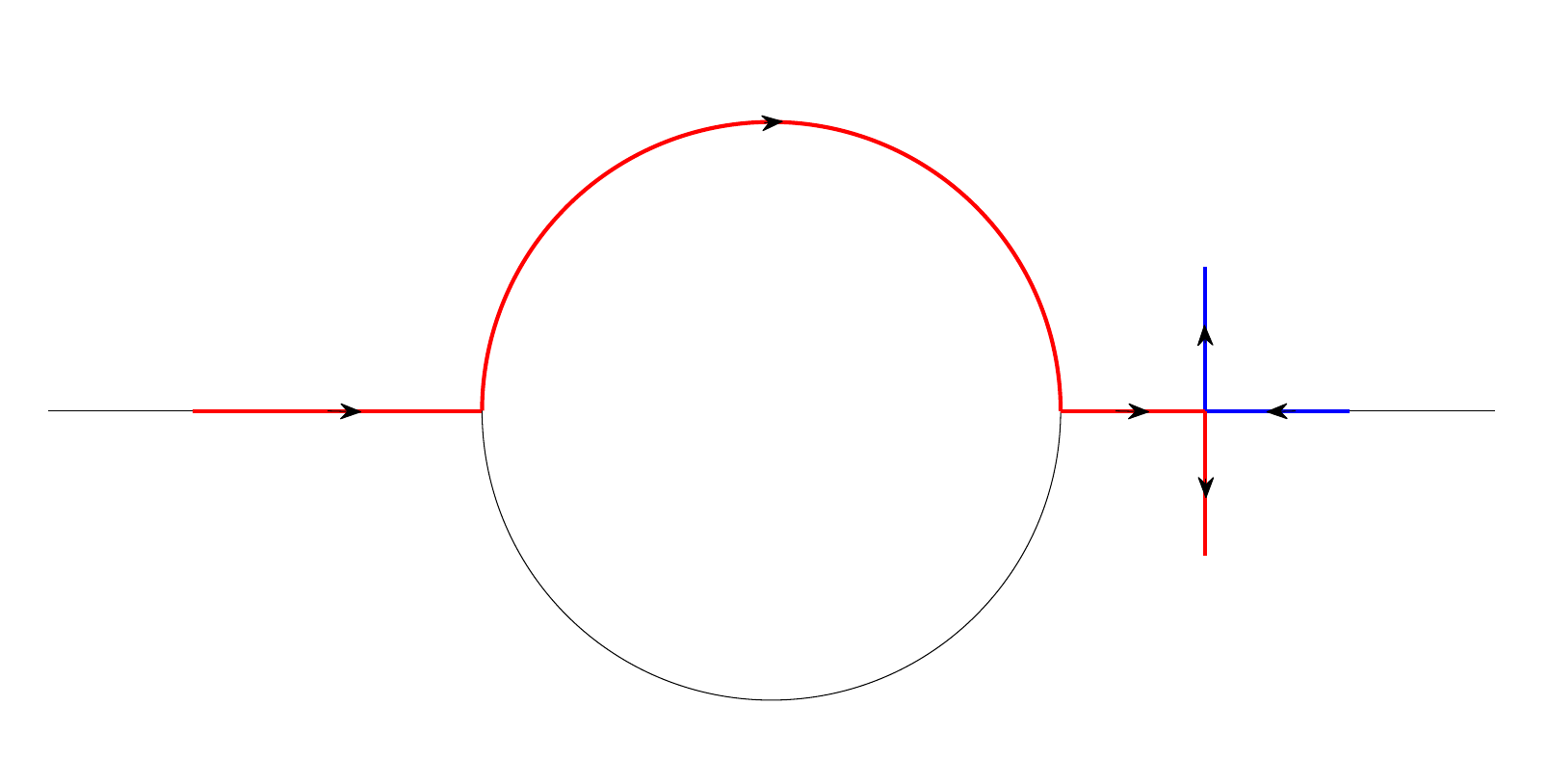}
	\vskip -0.8 cm
	\caption{A path with type $\mathcal{O}_{\mathcal{R^+,R^-}}
	  \rightarrow\mathcal{B}_{\mathcal{R^+},-1}
    \rightarrow\mathcal{O}_{\mathcal{R^+,U}}
      \rightarrow\mathcal{B}_{\mathcal{R^+},1}
    \rightarrow\mathcal{O}_{\mathcal{R^+,R^+}}
	\rightarrow\mathcal{B}_{\mathcal{R}}
	\rightarrow\mathcal{O}_{\mathcal{C}}$. Each curve
    with different colors represents the trajectory of
    a pair of eigenvalues.}
    \label{Sp4_type4}
\end{figure}

\begin{proof}
The proof of the statement (i) is similar to the proof of Theorem \ref{Th:transfer.collisions},
and (ii) immediately follows from (i).
\end{proof}

\begin{remark}\label{RM:transfer.R+.to.R-}
    For a positive path in $\Sp(4)$,
    if one pair of eigenvalues remain on $\R^-$ while another pair of eigenvalues travels from $\R^+$ through $\U$ to $\R^-$,
    they collide at $\R^-$.
    This collision can be transferred through $\U$ to $\R^+$ via positive homotopy.
\end{remark}

According to Theorem \ref{Th:transfer.collisions}, we can move the collision
from $\U\backslash\{\pm1\}$ to $\R\backslash\{\pm1\}$.
Furthermore, once the collision is on $\R\backslash\{\pm1\}$,
the next lemma shows that
the collision can freely moved along $\R^+\backslash\{1\}$ or $\R^-\backslash\{-1\}$.
In fact, by similar arguments to those in the proof of Lemma \ref{Th:move.collisions.on.U}, along with $4^\circ$ of Lemma 2.5, we have:

\begin{lemma}\lb{Th:transfer.collisions.on.R}
Suppose $\gamma$ is a positive path that has only one collision from
$\mathcal{O}_{\mathcal{R}^+,\mathcal{R}^+}$ (or $\mathcal{O}_{\mathcal{R}^-,\mathcal{R}^-}$) 
	to $\mathcal{O}_{\mathcal{C}}$.
	Then, for any $\lambda>1$ (resp. $\lambda<-1$), there exists a positive path
	$\mu$ that has only one collision at $\lambda$,
	such that $\gamma\sim_{+,con}\mu$.
\end{lemma}

\subsection{Eliminate the collisions on the real line}

According to Theorem \ref{Th:transfer.collisions},
if a positive path contains collisions on the unit circle, 
it can be positively homotopic to a positive path that only
has collisions on the real line.
Therefore, we only consider positive paths with possible collisions on the real line.

Let $\gamma$ be a positive loop that starts at $I$.
Suppose the eigenvalues first collide at $t=t_1$,
then they must collide at $\R^+$ (or $\R^-$) from $\mathcal{O}_\mathcal{R}$ to $\mathcal{O}_\mathcal{C}$.
Since the loop must return back to $I$,
there must be a subsequent collision
from $\mathcal{O}_\mathcal{C}$ to $\mathcal{O}_\mathcal{R}$.
Now, let's assume they successively collide at $t=t_2$,
which must collide from $\mathcal{O}_\mathcal{C}$ to $\mathcal{O}_\mathcal{R}$.
We have only four cases:

(i) The path starts from $\mathcal{O}_{\mathcal{R}^+,\mathcal{R}^+}$ \footnote{In \cite{Sli}, Slimowitz used the notation $\mathcal{O}_\mathcal{R}^+$, the notations below are similar},
enters $\mathcal{O}_\mathcal{C}$, and then returns to $\mathcal{O}_{\mathcal{R}^+,\mathcal{R}^+}$;

(ii) The path starts from $\mathcal{O}_{\mathcal{R}^-,\mathcal{R}^-}$,
enters $\mathcal{O}_\mathcal{C}$, and then returns to $\mathcal{O}_{\mathcal{R}^-,\mathcal{R}^-}$;

(iii) The path starts from $\mathcal{O}_{\mathcal{R}^+,\mathcal{R}^+}$,
enters $\mathcal{O}_\mathcal{C}$, and then returns to $\mathcal{O}_{\mathcal{R}^-,\mathcal{R}^-}$;

(iv) The path starts from $\mathcal{O}_{\mathcal{R}^-,\mathcal{R}^-}$,
enters $\mathcal{O}_\mathcal{C}$, and then returns to $\mathcal{O}_{\mathcal{R}^+,\mathcal{R}^+}$.

We will demonstrate that the two adjacent collisions can be simultaneously removed through a positive homotopy.

\medskip

We first have

\begin{lemma}\label{Lemma:Sp4.case1&2}
    Suppose $\gamma$ is a positive truly hyperbolic path that connects two points $A_0,A_1\in\mathcal{O}_\mathcal{R^+,R^+}$ (or $A_0,A_1\in\mathcal{O}_\mathcal{R^-,R^-}$).
    Then there exists another positive truly hyperbolic path $\mu$ entirely within $\mathcal{O}_\mathcal{R^+,R^+}$ (resp. $\mathcal{O}_\mathcal{R^-,R^-}$)
    such that the $C^1$-connectable conditions between
    $\gamma$ and $\mu$ hold.
    Moreover, we have $\gamma\sim_{+,con}\mu$.
\end{lemma}
\begin{proof}
First, we note that $\mathcal{O}_\mathcal{R^+,R^+}$ and $\mathcal{O}_\mathcal{R^-,R^-}$ are path connected. Indeed, for any $M_1,M_2\in\mathcal{O}_\mathcal{R^+,R^+}$ (resp. $\mathcal{O}_\mathcal{R^-,R^-}$), they
can be represent as
\begin{equation}
    M_1=P_1^{-1}\diag(D(a_1),D(b_1))P_1,
    \quad
    M_2=P_2^{-1}\diag(D(a_2),D(b_2))P_2,
\end{equation}
where $a_1,b_1,a_2,b_2\in\R^+$ (resp. $\R^-$)
and $P_1,P_2\in\Sp(4)$.
Since $D(\lambda_1)$ can be connected to $D(\lambda_2)$ if $\lambda_1\lambda_2>0$,
together with the path connection of $\Sp(4)$,
$M_1$ can be connected to $M_2$.

Due to the path connection of $\mathcal{O}_\mathcal{R^+,R^+}$ (or $\mathcal{O}_\mathcal{R^-,R^-}$),
there exists a $C^1$ path (not necessarily positive) entirely within $\mathcal{O}_\mathcal{R^+,R^+}$,
which connects $A_0$ and $A_1$.
Furthermore, we perturb such a path
so that it coincides with $\gamma$ when
$t\in[0,\ep)\cup(1-\ep,1]$ for some $\ep>0$.
Then
    the existence of $\mu$ is guaranteed by Lemma \ref{Lemma:hyperbolic.positive.path},
    and $\gamma\sim_{+,con}\mu$ follows from Theorem \ref{Thm:remove.hyperbolic.collisions}.
\end{proof}

According to Cases (i)-(iv), we have


\begin{theorem}\label{Thm.remove.collisions.on.R.of.Sp4}
	Given a positive loop $\gamma$ in
 $\Sp(4)$ which only 
	has collisions on  $\R\backslash\{0,\pm1\}$,
	it can be homotopic to a positive loop 
    with no collisions.
\end{theorem}

\begin{proof}

For Cases (i) and (ii), we can eliminate the two collisions by Lemma \ref{Lemma:Sp4.case1&2},
through a $C^1$-connectable positive homotopy.
Consequently, we only need to address Cases (iii) and (iv).
		
For Case (iii),
we have $i_1(\ga|_{[0,t_1-\epsilon]})=2(k_1+k_2)$ 
where $k_1, k_2\ge1$ are the winding numbers about $\U$
for the two pairs of eigenvalues, respectively.
Since $\ga|_{[0,t_1-\epsilon]}$ has no collisions,
we can use Lemma \ref{Lemma:change.part.time},
to change the speed of the two pairs of eigenvalues.
This allows $\gamma|_{[0,t_1-\epsilon]}$ to be positively homotopic to $\mu|_{[0,t_1-\epsilon]}$ such that,
for $t\in[0,t_1-2\epsilon]$, one pair of eigenvalues of $\mu$ rotates $k_1-{1\over2}$ rounds to $\R^-$,
and another pair of eigenvalues of $\mu$ rotates $k_2$ rounds to $\R^+$;
for $t\in[t_1-2\epsilon,t_1-{\epsilon}]$,
the first pair of eigenvalues of $\mu$ rotates ${1\over2}$ rounds from $\R^-$ to $\R^+$,
while the second pair of eigenvalues of $\mu$ remains on $\R^+$.
Then we glue the two sub-paths $\mu|_{[0,t_1-\epsilon]}$
	and $\gamma|_{[t_1-\epsilon,1]}$ together to form a new positive path $\mu$.

 Now the path $\mu|_{[t_1-2{\epsilon},t_1+\epsilon]}$ has the type
 \begin{equation}
	\mathcal{O}_{\mathcal{R^+,R^-}}
	  \rightarrow\mathcal{B}_{\mathcal{R^+},-1}
    \rightarrow\mathcal{O}_{\mathcal{R^+,U}}
      \rightarrow\mathcal{B}_{\mathcal{R^+},1}
    \rightarrow\mathcal{O}_{\mathcal{R^+,R^+}}
	\rightarrow\mathcal{B}_{\mathcal{R}}
	\rightarrow\mathcal{O}_{\mathcal{C}}.
	\end{equation}
 By applying Lemma \ref{Th:transfer.R+.to.R-},
 we can transfer the collision from $\R^+$ to $\R^-$.
Consequently, $\mu$ satisfies our Case (ii), allowing us to remove the two subsequent collisions.

    For Case (iv),
	we have $i_1(\ga|_{[0,t_1-\epsilon]})=2(k_1+k_2+1)$ 
	where $k_1+{1\over2}, k_2+{1\over2}$ are the winding number about $\U$
 for the two pairs of eigenvalues, respectively.	
If $\gamma$ has only two collisions,
    and since $\gamma(0)=I_4$, for $t\in[t_2+\epsilon,1]$, one pair of eigenvalues
    travels from $\R^+$ to $\U$, and then rotates $k_3$ rounds to
    $\{1,1\}$,
    while another pair of eigenvalues
    also travels from $\R^+$ to $\U$, and then rotates $k_4$ rounds
    to $\{1,1\}$.
    Here $k_3$ and $k_4$ are positive integers (see \cite[Lemma 3.1 (i)]{LaM} or \cite[pp.16, Proposition 1]{Eke}).
    Under a similar argument as in Case (iii), we can move the collision
    at time $t_2$ from $\R^+$
	to $\R^-$ through $\U$,
    and hence it can be removed together with the previous collision at time $t_1$.
 
	If there exists a third collision,
	it must be one of the following two sub-cases:
	
	i) The path enters $\mathcal{O}_\mathcal{C}$ from 
	$\mathcal{O}_{\mathcal{R^+,R^+}}$ at $t_3>t_2$ with no rotations on $[t_2,t_3]$;
	
	ii) The path rotates several (and half) rounds, and then
	enters $\mathcal{O}_\mathcal{C}$ from 
	$\mathcal{O}_{\mathcal{R^+,R^+}}$ (resp. $\mathcal{O}_{\mathcal{R^-,R^-}}$) at $t_3>t_2$.
	
For Subcase i), we can use Lemma \ref{Lemma:Sp4.case1&2} to remove the two collisions
	at $t=t_2$ and $t=t_3$;
	for Subcase ii), by utilizing Lemma \ref{Th:transfer.R+.to.R-},
	Theorem \ref{Th:transfer.collisions.on.R}
    and a similar argument to Case (iii),
	we can transfer the collision at $t=t_2$ from $\R^+$
	to $\U$, and then to $\R^-$.
	Now the new positive path travels from $\mathcal{O}_{\mathcal{R^-,R^-}}$,
	enters $\mathcal{O}_\mathcal{C}$ at $t=t_1$, and then returns to $\mathcal{O}_{\mathcal{R^-,R^-}}$ at $t=t_2$.
	Therefore, we have converted the path to Case (ii).
	
	Consequently, we can remove all collisions of $\gamma$ by induction.
\end{proof}

Now we can give the proof of our main theorem for $\Sp(4)$:

\begin{proof}[Proof of Theorem A for $\Sp(4)$]
According to Theorem \ref{Th:transfer.collisions} and induction, if a positive path includes collisions on the unit circle, it can be positively homotopic to another positive path that only has collisions on the real line. Subsequently, applying Theorem \ref{Thm.remove.collisions.on.R.of.Sp4} and Theorem \ref{Thm:pos.homotopy.of.loops.has.no.collisions}, we can conclude that the case $n=2$ of Theorem A.
\end{proof}

\setcounter{equation}{0}
\section{Positive paths in $\Sp(2n)$ for $n\ge3$}
\label{sec:6}

\subsection{Remove the collisions 
of the elementary two double-collisions paths}
\label{subsec:6.1}

Since $\mathcal{C}onj(\Sp(2n))$ is substantially more complicated,
we will not provide a complete classification of $\mathcal{C}onj(\Sp(2n))$.
Below we present some regions of $\mathcal{C}onj(\Sp(6))$ that we will used later:

\hangafter 1
\hangindent 3.5em
\;\;(i) $\mathcal{O}_{\mathcal{C,U}}$,
consisting of all matrices with $4$ distinct eigenvalues in
$\C\backslash(\U\cup\R)$, and a conjugate pair of eigenvalues on
$\U\backslash\{\pm1\}$;

\hangafter 1
\hangindent 3.5em
\;(ii) $\mathcal{O}_{\mathcal{C,R}}$,
consisting of all matrices with $4$ distinct eigenvalues in
$\C\backslash(\U\cup\R)$, and a pair of eigenvalues on
$\R\backslash\{\pm1\}$;

\hangafter 1
\hangindent 3.5em
(iii) $\mathcal{O}_{\mathcal{U}}$,
consisting of all matrices with eigenvalues on $\U\backslash\{\pm1\}$
where each eigenvalues has multiplicity 1 (or multiplicity 2 with splitting number $\pm2$, or multiplicity 3 with splitting number $\pm3$);

\hangafter 1
\hangindent 3.5em
\;(iv) $\mathcal{B}_{\mathcal{U}}^2$,
consisting of all matrices with a conjugate pair of eigenvalues on $\U\backslash\{\pm1\}$ 
which has multiplicity 2
with only one Jordan block,
and another conjugate pair of eigenvalues on
$\U\backslash\{\pm1\}$ which has multiplicity 1;

\hangafter 1
\hangindent 3.5em
\;\;(v) $\mathcal{B}_{\mathcal{U}}^3$,
consisting of all matrices with a conjugate pair of eigenvalues which has multiplicity 3
and has only one Jordan block.

\noindent Moreover, we define
$\mathcal{B}_{\mathcal{U}}=\mathcal{B}_{\mathcal{U}}^2\cup\mathcal{B}_{\mathcal{U}}^3$. 
\medskip

For the collisions of a generic positive path in $\Sp(2n)$, where $n \geq 3$, we do not necessarily encounter the constraints outlined in Lemma \ref{toplogical.constraint.of.Sp4}. Instead, we have an additional situation:

\begin{lemma}
	Let $\gamma$ be a generic positive path in $\Sp(2n)$. If a quadruplet of eigenvalues of $\gamma$ exits $\U\backslash\{\pm1\}$ and enters $\C\backslash(\U\cup\R)$ at time $t_2$, denoted by ${\lambda_1,{1\over\lambda_1}}$ and ${\lambda_2,{1\over\lambda_2}}$, then one of the following cases must occur:
	
	(i) after the recent preceding double-collision with $\{\lambda_i,{1\over\lambda_i}\},i=1,2$
	 of $\gamma$, one pair of eigenvalues, denoted as $\{\lambda_1,{1\over\lambda_1}\}$, 
	 enters from $\{1,1\}$ or $\{-1,-1\}$ at time $t_1 < t_2$;
	
	(ii) one pair of eigenvalues, denoted as $\{\lambda_2,{1\over\lambda_2}\}$,
	collide with a third pair of eigenvalues $\{\lambda_3,{1\over\lambda_3}\}$ of $\gamma$
	at time $t_1<t_2$ from $\C\backslash(\U\cup\R)$,
	and  then $\{\lambda_i,{1\over\lambda_i}\},i=1,2,3$ remain on $\U\backslash\{\pm1\}$ for $t\in(t_1,t_2)$.
\end{lemma}

If the first case occurs, we can handle the collision in a similar way as in $\Sp(4)$. Thus, we only need to address the second case.

For Case (ii), in general, on the time interval $t\in(t_1, t_2+2\epsilon)$
for some $\ep>0$ small enough, the third pair of eigenvalues ${\lambda_3, {1\over\lambda_3}}$ may collide with other eigenvalues different from ${\lambda_i, {1\over\lambda_i}}, i=1,2$. By adjusting the speed of some eigenvalues (see Lemma \ref{Lemma:change.part.time}), we can move these collisions with respect to ${\lambda_3, {1\over\lambda_3}}$ between times $t_2+\epsilon$ and $t_2+2\epsilon$ via a positive homotopy. 
Now the other eigenvalues do not collide with ${\lambda_i, {1\over\lambda_i}}, i=1,2,3$ in the time interval $(t_1-\epsilon, t_2+\epsilon)$. Hence, according to the Positive Decomposition Lemma,
there exists a symplectic path $X(t), t\in(t_1-\epsilon, t_2+\epsilon)$ such that

\begin{equation}\label{Sp6.decomposition}
\gamma(t)=X(t)^{-1}\left(\matrix{ \gamma_1(t)& O\cr O& \gamma_2(t)}\right)X(t).
\end{equation}
Here $\gamma_1\in\Sp(6), \gamma_2\in(\Sp(2(n-3)))$ 
are positive paths (see Proposition \ref{Prop:positive.decomposition} below), and
\begin{equation}
\sigma(\gamma_1(t))
=\{\lambda_1(t),{1\over\lambda_1(t)},
   \lambda_2(t),{1\over\lambda_2(t)},
   \lambda_3(t),{1\over\lambda_3(t)}\}.
\end{equation}

Now, let's assume $\gamma_1$ exits $\mathcal{O}_\mathcal{C,U}$ and enters $\mathcal{O}_\mathcal{U}$ at time $t_1$. During the interval $t\in(t_1, t_2)$, the eigenvalues remain on $\U$. At $t=t_2$, one pair of the eigenvalues collides with the third pair and enters $\mathcal{O}_\mathcal{C,U}$. In other words, there exists an $\epsilon>0$ small enough such that
\begin{equation}
\gamma_(t)\in\left\{\matrix{
	\mathcal{O}_{\mathcal{C},\mathcal{U}},&t\in(t_1-2\epsilon,t_1),\cr 
	\mathcal{B}_{\mathcal{U}}^2,&t=t_1,\cr
	\mathcal{O}_{\mathcal{U}},&t\in(t_1,t_2),\cr 
	\mathcal{B}_{\mathcal{U}}^2,&t=t_2,\cr
	\mathcal{O}_{\mathcal{C,U}},&t\in(t_2,t_2+2\epsilon) \cr }\right.
\end{equation}
Then $\mu=\gamma_1|_{[t_1-\epsilon,t_2+\epsilon]}$ is a positive path connecting $\gamma_1(t_1-\epsilon)\in\mathcal{O}_{\mathcal{C,U}}$ and $\gamma_1(t_2+\epsilon)\in\mathcal{O}_{\mathcal{C,U}}$. We refer to such a path as an ``elementary two double-collisions path".

\begin{figure}[ht]
	\centering
	\includegraphics[height=8.0cm]{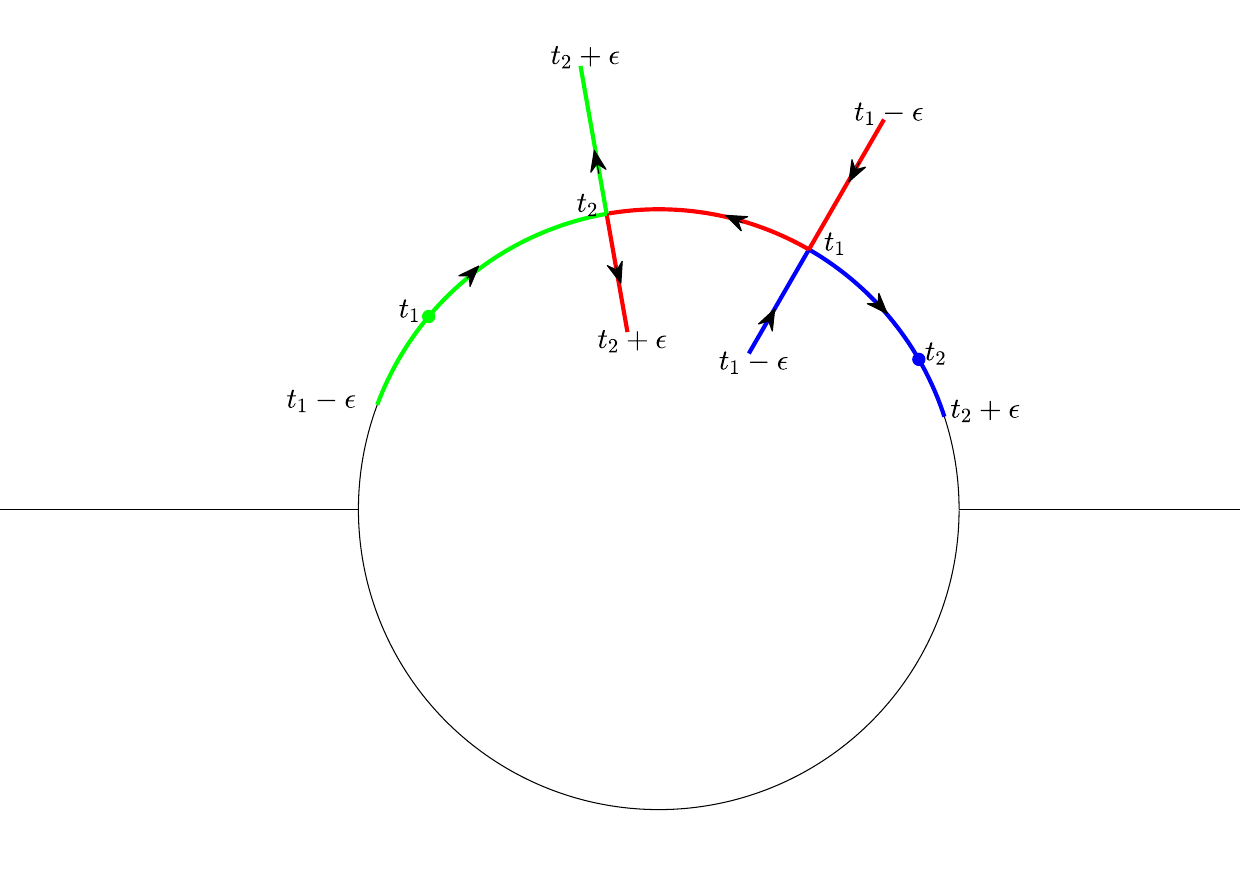}
	\vskip -0.6 cm
	\caption{An elementary two double-collisions path. Each curve with different colors represent a trajectory of the eigenvalues of 
    the path in $\Sp(6)$ for $t\in[t_1-\epsilon,t_2+\epsilon]$.}
	\label{fig:elementary.path.of.Sp6}
\end{figure}

\begin{definition}\label{elementary.two.double-collision.path}
	An {\bf elementary two double-collisions path} is a positive symplectic path in $\Sp(6)$
	that follows the trajectory:
	\begin{equation}
	\mathcal{O}_{\mathcal{C,U}}
	\rightarrow\mathcal{B}^2_{\mathcal{U}}
	\rightarrow\mathcal{O}_{\mathcal{U}}
	\rightarrow\mathcal{B}^2_{\mathcal{U}}
	\rightarrow\mathcal{O}_{\mathcal{C,U}}.
	\end{equation}
\end{definition}

We will eliminate the two double-collisions in the elementary double-collisions path through positive homotopy. Now we may have

\begin{theorem}\label{Thm:Sp6.remove.collisions.on.U}
	Any elementary two double-collisions path
	can be $C^1$-connectably positively homotopic to a positive path
	which has no collisions.
\end{theorem}

\begin{proof}
We will use several steps to prove this theorem.	 
	
{\bf  Step 1: Construct a positive path with a single triple-collision}
	
Let $\gamma\in\Sp(6)$ be an elementary two double-collisions path,
and $\Delta(t):=\Delta(\gamma(t))$ be the discriminant of its collisions.
 As depicted in Figure \ref{fig:elementary.path.of.Sp6},
 we have
	\begin{equation}
	\Delta(t)\;\left\{\matrix{
		>0, &t\in[t_1-\epsilon,t_1),\cr 
		=0, &t=t_1,\cr
		<0, &t\in(t_1,t_2),\cr 
		=0, &t=t_2,\cr
		>0, &t\in(t_2,t_2+\epsilon].\cr }\right.
	\end{equation}
Without loss of generality, we assume that 
the total time interval of $\gamma$ is $[0,1]$ and two collision times are
$t_1={1\over3}$ and $t_2={2\over3}$.	
According to the insight provided by Figure \ref{fig:elementary.path.of.Sp6} again,
	we can assume
	\begin{eqnarray}
	\sigma(\gamma(t_1))&=&\{e^{\sqrt{-1}\th_1},e^{-\sqrt{-1}\th_1},
	e^{\sqrt{-1}\th_1},e^{-\sqrt{-1}\th_1},
	e^{\sqrt{-1}\th_3}, e^{-\sqrt{-1}\th_3}\},
	\\
	\sigma(\gamma(t_2))&=&\{e^{\sqrt{-1}\th_0},e^{-\sqrt{-1}\th_0},
	e^{\sqrt{-1}\th_2},e^{-\sqrt{-1}\th_2},
	e^{\sqrt{-1}\th_2}, e^{-\sqrt{-1}\th_2}\},
	\end{eqnarray}
	where $\th_0<\th_1<\th_2<\th_3$.

Letting $\varphi_0={\th_1+\th_2\over2}$,
	for $\epsilon>0$ sufficiently small, 	$\mu(t)=N_3(\varphi_0,b)\cdot\diag(R(t-{1\over2}),R(t-{1\over2}),R(t-{1\over2}))$ is a positive path
	for $t\in[{1\over2}-\epsilon,{1\over2}+\epsilon]$. By $5^\circ$ of Lemma \ref{lemma:N_3}, for $t\ne{1\over2}$,
	we have $\Delta(\mu(t)) > 0$. Consequently, $\sigma(\mu(t))$ possesses
	one pair of eigenvalues on $\U$ and a quadruple of $4$ eigenvalues in $\C\backslash(\U\cup\R)$.
Therefore, $\mu$ has the type
	$\mathcal{O}_{\mathcal{C,U}}
	\rightarrow\mathcal{B}_{\mathcal{U}}^3
	\rightarrow\mathcal{O}_{\mathcal{C,U}}$.

Since $\gamma(0),\mu({1\over2}-\epsilon)\in\mathcal{O}_{\mathcal{C,U}}$,
    for $\epsilon>0$ sufficiently small,
	there exists a positive path $\gamma_1(t),t\in[0,{1\over2}-\epsilon]$ that connects $X_1^{-1}\gamma(0)X_1$
	and $\mu({1\over2}-\epsilon)$ for some $X_1\in\Sp(6)$, such that
	\begin{eqnarray}
	\gamma_1'(0)&=&X_1^{-1}\gamma'(0)X_1,
	\\
	\gamma_1'({1\over2}-\epsilon)&=&\mu'({1\over2}-\epsilon).
	\end{eqnarray}
    Here $\gamma_1$ is entirely in $\mathcal{O}_{\mathcal{C,U}}$.
	
	By a similar argument, there exists a positive path $\gamma_2(t),t\in[{1\over2}+\epsilon,1]$ that connects
	$\mu({1\over2}+\epsilon)$ and $X_2^{-1}\gamma(1)X_2$ for some $X_2\in\Sp(6)$, such that
	\begin{eqnarray}	
	\gamma_1'({1\over2}+\epsilon)&=&\mu'({1\over2}+\epsilon),
	\\
	\gamma_1'(1)&=&X_2^{-1}\gamma'(1)X_2,
	\end{eqnarray}
    where $\gamma_2$ is entirely in $\mathcal{O}_{\mathcal{C,U}}$.
	Now $\tilde\gamma=\gamma_1*\mu*\gamma_2$ is a positive path
	with type $\mathcal{O}_{\mathcal{C,U}}
	\rightarrow\mathcal{B}_{\mathcal{U}}^3
	\rightarrow\mathcal{O}_{\mathcal{C,U}}$.
	Moreover, we have $\tilde\gamma(0)=\mu(0)=N_3(\varphi_0,b)$,
	and
	\begin{eqnarray}	
	\tilde\gamma(0)=X_1^{-1}\gamma(0)X_1,&\quad&
	\tilde\gamma'(0)=X_1^{-1}\gamma'(0)X_1,
	\\
	\tilde\gamma(1)=X_2^{-1}\gamma(1)X_2,&\quad&
	\tilde\gamma'(1)=X_2^{-1}\gamma'(1)X_2.
	\end{eqnarray}
	 Thus $\tilde\gamma$ is the required
	 positive path
      with a single triple-collision.

 \medskip
	{\bf  Step 2: Eliminate the two double-collisions via a triple-collision path}
	
    By Lemma \ref{Lemma:cusp}, the bifurcation diagram of
    general deformations for $N_3(\varphi,b)\in\mathcal{B}_{\mathcal{U}}^3$ is a cusp with type $A_2$, as illustrated in Figure \ref{fig:cusp}
    (or see Figure \ref{figure:bf_c2} (b)).
    Consequently, we can perturb $\tilde\gamma$ near the point $N_3(\varphi_0,b)$.
   This perturbation can lead to two scenarios:
	either the perturbed path intersects $\mathcal{B}_{\mathcal{U}}^2$ at precisely
	two points (denoted as $\gamma_-$),
	or the perturbed path does not intersect with $\mathcal{B}_{\mathcal{U}}^2$
	(denoted as $\gamma_+$), as depicted in Figure \ref{fig:cusp}.
    Note that $\gamma_+,\gamma_-$, and $\tilde\gamma$
    are positive homotopy with fixed end-segments (refer to Formula (\ref{fix.condition}) and subsequent content),
    which implies that $\gamma_+\sim_{+,con}\tilde\gamma\sim_{+,con}\gamma_-$.
	
	\begin{figure}[ht]
		\centering
		\includegraphics[height=8.0cm]{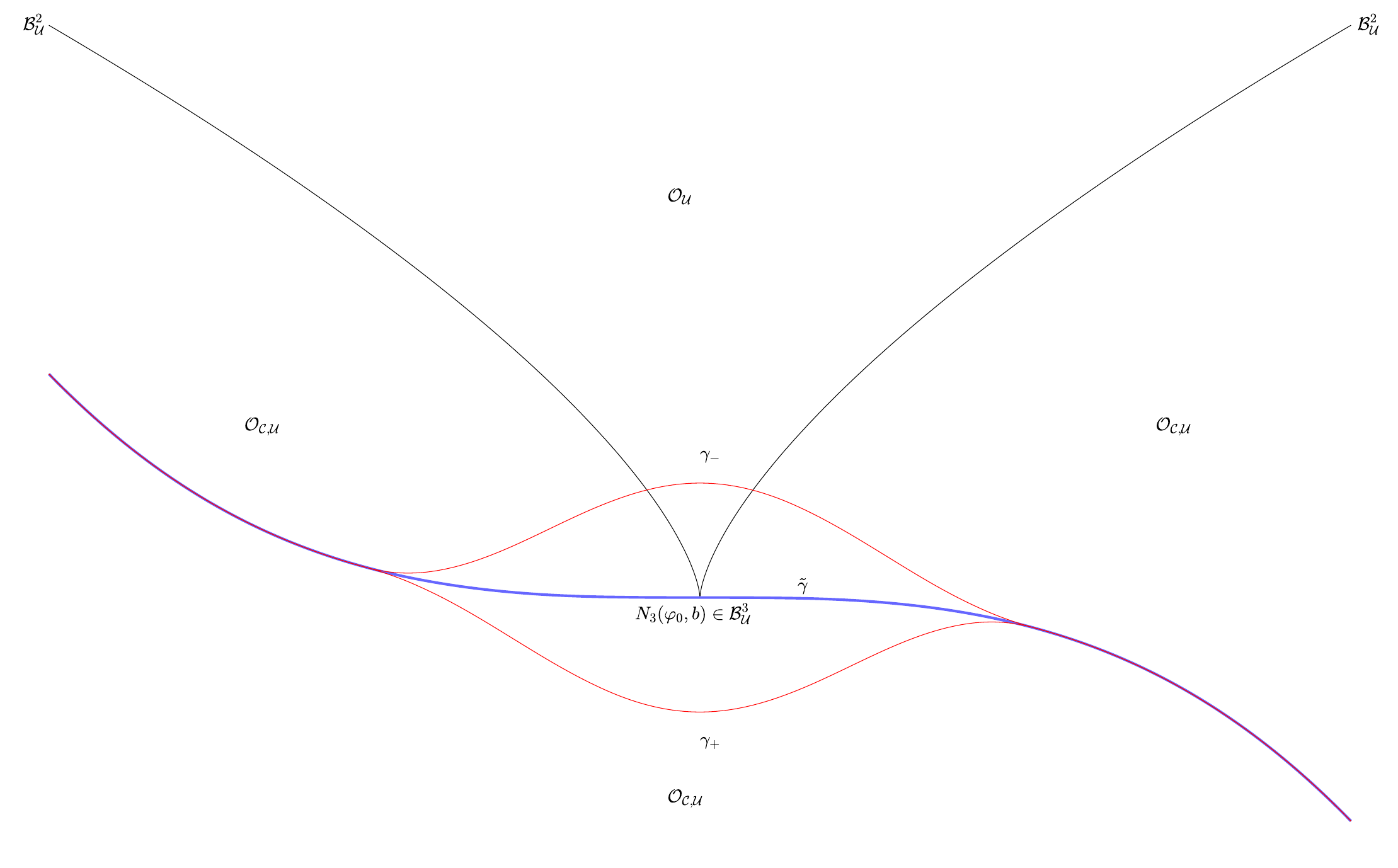}
		\caption{Perturb the path near the cusp point.}
		\label{fig:cusp}
	\end{figure}
	
	The positive path $\gamma_-$ has the same type as $\gamma$.
	Using a similar argument to that in Theorem \ref{Th:transfer.collisions},
	we can move the two double-collision points
	such that $\gamma_-$ and $\gamma$ has the same double-collision points in the sense of conjugate class.
	Moreover, we have $\gamma\sim_{+,con}\gamma_-$,
	and hence $\gamma\sim_{+,con}\gamma_+$.
\end{proof}



\subsection{Eliminate collisions on the real line}

By applying Theorem \ref{Th:transfer.collisions} and Theorem \ref{Thm:Sp6.remove.collisions.on.U},
we can eliminate the collisions on $\U$
or transfer the collisions from $\U$ to $\R$.
After these operations, the positive path has collisions only on the real line $\R$.

    Now, let $\ga$ be a positive loop that starts at $I$
    and has collisions only on the real line $\R$.
	Suppose the first collision occurs at time $t=t_1$,
    where two pairs of eigenvalues, 
    $\{\lambda_1,{1\over\lambda_1}\}$ and $\{\lambda_2,{1\over\lambda_2}\}$ rotate $k_1$ and $k_2$ rounds, respectively;
    and subsequently, they collide at $\R^+$ (or $\R^-$).
   Because the loop must return back to $I$,
	there must be a subsequent collision
	from $\mathcal{O}_\mathcal{C}$ to $\R$.
    Moreover, we assume they collide successively at $t=t_2$.

    Since the other eigenvalues do not collide with 
$\{\lambda_i,{1/\lambda_i}\},i=1,2$ on the time interval $[0,t_2+\epsilon)$ where $\ep>0$ is small enough,
there exist a symplectic path $X_t,t\in[0,t_2+\epsilon]$
such that
\begin{equation}\label{path.decomposition}
\gamma(t)=X_t^{-1}\left(\matrix{ \gamma_1(t)& O\cr O& \gamma_2(t)}\right)X_t.
\end{equation}
Here $\gamma_1\subset\Sp(4), \gamma_2\subset\Sp(2(n-2))$, and
\begin{equation}
\sigma(\gamma_1(t))
=\{\lambda_1(t),{1\over\lambda_1(t)},
   \lambda_2(t),{1\over\lambda_2(t)}\}.
\end{equation}

Using the same argument as in proof of Theorem \ref{Thm.remove.collisions.on.R.of.Sp4} for $\Sp(4)$, 
$\gamma_1|_{[0,t_2+\epsilon]}$ falls into one of the four cases:
	
	\noindent(i) The path starts from $\mathcal{O}_{\mathcal{R}^+,\mathcal{R}^+}$
	enters $\mathcal{O}_\mathcal{C}$, and then 
re-enters $\mathcal{O}_{\mathcal{R}^+,\mathcal{R}^+}$;
	
	\noindent(ii) The path starts from $\mathcal{O}_{\mathcal{R}^-,\mathcal{R}^-}$
	enters $\mathcal{O}_\mathcal{C}$, and then 
 re-enters $\mathcal{O}_{\mathcal{R}^-,\mathcal{R}^-}$;
	
	\noindent(iii) The path starts from $\mathcal{O}_{\mathcal{R}^+,\mathcal{R}^+}$
	enters $\mathcal{O}_\mathcal{C}$, and then re-enters $\mathcal{O}_{\mathcal{R}^-,\mathcal{R}^-}$;
 
	\noindent(iv) The path starts from $\mathcal{O}_{\mathcal{R}^-,\mathcal{R}^-}$
	enters $\mathcal{O}_\mathcal{C}$, and then re-enters $\mathcal{O}_{\mathcal{R}^+,\mathcal{R}^+}$.
    
For Cases (i) and (ii), the two double-collisions can be removed together, as shown by Lemma
\ref{Lemma:Sp4.case1&2}.
For Case (iii), we have

\begin{lemma}\label{Lemma:case3}
For Case (iii), we can remove at least two collisions 
through a proper positive homotopy.
\end{lemma}

\begin{proof}
    We need to consider whether there exists the third collision with $\{\lambda_i,{1\over\lambda_i}\},i=1,2$.
    If not, the decomposition (\ref{path.decomposition})
    can be extended to the entire interval $[0,1]$.
    Then, following the proof of Theorem \ref{Thm.remove.collisions.on.R.of.Sp4} for $\Sp(4)$, 
    the two collisions of $\gamma_1$ can be removed together.
    
	Now, suppose there exists a third collision at time $t_3$.
    If the third collision is also collided by  $\{\lambda_1,{1\over\lambda_1}\}$ and $\{\lambda_2,{1\over\lambda_2}\}$,
    we can apply the proof of Theorem \ref{Thm.remove.collisions.on.R.of.Sp4} for $\Sp(4)$
    to eliminate the collisions of $\gamma_1$.

    A new phenomenon occurs when the next collision with respect to $\{\lambda_i,{1\over\lambda_i}\},i=1,2$ 
   involves a third pair of eigenvalues $\{\lambda_3,{1\over\lambda_3}\}$.
    Without loss of generality, suppose $\{\lambda_1,{1\over\lambda_1}\}$ 
    collides with $\{\lambda_3,{1\over\lambda_3}\}$ at time $t=t_3$.
    Then $\{\lambda_2,{1\over\lambda_2}\}$ does not
    collide with $\{\lambda_i,{1\over\lambda_i}\},i=1,3$, during the time interval 
    $(t_2,t_3)$.

    If necessary, we first move
    all the collisions between $\{\lambda_3,{1\over\lambda_3}\}$ and other eigenvalues
    except for $\{\lambda_i,{1\over\lambda_i}\},i=1,2$, on the time interval 
    $(t_1-2\epsilon,t_3-2\epsilon)$ to the new
    time interval $(t_1-2\epsilon,t_1-\epsilon)$ 
    by using the Part Speed Changing Lemma,
    we obtain a similar decomposition as in
    (\ref{Sp6.decomposition}) on the time interval
    $(t_1-\epsilon,t_3+2\epsilon)$.
    
    By Theorem \ref{Th:transfer.R+.to.R-},
    Remark \ref{RM:transfer.R+.to.R-}
	and Theorem \ref{Th:transfer.collisions.on.R},
	we can transfer the collision at $t=t_1$ from $\R^+$
	to $\U$, and then to $\R^-$.
	Now the new positive path travels from $\mathcal{O}_{\mathcal{R}^-,\mathcal{R}^-}$
	enters $\mathcal{O}_\mathcal{C}$ at $t=t_1$, and then re-enters $\mathcal{O}_{\mathcal{R}^-,\mathcal{R}^-}$ at $t=t_2$.
    Therefore, we convert the path to Case (ii).

    At last, the two double-collisions can be removed
    together, as shown by Lemma \ref{Lemma:Sp4.case1&2}.
\end{proof}

Case (iv) is more complicated,  and we 
provide a detailed proof in the following lemma:
\begin{lemma}\label{Lemma:case4}
For Case (iv), we can remove at least two collisions 
through a proper positive homotopy.
\end{lemma}

\begin{proof}
    For Case (iv), 
    either there does not exist
    a third collision,
    or the third collision is also
    collided by $\{\lambda_1,{1\over\lambda_1}\}$ and $\{\lambda_2,{1\over\lambda_2}\}$,
    we can eliminate the collisions
    as a similar argument of Case (iii).
    

    Now we consider when
    the next collision with respect to $\{\lambda_i,{1\over\lambda_i}\},i=1,2$ 
   involves a third pair of eigenvalues $\{\lambda_3,{1\over\lambda_3}\}$.
    Without loss of generality, suppose $\{\lambda_1,{1\over\lambda_1}\}$ 
    collides with $\{\lambda_3,{1\over\lambda_3}\}$ at time $t=t_3$.
    Then $\{\lambda_2,{1\over\lambda_2}\}$ does not
    collide with $\{\lambda_i,{1\over\lambda_i}\},i=1,3$, during the time interval 
    $(t_2,t_3)$.

    By a similar argument to Case (iii),
    and utilizing the Part Speed Changing Lemma,
    we obtain a decomposition similar to
    (\ref{Sp6.decomposition}) on time interval
    $(t_1-\epsilon,t_3+\epsilon)$.
    Now, it falls into one of the following three sub-cases:
	
    i) $\{\lambda_1,{1\over\lambda_1}\}$ remains in $\R^+$ if $t\in(t_2,t_3)$,
    and $\{\lambda_1,{1\over\lambda_1}\}$ collides
    with $\{\lambda_3,{1\over\lambda_3}\}$ from $\R^+$ to $\C\backslash(\U\cup\R)$ at time $t_3$;
	
    ii) $\{\lambda_1,{1\over\lambda_1}\}$ rotates $k\ge1$ rounds to
    $\R^+$ if $t\in(t_2,t_3)$,
    and then $\{\lambda_1,{1\over\lambda_1}\}$ collides
    with $\{\lambda_3,{1\over\lambda_3}\}$ from $\R^+$ to $\C\backslash(\U\cup\R)$ at time $t_3$;

    iii) $\{\lambda_1,{1\over\lambda_1}\}$ rotates $k+{1\over2}$ rounds to $\R^-$ if $t\in(t_2,t_3)$,
    and then $\{\lambda_1,{1\over\lambda_1}\}$ collides
    with $\{\lambda_3,{1\over\lambda_3}\}$ from $\R^-$ to $\C\backslash(\U\cup\R)$ at time $t_3$.

    For Subcase i), if necessary, we first relocate all collisions between $\{\lambda_3,{1\over\lambda_3}\}$ and eigenvalues other than $\{\lambda_i,{1\over\lambda_i}\}, i=1,2$, from the time interval $(t_2-2\epsilon, t_3-2\epsilon)$ to the new interval $(t_2-2\epsilon, t_2-{\epsilon})$ by using the Part Speed Changing Lemma. 
    Then we have a similar decomposition with
    (\ref{Sp6.decomposition}) on time interval
    $(t_2-\epsilon,t_3+\epsilon)$,
   as illustrated in Figure \ref{Sp6_hyp}.
    Therefore, by Lemma \ref{Lemma:hyperbolic.positive.path} and
    Theorem \ref{Thm:remove.hyperbolic.collisions} 
    (The Positive Homotopy Theorem of Positive Truly Hyperbolic Paths), 
    we can remove the two collisions
	at $t=t_2$ and $t=t_3$.

    \begin{figure}[ht]
	\centering
	\includegraphics[height=7.0cm]{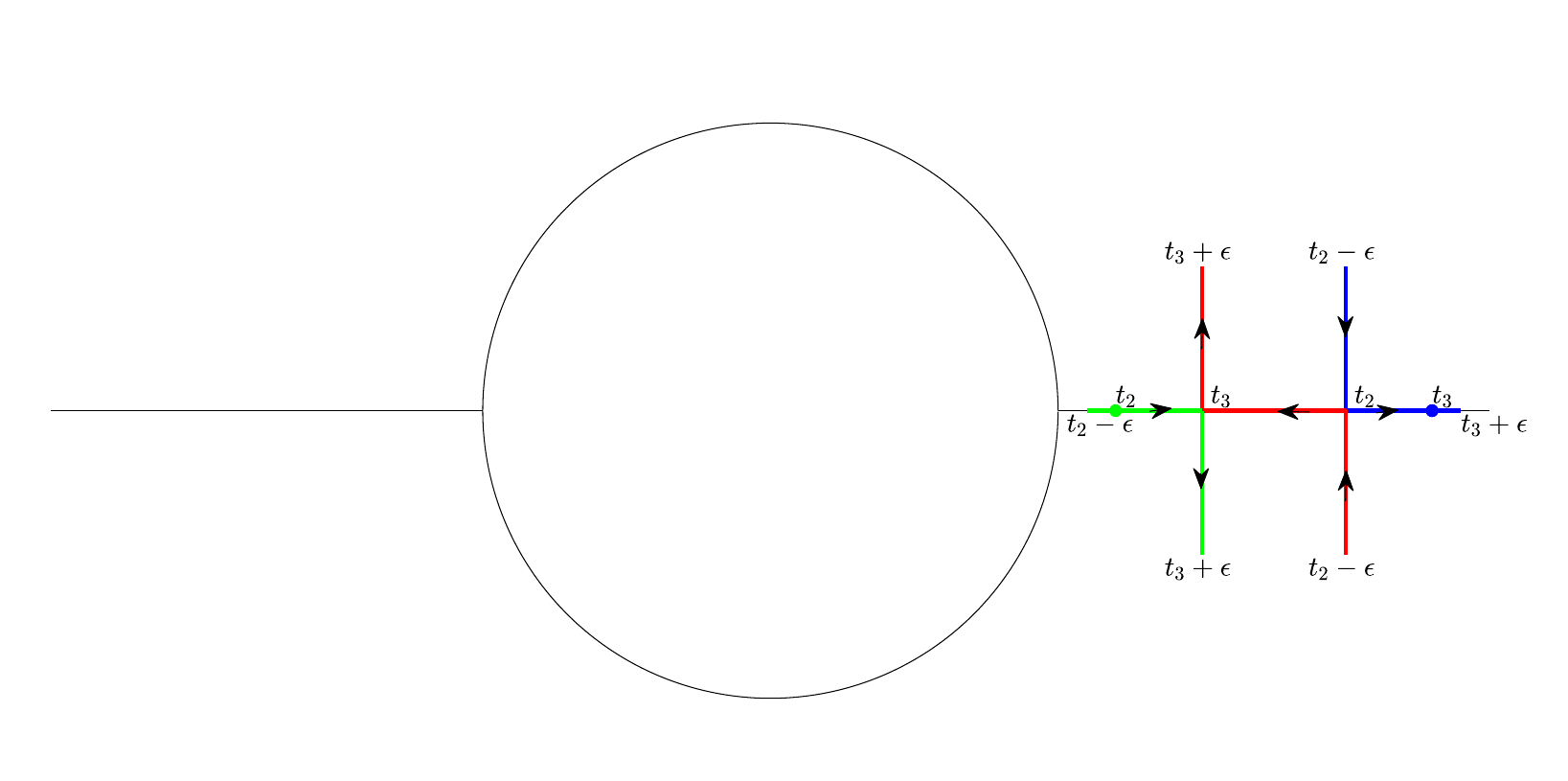}
	\vskip -1 cm
	\caption{Subcase i): Each curve with different colors represents a trajectory of the eigenvalues for $\gamma_1\in\Sp(6)$ for $t\in(t_2-\epsilon,t_3+\epsilon)$.}
	\label{Sp6_hyp}
    \end{figure}

   

Now we treat subcase ii). 
By Theorem \ref{Th:transfer.R+.to.R-},
Remark \ref{RM:transfer.R+.to.R-},
Lemma \ref{Th:transfer.collisions.on.R}
and a similar argument as in Case (iii),
we can transfer the collision
 of $\gamma_1$ in (\ref{path.decomposition})
 at $t=t_2$ from $\R^+$
	to $\U$, and then to $\R^-$.
	Now the new positive path $\gamma_1$ travels from $\mathcal{O}_{\mathcal{R}^-,\mathcal{R}^-}$,
	enters $\mathcal{O}_\mathcal{C}$ at $t=t_1$, and then re-enters $\mathcal{O}_{\mathcal{R}^-,\mathcal{R}^-}$ at $t=t_2$.
	Therefore, we convert the path to Case (ii).

    Subcase iii) is similar to Subcase ii).
 \end{proof}

 According to Cases (i)-(iv), we have

 \begin{theorem}\label{Thm.remove.collisions.on.R.of.Sp2n}
	Given a positive loop $\gamma$ in
 $\Sp(2n)$ which only 
	has collisions on  $\R\backslash\{0,\pm1\}$,
	it can be homotopic to a positive loop 
    with no collisions.
\end{theorem}

 \begin{proof}

    Since for Cases (i) and (ii), the two double-collisions can be removed together by Lemma
\ref{Lemma:Sp4.case1&2},
     it immediately follows from Lemma \ref{Lemma:case3} 
     and Lemma \ref{Lemma:case4}. Then the theorem follows from the argument by induction as in the introduction (see Section \ref{subsec, n>2}).
 \end{proof}

Now we can give the proof of our main theorem for $\Sp(2n),n\ge3$.

\begin{proof}[Proof of Theorem A for $\Sp(2n),n\ge3$]
According to Theorem \ref{Th:transfer.collisions}, Theorem \ref{Thm:Sp6.remove.collisions.on.U} and induction, 
if a positive path includes collisions on the unit circle, it can be positively homotopic to another positive path that only has collisions on the real line. Subsequently, applying Theorem \ref{Thm.remove.collisions.on.R.of.Sp2n} and Theorem \ref{Thm:pos.homotopy.of.loops.has.no.collisions}, we can conclude the cases $n>2$ of Theorem A.
\end{proof}



\setcounter{equation}{0}
\section{Appendix}

\subsection{The proofs of technical lemmas}
This section presents the proofs of technical lemmas required in the preceding sections. We will restate the lemmas here for the reader's convenience.

\noindent{\bf Lemma 2.3.}
{\it	
	For $\th\in(0,2\pi)\backslash\{\pi\}$ and $(b_1,b_2,b_3,b_4)\in\R^4$, 
	consider the matrix $N=N_2(\th,b)$ and $N_\alpha=N\cdot\diag(R(\alpha),R(\alpha))$.
    Then

    $1^\circ$
    $\Delta(N_\alpha)=8(b_2-b_3)\sin\th\cos\alpha\sin\alpha
+[16\sin^2\th+(b_1-b_4)^2+4b_2b_3]\sin^2\alpha$;

    $2^\circ$ If $N$ is non-trivial, i.e., $(b_2-b_3)\sin\th<0$,
    then there exists $\alpha_0>0$ small enough such that
    $\Delta(N_\alpha)<0$ if $\alpha\in(0,\alpha_0)$,
    and $\Delta(N_\alpha)>0$ if $\alpha\in(-\alpha_0,0)$;

    $3^\circ$ If $N$ is trivial, i.e., $(b_2-b_3)\sin\th>0$,
    then there exists $\alpha_0>0$ small enough such that
    $\Delta(N_\alpha)>0$ if $\alpha\in(0,\alpha_0)$,
    and $\Delta(N_\alpha)<0$ if $\alpha\in(-\alpha_0,0)$.
}

\begin{proof}
$1^\circ$ First, we use $c_\th,s_\th$ and $c_\alpha,s_\alpha$
to denote $\cos\th,\sin\th$ and $\cos\alpha,\sin\alpha$, respectively.
By direct computation, we have
\begin{eqnarray*}
N_\alpha&=&N_2(\th,b)\cdot \diag(R(\alpha),R(\alpha))  \nonumber\\
&=&
\left(\matrix{c_\th &   b_1 &       -s_\th &         b_2 \cr
	0 & c_\th &       0 & -s_\th \cr
	s_\th &   b_3 &  c_\th &         b_4 \cr
	0 &   s_\th & 0 &  c_\th \cr}\right)
\left(\matrix{c_\alpha &  -s_\alpha &   0&         0\cr
	s_\alpha & c_\alpha &       0 & 0 \cr
	0 &   0 &  c_\alpha &  -s_\alpha \cr
	0 &    0 & s_\alpha &  c_\alpha \cr}\right)
\nonumber\\
&=&
\left(\matrix{c_\th c_\alpha+b_1s_\alpha &   
        -c_\th s_\alpha+b_1c_\alpha &       
        -s_\th c_\alpha+b_2s_\alpha &         
        s_\th s_\alpha+b_2c_\alpha \cr
	c_\th s_\alpha & c_\th c_\alpha &  
        -s_\th s_\alpha & -s_\th c_\alpha \cr
	s_\th c_\alpha+b_3s_\alpha &   
        -s_\th s_\alpha+b_3c_\alpha &  
        c_\th c_\alpha+b_4s_\alpha &         
        -c_\th s_\alpha+b_4c_\alpha \cr
	s_\th s_\alpha &   s_\th c_\alpha & 
        c_\th s_\alpha &  c_\th c_\alpha \cr
}\right).
\end{eqnarray*}
Then we have
$$
\sigma_1(N_\alpha)=tr(N_\alpha)=4\cos\th\cos\alpha+(b_1+b_4)\sin\alpha
$$
and
\begin{eqnarray*}
\sigma_2(N_\alpha)&=&(4\cos^2\th+2)\cos^2\alpha+(2-4\sin^2\th+b_1b_4-b_2b_3)\sin^2\alpha
\nonumber\\
&&+2[(b_1+b_4)\cos\th-(b_2-b_3)\sin\th]
\cos\alpha\sin\alpha.
\end{eqnarray*}
Hence,
we have
\begin{eqnarray*}
\Delta(N_\alpha)&=&
\sigma_1(N_\alpha)^2-4\sigma_2(N_\alpha)+8
\nonumber\\
&=&8(b_2-b_3)\sin\th\cos\alpha\sin\alpha
+[16\sin^2\th+(b_1-b_4)^2+4b_2b_3]\sin^2\alpha.
\end{eqnarray*}

$2^\circ$ Since $\sin^2\alpha$ is a higher-order term
with respect to $\alpha$,
the signature of $\Delta(N_\alpha)$ only depends on
its first term for $\alpha$ small enough.
Therefore, $2^\circ$ follows from $(b_2-b_3)\sin\th<0$.

Using similar arguments, we obtain $3^\circ$.
\end{proof}

\noindent{\bf Lemma 2.4.}
{\it
	For $(c_1,c_2)\in\R^2$ where $c_2<0$, 
	consider the matrix $M^{\pm}_\alpha=M_2(\pm1,c)\cdot\diag(R(\alpha),R(\alpha))$.
    Then we have

    $1^\circ$
    $\Delta(M_2(1,c)\cdot \diag(R(\alpha),R(\alpha)))=
    4c_2\cos\alpha\sin\alpha+[(c_1+c_2)^2+4]\sin^2\alpha$.
    Hence there exists $\alpha_0>0$ small enough such that
    $\Delta(M^+_\alpha)<0$ if $\alpha\in(0,\alpha_0)$,
    and $\Delta(M^+_\alpha)>0$ if $\alpha\in(-\alpha_0,0)$;

    $2^\circ$
    $\Delta(M_2(-1,c)\cdot \diag(R(\alpha),R(\alpha)))=
    4c_2\cos\alpha\sin\alpha+[(c_1-c_2)^2+4]\sin^2\alpha$.
    Hence there exists $\alpha_0>0$ small enough such that
    $\Delta(M^-_\alpha)<0$ if $\alpha\in(0,\alpha_0)$,
    and $\Delta(M^-_\alpha)>0$ if $\alpha\in(-\alpha_0,0)$.
} 
\begin{proof}
$1^\circ$ First, we have
\begin{eqnarray*}
M^+_\alpha
&=&M_2(1,c)\cdot\diag(R(\alpha),R(\alpha))  \nonumber\\
&=&
\left(\matrix{1 &   c_1 &       1 &         0 \cr
	0 & 1 &       0 & 0 \cr
	0 &   c_2 &  1 &         -c_2 \cr
	0 &   -1 & 0 &  1 \cr}\right)
\left(\matrix{\cos\alpha &  -\sin\alpha &   0&         0\cr
	\sin\alpha & \cos\alpha &       0 & 0 \cr
	0 &   0 &  \cos\alpha &  -\sin\alpha \cr
	0 &    0 & \sin\alpha &  \cos\alpha \cr}\right)
\nonumber\\
&=&
\left(\matrix{\cos\alpha+c_1\sin\alpha &   
             -\sin\alpha+c_1\cos\alpha &       
             \cos\alpha&   -\sin\alpha  \cr
	\sin\alpha & \cos\alpha &       0 & 0 \cr
	c_2\sin\alpha &   c_2\cos\alpha&  
                \cos\alpha-c_2\sin\alpha &         
                -\sin\alpha-c_2\cos\alpha\cr
	-\sin\alpha &   -\cos\alpha & 
             \sin\alpha &  \cos\alpha\cr
}\right).
\end{eqnarray*}
Then we get
$$
\sigma_1(M^+_\alpha)=tr(M^+_\alpha)=4\cos\alpha+(c_1-c_2)\sin\alpha,
$$
and
\begin{eqnarray*}
\sigma_2(M^+_\alpha)=2+4\cos^2\alpha-(1+c_1c_2)\sin^2\alpha
+(2c_1-3c_2)\cos\alpha\sin\alpha;
\end{eqnarray*}
and hence,
we have
\begin{eqnarray*}
\Delta(M^+_\alpha)=
\sigma_1(M^+_\alpha)^2-4\sigma_2(M^+_\alpha)+8
=4c_2\cos\alpha\sin\alpha
+[(c_1+c_2)^2+4]\sin^2\alpha.
\end{eqnarray*}
Finally, the signature of $\Delta(M^+_\alpha)$ 
is the same as that of $-\alpha$
since $c_2<0$ and $\sin^2\alpha=o(\alpha)$ 
for $\alpha$ small enough.

$2^\circ$ Similarly, We have
\begin{eqnarray*}
M^-_\alpha
&=&M_2(-1,c)\cdot\diag(R(\alpha),R(\alpha))  \nonumber\\
&=&
\left(\matrix{-1 &   c_1 &       1 &         0 \cr
	0 & -1 &       0 & 0 \cr
	0 &   c_2 &  -1 &         c_2 \cr
	0 &   -1 & 0 &  -1 \cr}\right)
\left(\matrix{\cos\alpha &  -\sin\alpha &   0&         0\cr
	\sin\alpha & \cos\alpha &       0 & 0 \cr
	0 &   0 &  \cos\alpha &  -\sin\alpha \cr
	0 &    0 & \sin\alpha &  \cos\alpha \cr}\right)
\nonumber\\
&=&
\left(\matrix{-\cos\alpha+c_1\sin\alpha &   
             \sin\alpha+c_1\cos\alpha &       
             \cos\alpha&   -\sin\alpha  \cr
	-\sin\alpha & -\cos\alpha &       0 & 0 \cr
	c_2\sin\alpha &   c_2\cos\alpha&  
             -\cos\alpha+c_2\sin\alpha &         
             \sin\alpha+c_2\cos\alpha\cr
	-\sin\alpha &   -\cos\alpha & 
             -\sin\alpha &  -\cos\alpha\cr
}\right)
\end{eqnarray*}
and
\begin{eqnarray*}
\Delta(M^-_\alpha)
=4c_2\cos\alpha\sin\alpha
+[(c_1-c_2)^2+4]\sin^2\alpha.
\end{eqnarray*}
In this case, the signature of $\Delta(M^-_\alpha)$ with respect to $-\alpha$ remains the same since $c_2 < 0$ and $\sin^2\alpha = o(\alpha)$ for sufficiently small $\alpha$.

\end{proof}

\noindent{\bf Lemma 2.5.}	
{\it For $\lambda\in\R\backslash\{0,\pm1\}$ and $c=(c_1,c_2)\in\R^2$, 
	 consider the matrix $M_2(\lambda,c)$ and $M_\alpha=M_2(\lambda,c)\cdot\diag(R(\alpha),R(\alpha))$
  for some small $\alpha$.
	 Then
	 
	 $1^\circ$ $M_2(\lambda,c)\in\Sp(4)$.
	  
	  $2^\circ$ $\lambda$ and ${1\over\lambda}$ are double eigenvalues of $M$.
	  
	  $3^\circ$ $\dim_{\C}\ker_{\C}(M-\lambda I)=1$ and
   $\dim_{\C}\ker_{\C}(M-{1\over\lambda} I)=1$.
	  
	  $4^\circ$         $\Delta(M_2(\lambda,c)\cdot\diag(R(\alpha),R(\alpha)))=4c_2\cos\alpha\sin\alpha+[(c_1+\lambda c_2)^2+{4\over\lambda^2}]\sin^2\alpha$.
    Hence, if $c_2<0$,
    there exists $\alpha_0>0$ small enough such that
    $\Delta(M_\alpha)<0$ if $\alpha\in(0,\alpha_0)$,
    and $\Delta(M_\alpha)>0$ if $\alpha\in(-\alpha_0,0)$.
}
\begin{proof}
    $1^\circ$ to $3^\circ$ follow from direct computations.

    For $4^\circ$, by a similar argument
    as in Lemma 2.4,
    we have
    \begin{eqnarray*}
M_\alpha
&=&M_2(\lambda,c)\cdot\diag(R(\alpha),R(\alpha))  \nonumber\\
&=&
\left(\matrix{\lambda &   c_1 &       1 &         0 \cr
	0 & \lambda^{-1} &       0 & 0 \cr
	0 &   c_2 &  -\lambda&         -\lambda c_2 \cr
	0 &   -\lambda^{-2} & 0 &  \lambda^{-1}\cr}\right)
\left(\matrix{\cos\alpha &  -\sin\alpha &   0&         0\cr
	\sin\alpha & \cos\alpha &       0 & 0 \cr
	0 &   0 &  \cos\alpha &  -\sin\alpha \cr
	0 &    0 & \sin\alpha &  \cos\alpha \cr}\right)
\nonumber\\
&=&
\left(\matrix{\lambda\cos\alpha+c_1\sin\alpha &   
             -\lambda\sin\alpha+c_1\cos\alpha &       
             \cos\alpha&   -\sin\alpha  \cr
	\lambda^{-1}\sin\alpha & -\lambda^{-1}\cos\alpha &       0 & 0 \cr
	c_2\sin\alpha &   c_2\cos\alpha&  
             \lambda\cos\alpha-\lambda c_2\sin\alpha &         
             -\lambda\sin\alpha-\lambda c_2\cos\alpha\cr
	-\lambda^{-2}\sin\alpha &   -\lambda^{-2}\cos\alpha & 
             \lambda^{-1}\sin\alpha &  \lambda^{-1}\cos\alpha\cr
}\right)
\end{eqnarray*}
and
\begin{eqnarray*}
\Delta(M_\alpha)
=4c_2\cos\alpha\sin\alpha
+[(c_1+\lambda c_2)^2+{4\over\lambda^2}]\sin^2\alpha.
\end{eqnarray*}
The signature of $\Delta(M_\alpha)$ 
is the same with respect to $-\alpha$
since $c_2<0$ and $\sin^2\alpha=o(\alpha)$
for $\alpha$ small enough.
\end{proof}

\noindent{\bf Lemma 2.6.}
{\it	
    For $\om=e^{\th\sqrt{-1}}$ with
	$\th\in(0,2\pi)\backslash\{\pi\}$ and $(b_1,b_2,b_3,b_4)\in\R^4$, 
	consider the matrices $N_3(\th,b),\tilde{N}_3(\th,b)$ and $M_\alpha=N_3(\th,b)\cdot\diag(R(\alpha),R(\alpha),R(\alpha))$.
	Then
	
	$1^\circ$ $N_3(\th,b)\in\Sp(6)$ if and only if
	$\left(\matrix{f_1 & f_2\cr g_1 & g_2}\right)=
	\left(\matrix{\sin2\th & \cos2\th\cr -\cos2\th & \sin2\th}\right)$
	and
	$(b_2-b_3)\cos\th+(b_1+b_4)\sin\th=1$.
	
	$2^\circ$ $\tilde{N}_3(\th,b)\in\Sp(6)$ if and only if
	$\left(\matrix{f_1 & f_2\cr g_1 & g_2}\right)=
	\left(\matrix{\cos2\th & -\sin2\th\cr -\sin2\th & -\cos2\th}\right)$
	and
	$(b_2-b_3)\cos\th+(b_1+b_4)\sin\th=-1$.
   In the following we always suppose $N_3(\th,b),\tilde{N}_3(\th,b)\in\Sp(6)$.
	
	$3^\circ$ $\om$ and $\bar\om$ are triple eigenvalues of $N_3(\th,b)$ (or $\tilde{N}_3(\th,b)$).
	
	$4^\circ$ $\dim_{\C}\ker_{\C}(N_3(\th,b)-\om I)=1$ and $\dim_{\C}\ker_{\C}(\tilde{N}_3(\th,b)-\om I)=1$.
	
	$5^\circ$ Let $\Delta(\alpha)=\Delta(M_\alpha)$ for $\alpha\in\R$.
	Then $\Delta(0)=\Delta'(0)=0$ and $\Delta''(0)=10368\sin^6\th>0$.

    $6^\circ$ There exists $\alpha_0>0$ small enough such that
	$D_\om(M_\alpha)\ne0$ if $0<|\alpha|<\alpha_0$.
}
\begin{proof}
$1^\circ$ to $4^\circ$ follow from direct
computations.

For $5^\circ$,
since the eigenvalues of $N_3(\th,b)$ are
a triple pair of $\{e^{\sqrt{-1}\th},e^{-\sqrt{-1}\th}\}$,
we have
\begin{eqnarray*}
\sigma_1(M_\alpha)|_{\alpha=0}&=&\sigma_1(N_3(\th,b))=6\cos\th,
\\
\sigma_2(M_\alpha)|_{\alpha=0}&=&\sigma_2(N_3(\th,b))=3+12\cos^2\th,
\\
\sigma_3(M_\alpha)|_{\alpha=0}&=&\sigma_3(N_3(\th,b))=12\cos\th+8\cos^3\th.
\end{eqnarray*}
Then we have
$$
    A(M_\alpha)|_{\alpha=0}=B(M_\alpha)|_{\alpha=0}=
    C(M_\alpha)|_{\alpha=0}=0,
$$
where $A,B,C$ are given by (\ref{A})-(\ref{C}).
Hence,
$$
    \Delta(M_\alpha)|_{\alpha=0}=(B^2-4AC)|_{\alpha=0}=0.
$$
Moreover, we have
\begin{eqnarray}
    \Delta(M_\alpha)'|_{\alpha=0}&=&(2BB'-4A'C-4AC')|_{\alpha=0}=0,
    \nonumber\\
    \Delta(M_\alpha)''|_{\alpha=0}&=&
   [(2BB''+2(B')^2-4(4A'C+AC')-8A'C']|_{\alpha=0}  \nonumber
    \\
    &=&
    2[(B'|_{\alpha=0})^2-4(A'|_{\alpha=0})(C'|_{\alpha=0})], \label{7.1}
\end{eqnarray}
where $'$ and $''$ represent the first and second derivatives with respect to $\alpha$,
respectively.

Now we use $c_\th,s_\th$ and $c_\alpha,s_\alpha$
to denote $\cos\th,\sin\th$ and $\cos\alpha,\sin\alpha$, respectively.
Then we have
\begin{eqnarray*}
&&M_\alpha 
=N_3(\th,b)\cdot \diag(R(\alpha),R(\alpha),R(\alpha))  \nonumber\\
&&=
\left(\matrix{
	c_\th &   b_1   & -s_\th & b_2      & 1     & 0\cr
	      0 & c_\th &    0     & -s_\th & 0     & 0\cr
	s_\th &   b_3   &  c_\th &      b_4 & 0     & 1\cr
	      0 & s_\th &    0     &  c_\th & 0     & 0\cr
          0 &    f_1  &    0     &     f_2  &c_\th&-s_\th\cr
          0 &    g_1  &    0     &     g_2  &s_\th& c_\th\cr
      }\right)
\left(\matrix{c_\alpha &  -s_\alpha &   0&         0& 0& 0\cr
	s_\alpha & c_\alpha &       0 & 0& 0& 0 \cr
	0 &   0 &  c_\alpha &  -s_\alpha& 0& 0 \cr
	0 &    0 & s_\alpha &  c_\alpha& 0& 0 \cr
 0& 0& 0& 0& c_\alpha& -s_\alpha\cr
 0& 0& 0& 0& s_\alpha& c_\alpha\cr}\right)
\nonumber\\
&&=
\left(\matrix{c_\th c_\alpha+b_1 s_\alpha &   
        -c_\th s_\alpha+b_1 c_\alpha &       -s_\th c_\alpha+b_2 s_\alpha &         s_\th s_\alpha+b_2 c_\alpha&
        c_\alpha & -s_\alpha\cr
	c_\th s_\alpha & c_\th c_\alpha &  
        -s_\th s_\alpha & -s_\th c_\alpha &
        0 & 0\cr
	s_\th c_\alpha+b_3s_\alpha &   
        -s_\th s_\alpha+b_3c_\alpha &  c_\th c_\alpha+b_4s_\alpha &         
        -c_\th s_\alpha+b_4c_\alpha &
        s_\alpha & c_\alpha\cr
	s_\th s_\alpha &   s_\th c_\alpha & 
        c_\th s_\alpha &  c_\th c_\alpha&
        0 & 0\cr
    f_1s_\alpha & f_1c_\alpha & 
        f_2s_\alpha & f_2c_\alpha &
        c_\th c_\alpha-s_\th s_\alpha&
        -c_\th s_\alpha-s_\th c_\alpha \cr
    g_1s_\alpha & g_1c_\alpha & 
        g_2s_\alpha & g_2c_\alpha &
        s_\th c_\alpha+c_\th s_\alpha&
        -s_\th s_\alpha+c_\th c_\alpha
}\right).
\nonumber\\
\end{eqnarray*}
Then we have
$$
\sigma_1(M_\alpha)=tr(M_\alpha)=6\cos\th\cos\alpha+(b_1+b_4-2\sin\th)\sin\alpha;
$$
and hence,
\begin{eqnarray*}
\sigma_1'(M_\alpha)|_{\alpha=0}=b_1+b_4-2\sin\th.
\end{eqnarray*}
However, since $\sigma_2(M_\alpha)$ and
$\sigma_3(M_\alpha)$ are much complicated, we only give their first derivatives with respect to $\alpha$ at $\alpha=0$:
\begin{eqnarray*}
\sigma_2'(M_\alpha)|_{\alpha=0}&=&
4(b_1+b_4)\cos\th-2(b_2-b_3)\sin\th-8\cos\th\sin\th-(f_1+g_2)
\nonumber\\
&=&4(b_1+b_4)\cos\th-2(b_2-b_3)\sin\th-12\cos\th\sin\th,
\\
\sigma_3'(M_\alpha)|_{\alpha=0}&=&6(b_1+b_4)-8(3-\sin^2\th)\sin\th.
\end{eqnarray*}
Hence, by (\ref{A})-(\ref{C}),
we have
\begin{eqnarray}
A'(M_\alpha)|_{\alpha=0}&=&6[(b_2-b_3)+2\cos\th]\sin\th,
\label{A'}\\
B'(M_\alpha)|_{\alpha=0}&=&-24[(b_2-b_3)\cos\th+2+\sin^2\th]\sin\th,
\\
C'(M_\alpha)|_{\alpha=0}&=&24[(b_2-b_3)\cos\th+2+4\sin^2\th]\sin\th\cos\th. \label{C'}
\end{eqnarray}

Therefore, by (\ref{7.1})-(\ref{C'}),
we have
\begin{eqnarray*}
\Delta''(M_\alpha)|_{\alpha=0}=10368\sin^6\th.
\end{eqnarray*}

For $6^\circ$,
suppose the eigenvalues of $M_\alpha\in\Sp(6)$ are
	$\sigma(M)=\{\lambda_1,\lambda_2,\lambda_3,{1\over \lambda_1},{1\over\lambda_2},{1\over\lambda_3}\}$
	with
	$Im(\lambda_i)>0$ (or $|\lambda_i|>1$ for $\lambda_i\in\R$).
 Here $\lambda_i,i=1,2,3$ depend on $\alpha$.
 Using the same notations of (\ref{mu})-(\ref{C}),
 we have
 \begin{eqnarray*}
     D_\omega(M_\alpha)&=&\om^{-3}\det(M_\alpha-\om I)
     \nonumber\\
     &=&\om^{-3}\Pi_{i=1}^3(\om-\lambda_i)(\om-{1\over\lambda_i})
     \nonumber\\
     &=&\om^{-3}\Pi_{i=1}^3(\om^2-\mu_i\om+1)
     \nonumber\\
     &=&\Pi_{i=1}^3(\om+\bar\om-\mu_i)
     \nonumber\\
     &=&\Pi_{i=1}^3(2\cos\th-\mu_i)
     \nonumber\\
     &=&8\cos^3\th-4\Sigma_1\cos^2\th+2\Sigma_2\cos\th-\Sigma_3
     \nonumber\\
     &=&8\cos^3\th-4\sigma_1\cos^2\th+2(\sigma_2-3)\cos\th-(\sigma_3-2\sigma_1),
 \end{eqnarray*}
 where $\om=e^{\sqrt{-1}\th}$.
 Hence, we have
 \begin{eqnarray}
     \frac{\partial D_\omega(M_\alpha)}{\partial\alpha}\Big|_{\alpha=0}
  &=&-4\cos^2\th\cdot\sigma_1'(M_\alpha)|_{\alpha=0}
  +2\cos\th\cdot\sigma_2'(M_\alpha)|_{\alpha=0}
  -(\sigma_3'(M_\alpha)-2\sigma_1'(M_\alpha))|_{\alpha=0}
  \nonumber\\
  &=&-4\cos^2\th\cdot(b_1+b_4-2\sin\th)
  \nonumber\\
  &&+2\cos\th\cdot[4(b_1+b_4)\cos\th-2(b_2-b_3)\sin\th-12\cos\th\sin\th]
  \nonumber\\
  &&-[6(b_1+b_4)-8(3-\sin^2\th)\sin\th-2(b_1+b_4-2\sin\th)]
  \nonumber\\
  &=&-[4(b_1+b_4)\sin^2\th+4(b_2-b_3)\sin\th\cos\th]+12\sin\th-8\cos^2\th\sin\th
  \nonumber\\
  &=&8\sin^3\th,  \label{D_prime}
 \end{eqnarray}
 where in the last equality, we have used
 $(b_2-b_3)\cos\th+(b_1+b_4)\sin\th=1$.
 Since $D_\omega(M_\alpha)|_{\alpha=0}=D_\omega(N_3(\th,b))=0$, together with (\ref{D_prime}),
 we conclude that
 $D_\om(M_\alpha)\ne0$ for $\alpha\ne0$ small enough.
\end{proof}

\subsection{Proofs of facts}
\label{proofs.facts}

\noindent{\bf Fact 3.1.}
{\it	
    There does not exist a positive path which entirely
    in the truly hyperbolic set of $\Sp(2)$ with endpoints $D({1\over2})$ and $D(2)$.
}

\begin{proof}
    Suppose $\gamma(t),t\in[0,1]$ is a positive path with conditions $\gamma(0)=D({1\over2})$ and $\gamma(1)=D(2)$.
    By the structure of $\Sp(2)$ in Chapter 2 of
    \cite{Lon}, for each $t$, $\gamma(t)$ can be
    written in the form uniquely in the cylinder coordinate:
    \begin{eqnarray*}
    \gamma(t)&=&
        \left(\matrix{r(t)& z(t)\cr 
        z(t)&{1+z(t)^2\over r(t)}}\right)
        \left(\matrix{\cos\theta(t)&-\sin\theta(t)\cr \sin\theta(t)&\cos\theta(t)}\right)
        \nonumber\\
        &:=&M(t)R(\theta(t)).
    \end{eqnarray*}
    Here $r(t)>0$ for $t\in[0,1]$.
    Since $\gamma$ is entirely in the truly hyperbolic set,
    there is no collision;
    and hence, $r,z$ and $\theta$ are $C^1$ functions.
    The endpoints of $D({1\over2})$ and $D(2)$ give:
    \begin{eqnarray*}
    r(0)={1\over2},\quad z(0)=0,\quad \theta(0)=0;\\
    r(1)=2,\quad z(1)=0,\quad \theta(1)=2k\pi,
    \end{eqnarray*}
    for some integer $k$.

    Since
    \begin{eqnarray*}
    P&=&J^{-1}\gamma'\gamma^{-1}
      =J^{-1}[M'R+MR']R^{-1}M^{-1}
    \nonumber\\
    &=&J^{-1}[M'R+M(\theta'JR)]R^{-1}M^{-1}
    \nonumber\\
    &=&J^{-1}M'M^{-1}+\theta'M^{-T}M^{-1}
    \end{eqnarray*}
    is positive definite,
    then
    \begin{eqnarray*}
    M^TPM=&M^T[J^{-1}M'M^{-1}+\theta'M^{-T}M^{-1}]M
    =M^TJ^{-1}M'+\theta'I_2
    \end{eqnarray*}
    is also positive definite.
    However, we have
    $$
        \det(M^TJ^{-1}M')=det(M')
        =-\left({r'z-rz'\over r}\right)^2
         -\left({r'\over r}\right)^2
        \le0.
    $$
    Therefore, to guarantee the positive definite of
    $M^TPM$, we must have $\theta'(t)>0$ for any
    $t\in[0,1]$.
    Thus we must have $\gamma(1)=2k\pi$ for some positive integer $k$.
    Now, there exist a time $t_0\in(0,1)$ such that
    $\theta(t_0)=\pi$; and hence
    $$
        \gamma(t_0)=
        \left(\matrix{r(t)& z(t)\cr 
        z(t)&{1+z(t)^2\over r(t)}}\right)\cdot J_2
        =\left(\matrix{z(t)& -r(t)\cr 
        {1+z(t)^2\over r(t)}& -z(t)}\right),
    $$
    which is not truly hyperbolic, 
    a contradiction!
\end{proof}

\noindent{\bf Fact 3.2.}
{\it
There exist two truly hyperbolic matrices such that it is impossible to connect them via any positive truly hyperbolic path.
}
\begin{proof}
    We prove it for $\Sp(4)$, and
    the argument is similar for $\Sp(2n),n>2$.
    
    Let $A_0=\diag(D(2),D(3))$ and
    $\gamma_1(t)=\diag(e^{\pi Jt}D(2),e^{\pi Jt}D(3)),t\in[0,1]$.
    Then $\gamma_1(0)=A_0$
    and $\gamma_1(1)=-A_0$.
    Since $\gamma_1(1)$ and $\gamma_1(0)$
    are truly hyperbolic,
    according to Lemma \ref{Lemma:hyperbolic.positive.path},
    there exists a positive truly hyperbolic
    path $\gamma_2$ that connects $\gamma_1(1)=-A_0$ and $X^{-1}\gamma_1(0)X=X^{-1}A_0X$
    for some $X\in\Sp(4)$.
    
    We denote by $B_0=X^{-1}A_0X$.
    Then we claim that there does not exist
    any positive truly hyperbolic path that
    connects $B_0$ and $A_0$.
    If not, we assume $\gamma_3$ to be such a path.
    Gluing $\gamma_1,\gamma_2$ and $\gamma_3$, we obtain a piece-wise positive loop
    $$
        \gamma = \gamma_1*\gamma_2*\gamma_3.
    $$
    Furthermore, we can assume that $\gamma$ is $C^1$ up to the perturbations near the adjoint points, as permitted by the Smoothing Lemma.
    Moreover, by the construction,
    we have the Maslov-type index of $i_1(\gamma)=2$.
    On the other hand, 
    by Remark \ref{rk:multiplication},
    $\mu=\gamma(0)^{-1}\gamma$ is a positive loop that starts at identity.
    Then by the case $n=2$ of Theorem A,
    the Maslov-type index $i_1(\mu)\ge4$.
    However, $\gamma$ and $\mu$ are homotopic,
    thus they must have the same Maslov-type index, which raises a contradiction!
    In other words, the positive truly hyperbolic path that connects $B_0$ and $A_0$ does not exist.
\end{proof}

\medskip
\noindent {\bf Acknowledgements.}
We would like to thank Prof. Yiming Long for many helpful suggestions and remarks. We appreciate Prof. McDuff's comments on our paper and her encouragement for our future research.

\end{document}